\documentclass[12pt]{amsart}
\usepackage{amsmath}
\usepackage{amsfonts}
\usepackage{amssymb}
\usepackage{amscd}
\usepackage{graphicx}
\usepackage[abbrev,alphabetic]{amsrefs}
\RequirePackage[dvipsnames,usenames]{color}
\usepackage{soul,xcolor}
\setstcolor{red}
\usepackage{stmaryrd}
\usepackage{mathtools}
\usepackage{booktabs}
\usepackage{multirow}
\newtagform{tiny}{\tiny(}{)}
\usepackage{enumitem}

\usepackage{mathtools}
\usepackage{hyperref}
\usepackage[margin=1.25in]{geometry}
\usepackage{mathrsfs}

\usepackage{amsthm}
\usepackage{comment}
\usepackage[all,cmtip]{xy}
\usepackage{tikz-cd}
\usetikzlibrary{cd}

\tikzcdset{
  cells={font=\everymath\expandafter{\the\everymath\displaystyle}},
}

\usepackage[all]{xy}

\usepackage{cleveref}

\makeatletter
\def\@tocline#1#2#3#4#5#6#7{\relax
  \ifnum #1>\c@tocdepth 
  \else
    \par \addpenalty\@secpenalty\addvspace{#2}%
    \begingroup \hyphenpenalty\@M
    \@ifempty{#4}{%
      \@tempdima\csname r@tocindent\number#1\endcsname\relax
    }{%
      \@tempdima#4\relax
    }%
    \parindent\z@ \leftskip#3\relax \advance\leftskip\@tempdima\relax
    \rightskip\@pnumwidth plus4em \parfillskip-\@pnumwidth
    #5\leavevmode\hskip-\@tempdima
      \ifcase #1
       \or\or \hskip 1em \or \hskip 2em \else \hskip 3em \fi%
      #6\nobreak\relax
    \hfill\hbox to\@pnumwidth{\@tocpagenum{#7}}\par
    \nobreak
    \endgroup
  \fi}
\makeatother

\newcommand{\stacksproj}[1]{\cite{stacks-project}*{\href{https://stacks.math.columbia.edu/tag/#1}{Tag~{#1}}}}

\newcommand{\rup}[1]{\lceil #1 \rceil}
\newcommand{\rdown}[1]{\lfloor #1 \rfloor}

\newcommand{\Z}{\mathbb{Z}}
\newcommand{\Q}{\mathbb{Q}}
\newcommand{\R}{\mathbb{R}}
\newcommand{\F}{\mathbb{F}}

\newcommand{\cHom}{\mathcal{H}om}

\newcommand{\bN}{\mathbb{N}}

\newcommand{\bQ}{\mathbb{Q}}

\newcommand{\cA}{\mathcal{A}}

\newcommand{\cC}{\mathcal{C}}
\newcommand{\cD}{\mathcal{D}}

\newcommand{\cF}{\mathcal{F}}
\newcommand{\cG}{\mathcal{G}}
\newcommand{\cH}{\mathcal{H}}
\newcommand{\cJ}{\mathcal{J}}
\newcommand{\cI}{\mathcal{I}}

\newcommand{\cL}{\mathcal{L}}

\newcommand{\cO}{\mathcal{O}}

\newcommand{\sO}{\mathcal{O}}

\newcommand{\m}{\mathfrak{m}}
\newcommand{\n}{\mathfrak{n}}

\newcommand{\fp}{\mathfrak{p}}

\newcommand{\sK}{\mathscr{K}}

\newcommand{\wt}{\widetilde}

\DeclareMathOperator{\Mor}{Mor}

\DeclareMathOperator{\Supp}{Supp}
\DeclareMathOperator{\Spec}{Spec}

\DeclareMathOperator{\Coker}{Coker}
\DeclareMathOperator{\Hom}{Hom}

\DeclareMathOperator{\Ext}{Ext}

\DeclareMathOperator{\Exc}{Exc}

\DeclareMathOperator{\Ker}{Ker}

\DeclareMathOperator{\lct}{lct}

\DeclareMathOperator{\qfpt}{qfpt}
\DeclareMathOperator{\num}{num}

\renewcommand{\div}{{\rm div}}

\newcommand{\WDiv}{\mathrm{WDiv}}
\newcommand{\MDiv}{\mathrm{MDiv}}

\newcommand{\mydot}{{{\,\begin{picture}(1,1)(-1,-2)\circle*{2}\end{picture}\ }}}
\newcommand{\Mod}{\mathfrak{Mod}}
\newcommand{\Serre}[1]{S_{#1}}
\newcommand{\Gorenstein}{G_1}

\def\ge{\geqslant}
\def\geq{\geqslant}
\def\le{\leqslant}
\def\leq{\leqslant}
\def\phi{\varphi}
\def\epsilon{\varepsilon}
\def\tilde{\widetilde}

\def\mapsto{\longmapsto}
\def\into{\lhook\joinrel\longrightarrow}
\def\onto{\relbar\joinrel\twoheadrightarrow}
\renewcommand{\cong}{\simeq}

\theoremstyle{plain}
\newtheorem{theorem}{Theorem}[section]
\newtheorem{thm}[theorem]{Theorem}

\newtheorem{proposition}[theorem]{Proposition}
\newtheorem{prop}[theorem]{Proposition}
\newtheorem{lemma}[theorem]{Lemma}
\newtheorem{lem}[theorem]{Lemma}
\newtheorem{corollary}[theorem]{Corollary}
\newtheorem{cor}[theorem]{Corollary}

\newtheorem{claim}[theorem]{Claim}
\newtheorem*{claim*}{Claim}

\newtheorem{theoremA}{Theorem}

\theoremstyle{definition}
\newtheorem{definition}[theorem]{Definition}

\newtheorem{setting}[theorem]{Setting}
\newtheorem{example}[theorem]{Example}
\newtheorem{notation}[theorem]{Notation}

\newtheorem*{setup*}{Setup}

\theoremstyle{remark}
\newtheorem{remark}[theorem]{Remark}
\newtheorem*{ackn}{Acknowledgments}

\theoremstyle{plain}

\newenvironment{claimproof}[0]
  {%
   \paragraph{\it Proof.}%
  }
  {%
    \hfill$\blacksquare$%
  }

\numberwithin{equation}{section}

\definecolor{Mahogany}{rgb}{0.753, 0.251, 0.0}



\crefname{theorem}{Theorem}{Theorems}
\crefname{prop}{Proposition}{Propositions}
\crefname{proposition}{Proposition}{Propositions}
\crefname{lemma}{Lemma}{Lemmas}
\crefname{lem}{Lemma}{Lemmas}
\crefname{corollary}{Corollary}{Corollaries}
\crefname{cor}{Corollary}{Corollaries}
\crefname{definition}{Definition}{Definitions}
\crefname{remark}{Remark}{Remarks}

\crefname{conjecture}{Conjecture}{Conjectures}
\crefname{claim}{Claim}{Claims}
\crefname{notation}{Notation}{Notations}
\crefname{setting}{Setting}{Settings}
\crefname{example}{Example}{Examples}

\title{Quasi-$F$-splitting versus log canonicity}
\author{Kenta Sato}
\address{Department of Mathematics and Informatics, Chiba University, Chiba, 263-8522,
Japan}
\email{sato@math.s.chiba-u.ac.jp}
\author{Shunsuke Takagi}
\address{Graduate School of Mathematical Sciences, University of Tokyo, 3-8-1 Komaba,
Meguro-ku, Tokyo 153-8914, Japan}
\email{stakagi@ms.u-tokyo.ac.jp}
\author{Shou Yoshikawa}
\address{Institute of Science Tokyo, Tokyo 152-8551, Japan}
\email{yoshikawa.s.9fe9@m.isct.ac.jp}
\dedicatory{Dedicated to Professor J\'anos Koll\'ar on the occasion of his seventieth birthday.}

\begin{document}
\begin{abstract}
In this paper, we investigate the relationship between quasi-$F$-splitting and log canonicity. 
We show that if a numerically $\Q$-Gorenstein normal singularity is quasi-$F^e$-split for every $e\geq 1$, then it is numerically log canonical. 
In dimension two, we prove the converse under the condition that the Gorenstein index is not divisible by the characteristic $p$.  
We also classify two-dimensional quasi-$F$-split normal singularities.  
\end{abstract}

\maketitle
\tableofcontents

\section{Introduction}
The theory of $F$-singularities plays a central role in the study of singularities in positive characteristic and has deep connections with birational geometry. 
The notion of $F$-purity, which is equivalent to local $F$-splitting in our setting, 
was introduced by Hochster and Roberts \cite{HR76} in their study of local cohomology in positive characteristic. 
The global theory of $F$-splitting was initiated by Mehta and Ramanathan \cite{MR85} in their study of the cohomology of Schubert varieties. 
Hochster and Huneke \cite{HH89} also introduced a stronger variant of $F$-purity, called strong $F$-regularity, by requiring splittings of suitable iterations of Frobenius after perturbation by arbitrary nonzero divisors.
Hara and Watanabe \cite{hw02} proved that, under the $\Q$-Gorenstein assumption, $F$-pure singularities are log canonical and strongly $F$-regular singularities are klt. 
This result provides one of the basic links between $F$-singularities and singularities appearing in the minimal model program. 

Yobuko \cite{yobuko19} introduced the notion of quasi-$F$-splitting, a natural generalization of $F$-splitting formulated in terms of Witt vector sheaves. 
For the purposes of this introduction, we recall its local form. 
Let $R$ be an $F$-finite reduced local ring of characteristic $p>0$. The ring $W_n R$ of Witt vectors of length $n$ is  a local ring of characteristic $p^n$ and its reduced quotient $W_nR/\sqrt{0}$ is naturally isomorphic to $R$. 
The ring $W_nR$ is endowed with the Frobenius and restriction maps $F\colon W_nR \to F_*W_nR$ and $R^{n-1}\colon W_nR \to R$ (see \S\ref{Witt ring} for basic properties of the ring $W_nR$). 
\begin{definition}[\cite{yobuko19}]
We say that the local ring $R$ is \emph{quasi-$F$-split} if there exist an integer $n \ge 1$ and a $W_nR$-module homomorphism $\varphi \colon F_*W_nR \to R$ such that $\varphi \circ F=R^{n-1}$. 
\[
\xymatrix{
W_n R \ar[r]^F \ar[d]_{R^{n-1}} & F_*W_n R \ar@{.>}[dl]^{\varphi} \\
R & 
}
\]
\end{definition}
When $n=1$, this condition is the same as $F$-purity. 
Yobuko \cite{yobuko19} used this notion to study the liftability to characteristic zero of Calabi--Yau varieties of finite Artin--Mazur height. 
Quasi-$F$-splitting is considerably less sensitive to the arithmetic of the Frobenius morphism than $F$-purity, and consequently captures a much broader class of singularities.
For example, the affine cone over an elliptic curve $E$ is always quasi-$F$-split, whereas it is $F$-pure if and only if $E$ is ordinary. 

More recently, using Tanaka's Witt divisorial sheaves \cite{tanaka22}, Tanaka--Witaszek--Yobuko \cite{TWY} introduced quasi-$F^e$-splitting and quasi-$F$-regularity for log pairs, providing a framework for systematically studying singularities in birational geometry via Witt vector methods. 
Quasi-$F^e$-splitting is an iterated version of quasi-$F$-splitting. 
Quasi-$F$-regularity is a generalization of strong $F$-regularity and strengthens quasi-$F$-splitting in a way similar to the relationship between strong $F$-regularity and $F$-purity.
It was shown in \cite{KTTWYY3} that quasi-$F$-regularity implies klt singularities and that the converse holds in dimension two.

In this paper, we study the relationship between quasi-$F$-splitting and log canonicity in the setting of log pairs. 
Our first main result extends the result of Hara--Watanabe to quasi-$F$-splitting. 

\begin{theoremA}[cf.~Theorem~\ref{thm:qFs-to-lc-norm}]
Let $R$ be an $F$-finite Noetherian normal domain of characteristic $p>0$.  Let $\Delta$ be an effective $\Q$-Weil divisor on $X:=\Spec R$ such that $\lfloor \Delta \rfloor$ is reduced and $K_X+\Delta$ is numerically $\Q$-Cartier. If the pair 
$\left(R,\frac{p^e-1}{p^e} \Delta \right)$
is quasi-$F^e$-split for every integer $e\geq 1$, then $(X,\Delta)$ is numerically log canonical.  In particular, if $K_X+\Delta$ is $\Q$-Cartier and $(R,\Delta)$ is quasi-$F^e$-split for every $e \ge 1$, then $(X,\Delta)$ is log canonical.
\end{theoremA}

When $R$ is Gorenstein, \cite{KTTWYY3}*{Theorem~H} shows that $R$ is quasi-$F^e$-split for all $e \ge 1$ if and only if $R$ is quasi-$F$-split. 
Therefore, if an $F$-finite Gorenstein normal domain $R$ is quasi-$F$-split, then $\Spec R$ has log canonical singularities.  
We emphasize that, even in the Gorenstein case, quasi-$F$-splitting does not imply $F$-purity, as illustrated by the example above.

The proof of Theorem~A is inspired by the strategy of Sato--Takagi \cite{ST25}. To carry out this strategy in the quasi-$F$-setting, two additional ingredients are needed. 
First, we show that the quasi-$F^e$-splitting assumption forces the corresponding quasi-test ideals to become trivial after an arbitrarily small negative perturbation (Theorem~\ref{thm:qFs-test-norm}). 
Second, we extend the comparison between quasi-test ideals and multiplier ideals from the $\Q$-Cartier case treated in \cite{KTTWYY3} to the numerically $\Q$-Cartier case (Theorem~\ref{thm:comp-test-ideal-mult-ideal}).
Combining these results, we deduce that the corresponding multiplier ideals are also trivial after such perturbations, which implies that the pair is numerically log canonical. 

We next study the converse direction in dimension two.  
In this case, the obstruction to quasi-$F$-splitting is governed by the Cartier index of $K_X+\Delta$.
Our second main result is stated as follows.

\begin{theoremA}[\cref{thm:qFs-lc-Z_p-index} and Remark \ref{2-dimklt=>qFregular}]\label{intro:thm:qFs-lc-Z_p-index}
Let $(R,\m)$ be a two-dimensional $F$-finite normal local domain of characteristic $p>0$ with perfect residue field.  
Let $\Delta$ be an effective $\Q$-Weil divisor on $X:=\Spec R$ such that $(X,\Delta)$ is log canonical.  
If the Cartier index of $K_X+\Delta$ is not divisible by $p$, then $(X,\Delta)$ is purely quasi-$F^e$-split for every integer $e\geq 1$.
\end{theoremA}

We briefly explain the main idea of the proof of Theorem~B.
By looking at the non-klt locus of a dlt blow-up, we first reduce the question to the quasi-$F^e$-splitting of a pair $(S, \Delta_S)$, where $S$ is a one-dimensional simple normal crossing scheme and $\Delta_S$ is a (Mumford) $\Q$-divisor on $S$ (Theorem~\ref{thm:red-to-dlt}).
The index assumption then implies that $m(K_S+\Delta_S)$ is linearly trivial for some positive integer $m$ not divisible by $p$.
A key step is to show that the inverse limit over $n$ of the global sections of the cokernel of the Witt Frobenius map
\[W_n\mathcal O_{S}(K_{S}+\Delta_S) \longrightarrow F_*W_n\mathcal O_{S}(p(K_{S}+\Delta_S))\]
vanishes.
This yields the required quasi-$F^e$-splitting (Proposition~\ref{prop:van-B_0}).

Motivated by Hara’s classification of two-dimensional $F$-pure singularities \cite{hara98}, we apply the above results to classify two-dimensional quasi-$F$-split singularities.
\begin{theoremA}[\cref{thm:class-qFs}]\label{intro:thm:class-qFs}
Let $(R,\m)$ be a two-dimensional $F$-finite Noetherian normal local domain of characteristic $p>0$ with perfect residue field.  Then the following conditions are equivalent. 
\begin{enumerate}[label=$(\arabic*)$]
\item $R$ is quasi-$F$-split.
\item $R$ is quasi-$F^e$-split for all $e \ge 1$.
\item $R$ is log canonical and satisfies one of the following conditions.
\begin{enumerate}[label=\textup{(\alph*)}]
\item $R$ has log terminal singularities.
\item $R$ is not a rational singularity.
\item $p\neq 2,3$, and the dual graph is star-shaped of type $(2,3,6)$.
\item $p\neq 3$, and the dual graph is star-shaped or twisted star-shaped of type $(3,3,3)$. 
\item $p\neq 2$, and the dual graph is of type ${}_{*}\widetilde{D}_{n+3}$, twisted ${}_{*}\widetilde{D}_{n+3}$ with $n\geq 1$, or star-shaped or twisted star-shaped of type $(2,4,4)$.
\end{enumerate}
\end{enumerate}
\end{theoremA}
In particular, in characteristic $p>3$, we deduce from the above classification that quasi-$F$-splitting, quasi-$F^e$-splitting for every $e \ge 1$ and log canonicity are all equivalent for two-dimensional $F$-finite normal local domains, even when the residue field is imperfect (Theorem~\ref{thm:class-qFs imperfect}).  
We also discuss the non-normal case in Proposition \ref{prop:slc F-pure}. 

The appendix gives an alternative proof of the implication from quasi-$F^e$-splitting for every $e \geq 1$ to log canonicity in a certain setting, following an argument explained to us by Jakub Witaszek.

\begin{ackn}
The authors are grateful to Jakub Witaszek for sharing his proof with them and for permitting its inclusion in the appendix.
The first author was supported by JSPS KAKENHI Grant No.~JP24K16900, the second author by Grant Nos.~JP23K22383 and JP25H00399, and the third author by Grant No.~JP24K16889.
AI tools were used only to improve the English and readability of the manuscript. They were not used in developing or verifying the mathematical content.
\end{ackn}

\section{Preliminaries}

This section provides preliminary results needed for the rest of the paper.

\subsection{Notations}\label{subseq:Notation}
\begin{enumerate}
    \item An open subset $U$ of a scheme $X$ is said to be \emph{big} if $U$ contains all points of $X$ with codimension $\le 1$.
    \item For a chain complex $K^{\bullet}$ of objects in an abelian category, we denote by $\cH^i(K^{\bullet})$ the $i$-th cohomology of $K^{\bullet}$.
    \item An $\F_p$-scheme $X$ is said to be \emph{$F$-finite} if the Frobenius morphism $F \colon X \to X$ is a finite morphism.
    By \cite{Gabber}*{Remark 13.6}, an $F$-finite Noetherian ring $R$ is a quotient of an excellent finite dimensional regular ring.
    In particular, $R$ is excellent and admits a dualizing complex.
    \item A scheme $X$ is \emph{excellent} if it is Noetherian and every stalk $\sO_{X,x}$ is excellent. 
    \item A Noetherian scheme $X$ is said to be \emph{$\Gorenstein$} if for every point $x \in X$ with codimension $\le 1$, the stalk $\sO_{X,x}$ is Gorenstein.
    \item A \emph{prime divisor} on a Noetherian scheme $X$ is an irreducible closed subset of codimension one.
    A \emph{Weil divisor} (resp.~\emph{$\Q$-Weil divisor}) on $X$ is an element of the free $\Z$-module $\WDiv(X)$ (resp.~$\Q$-module $\WDiv_{\Q}(X)$) generated by the set of all prime divisors.
    A $\Q$-Weil divisor $D= \sum_i a_i E_i$ is \emph{effective} if $a_i \ge 0$ for all $i$.
    We say that $D$ has \emph{standard coefficients} if we have
    \[
    a_i \in \{1-\frac{1}{m} \mid m \in \Z_{\ge 1}\} \cup \{1\}
    \]
    for every $i$.
\end{enumerate}

\subsection{\texorpdfstring{$\Serre{2}$}{S2} sheaves with full support}
In this subsection, we summarize some basic properties of $\Serre{2}$ sheaves with full support (cf.~\cite{TWY}*{Subsection 2.2} for the case of an irreducible scheme).
Throughout this subsection, we assume that $X$ is a Noetherian scheme and $\cF$ is a coherent sheaf with full support (that is, $\Supp(\cF)=X$).

We say that $\cF$ is $\Serre{r}$ if one has $\mathrm{depth}_{\sO_{X,x}} (\cF_x) \ge \min \{r, \dim \sO_{X,x} \}$ for every point $x \in X$.
It follows from \cite{hartshorne_local_cohomology}*{Proposition 1.11 and Theorem 3.8} that the following conditions are equivalent:
\begin{enumerate}[label=\textup{(1-\alph*)}]
    \item $\cF$ is $\Serre{1}$.
    \item For every open subset $U \subseteq X$ and every dense open subset $V \subseteq U$, the restriction $\cF(U) \to \cF(V)$ is injective.
\end{enumerate}
Similarly, for an $\Serre{1}$-sheaf $\cF$ with full support, the following conditions are equivalent:
\begin{enumerate}[label=\textup{(2-\alph*)}]
    \item $\cF$ is $\Serre{2}$.
    \item For every open subset $U \subseteq X$ and every big open subset $V \subseteq U$, the restriction $\cF(U) \to \cF(V)$ is an isomorphism.
\end{enumerate}

\begin{remark}\label{rem:Serre condition basic}
Let $\cF$ be a coherent sheaf with full support.
\begin{enumerate}[label=\textup{(\roman*)}]
    \item $\cF$ is $\Serre{1}$ if and only if $\cF$ is torsion free (\cite{KollarTfS2}*{Definition 11}), that is, for any non-zero coherent subsheaf $\cG \subseteq \cF$, the support $\Supp(\cG)$ of $\cG$ contains some generic point of $X$.
    \item If $\cF$ is $\Serre{2}$ and $i \colon U \into X$ is the inclusion from a big open subset $U$, then we have
\[
\cF \simeq i_* (\cF|_U).
\]
\end{enumerate}
\end{remark}

\begin{definition}
An $\Serre{2}$-hull of a coherent sheaf $\cF$ with full support is an $\sO_X$-homomorphism $\phi \colon \cF \to \cF^H$ to a coherent $\sO_X$-module $\cF^H$ satisfying the following three conditions:
\begin{enumerate}[label=\textup{(\roman*)}]
    \item $\cF^H$ is $\Serre{2}$ with full support.
    \item $\phi$ is an isomorphism at every generic point of $X$.
    \item $\phi$ is surjective at every codimension one point of $X$.
\end{enumerate}
\end{definition}

\begin{remark}\label{rem:S2 hull basic}
Let $\cF$ be a coherent sheaf with full support.
\begin{enumerate}[label=\textup{(\roman*)}]
    \item An $\Serre{2}$-hull of $\cF$ is unique up to isomorphism if it exists.
    \item Let $i \colon U \into X$ be an inclusion from a big open subset $U$.
    If $\cF|_U$ is $\Serre{2}$ and $\mathcal{F}$ admits an $\Serre{2}$-hull, then the natural morphism 
    \[
    \cF \to i_* (\cF|_U) 
    \]
    is the $\Serre{2}$-hull of $\cF$.
    \item Since we assume that $\cF$ has full support, the definition of $\Serre{2}$-hull of $\cF$ is equivalent to that of the \emph{torsion free $\Serre{2}$-hull} defined in \cite{KollarTfS2}*{Definition 13}.
    In particular, if $X$ is an $\Serre{2}$-scheme and $X$ is excellent, then the $\Serre{2}$-hull $\cF^H$ exists by \cite{KollarTfS2}*{Theorem 2}.
    \item Let $\iota \colon Y \into X$ be a thickening (that is, a homeomorphic closed immersion).
    If $\cG$ is an $\Serre{2}$ coherent $\sO_Y$-module with full support, then $\iota_* \cG$ is an $\Serre{2}$ coherent $\sO_X$-module with full support.
    In particular, if $\theta \colon \cF \to \cF^H$ is the $\Serre{2}$-hull of a coherent $\sO_Y$-module $\cF$ with full support, then $\iota_* \theta$ is also the $\Serre{2}$-hull of $\iota_*\cF$. 
\end{enumerate}
\end{remark}

\begin{lem}\label{lem:Hom S2}
    Let $\cF$ and $\cG$ be coherent sheaves on a Noetherian scheme $X$.
    If $\cG$ is $\Serre{2}$ and $\Supp(\cF)=\Supp(\cG)=X$, then $\cHom_{\sO_{X}}(\cF,\cG)$ is $\Serre{2}$ with full support.
\end{lem}

\begin{proof}
    For every generic point $\eta \in X$, noting that $\sO_{X,\eta}$ is an Artin local ring, there is a non-zero homomorphism $\phi \colon \kappa(\eta) \into \cG_\eta$ from the residue field $\kappa(\eta)$ of $\sO_{X,\eta}$.
    Choosing a surjective homomorphism $\cF_{\eta} \onto \kappa(\eta)$ and composing it with $\phi$, we conclude that 
    \[
    \Hom_{\sO_{X,\eta}}(\cF_\eta, \cG_\eta) \neq 0.
    \]
    Therefore, $\cHom_{\sO_{X}}(\cF,\cG)$ has full support.
    The $\Serre{2}$-condition follows from \stacksproj{0AXQ}.
\end{proof}

\begin{lem}\label{lem:S2 hull}
    Let $\cF$ and $\cG$ be coherent sheaves with full support.
    We assume that $\cG$ is $\Serre{2}$.
    Then the following hold.
    \begin{enumerate}[label=\textup{(\arabic*)}]
        \item If $\phi \colon \cF \to \cF^H$ is an $\Serre{2}$-hull, then 
        \[
        \phi^* \colon \cHom_{\sO_{X}}(\cF^H,\cG) \to \cHom_{\sO_{X}}(\cF,\cG)
        \]
        is an isomorphism.
        \item If $\cL$ is a coherent sheaf which is invertible at every point of codimension $\le 1$, then the composite map
        \[
        \cHom_{\sO_X}(\cF,\cG) \to \cHom_{\sO_X}(\cF \otimes \cL,\cG \otimes \cL) \to \cHom_{\sO_X}(\cF \otimes \cL , (\cG \otimes \cL)^H)
        \]
        is an isomorphism.
    \end{enumerate}
\end{lem}

\begin{proof}
    By \cref{lem:Hom S2} and \cref{rem:Serre condition basic} (ii), we may replace $X$ by a big open subset.
    The assertion in (2) is obvious after shrinking $X$ to the invertible locus of $\cL$.

    For (1), after shrinking $X$, we may assume that $\phi$ is surjective.
    Noting that $\cG$ is torsion free (\cref{rem:Serre condition basic} (i)) and that the support $\Supp(\Ker(\phi))$ contains no generic point of $X$, we have 
    \[
    \cHom_{\sO_X}(\Ker(\phi), \cG) =0.
    \]
    Combining this with the exact sequence
    \[
    0 \to \Ker(\phi) \to \cF \xrightarrow{\phi} \cF^H \to 0,
    \]
    we conclude that $\phi^*$ is an isomorphism.
\end{proof}

\subsection{Weil divisors on non-normal schemes}

Throughout this subsection, we assume that $X$ is a Noetherian reduced scheme satisfying Serre's condition $\Serre{2}$.
Let $\mathscr{K}_X$ denote the sheaf of total quotients of $X$.

A \emph{Mumford divisor} (resp.~\emph{Mumford $\Q$-divisor}) on $X$ is a Weil divisor (resp.~$\Q$-Weil divisor) $B$ such that $X$ is regular at all generic points of $\Supp(B)$.
We denote by $\MDiv(X)$ (resp.~$\MDiv_{\Q}(X)$) the set of all Mumford divisors (resp.~Mumford $\Q$-divisors).
For a Mumford divisor $D$ on $X$, we define the coherent $\sO_X$-submodule 
\[
\cO_X(D) \subseteq \mathscr{K}_X
\]
of $\mathscr{K}_X$ as in \cite{ST23}*{Section~2.2}.
We note that $\cO_X(D)$ is a reflexive $\sO_X$-module, and in particular satisfies Serre's condition $\Serre{2}$ (cf.~\cite{ST23}*{Lemma~2.14 and Lemma~2.15}).
If $D$ is a Mumford $\Q$-divisor, then we define $\sO_X(D) \coloneqq \sO_X(\lfloor D \rfloor)$.

We say that $D \in \MDiv(X)$ is \emph{Cartier} if $\sO_X(D)$ is invertible.
If $D$ is a Mumford $\Q$-divisor, then $D$ is said to be \emph{$\Q$-Cartier} if $rD \in \MDiv(X)$ is Cartier for some integer $r>0$.

\begin{remark}\label{rem:WSh basic}
For any two elements $D, E \in \MDiv(X)$, we have 
\[
\sO_X(D) \cdot \sO_X(E) \subseteq \sO_X(D+E),
\]
as fractional ideals in $\mathscr{K}_X$.
If $D$ is Cartier, then this inclusion is an equality.
For any two elements $D,D' \in \MDiv(X)$, if we have $D \le D'$, then one has 
\[
\sO_X(D) \subseteq \sO_X(D').
\]
\end{remark}

Let $\varphi \in \Gamma(X,\mathscr{K}_X^*)$ be an element such that $\varphi_x \in \cO_{X,x}^*$ for every codimension-one point $x$ contained in the non-normal locus of $X$.  
We define ${\rm div}_X(\varphi) \in \MDiv(X)$ by
\[
    {\rm div}_X(\varphi) := \sum_{E_i} {\rm ord}_{E_i}(\varphi)\,E_i,
\]
where $E_i$ runs through all prime divisors on $X$ whose generic point $\eta_i$ is a regular point of $X$ and ${\rm ord}_{E_i}(\varphi)$ denotes the valuation of the stalk $\varphi_{\eta_i}$ with respect to the discrete valuation ring $\sO_{X,\eta_i}$.
We note that the sheaf $\sO_X( {\rm div}(\varphi))$ coincides with the fractional ideal $\frac{1}{\varphi}\sO_X$.

\begin{remark}\label{rem:become-effective}
With the above notation, we further assume that $X =\Spec R$ for some ring $R$.
Then for every $D \in \MDiv_{\Q}(X)$, there exists $f \in R^\circ$ with the following properties:
\begin{enumerate}
\item[\textup{(i)}] $f_\fp \in R_\fp^*$ for every height-one prime ideal $\fp$ contained in the non-normal locus of $R$, and 
\item[\textup{(ii)}] ${\rm div}(f)+D$ is effective.  
\end{enumerate}
Indeed, after replacing $D$ by a smaller one, we may assume that $D \in \MDiv(X)$ and $-D$ is effective. 
Then $R(D)$ is an ideal of $R$.  
Let $\fp_1,\ldots,\fp_r$ be the minimal primes and height-one prime ideals in the non-normal locus of $R$.  
Noting that $R(D)_{\fp_i} \simeq R_{\fp_i}$, it follows from prime avoidance that we can choose
\[
f \in R(D) \setminus \bigcup_{i=1}^r \fp_i,
\]
which satisfies the conditions.
\end{remark}

\begin{remark}\label{rem:AC divisor}
Suppose that $X$ is an $\Serre{2}$ reduced quasi-projective scheme over a Noetherian local ring $(R,\m)$ with $R/\m$ infinite.
Let $\cF$ be an \emph{AC divisor} on $X$, that is, $\cF$ is an $\Serre{2}$ coherent $\sO_X$-submodule of $\mathscr{K}_X$ which is invertible in codimension $\le 1$.
Then there exists an element $D \in \MDiv(X)$ such that $\cF \simeq \sO_X(D)$ (\cite{ST23}*{Lemma A.17}).
\end{remark}

Let $X$ be a Noetherian reduced scheme satisfying Serre's condition $\Serre{2}$.
The \textit{dualizing $\sO_X$-module} $\omega_X$ associated to a dualizing complex $\omega_X^{\bullet}$ is the coherent $\sO_X$-module defined as the first nonzero cohomology of $\omega_X^\bullet$.
A \textit{canonical Mumford divisor} of $X$ associated to $\omega_X^\bullet$ is any element $K_X \in \MDiv(X)$ such that $\sO_X(K_X) \cong \omega_X$.

\begin{proposition}\label{etale-cover-principal-index}
Let $f \colon Y \to X$ be a finite surjective flat morphism between Noetherian $\Serre{2}$ reduced schemes and $D$ be a Mumford divisor on $X$.
If $f^*D$ is a Cartier Mumford divisor, then $D$ is also Cartier.
\end{proposition}

\begin{proof}
After shrinking $X$, we may assume that $X$ is affine and that $f^*D$ is principal.
Since $f$ is finite and faithfully flat, the homomorphism $f^{\#} \colon \sO_X \to f_*\sO_Y$ splits (cf.~\cite{Hoc77}*{Proposition 5.5}).
Taking $\Serre{2}$-hull of $f^{\#} \otimes \sO_X(D)$, we obtain the homomorphism
\[
\mathcal{O}_X(D) \longrightarrow
(f_*f^*\mathcal{O}_X(D))^H \simeq (f_*\mathcal{O}_Y(f^*D))^H \simeq (f_*\sO_Y)^H
\]
which also splits.
We note that $f_*\sO_Y$ is locally free since $f$ is flat.
Therefore, its direct summand $\mathcal{O}_X(D)$ is also locally free, as desired.
\end{proof}

\subsection{Singularities in MMP}
In this subsection, we recall the notion of numerically $\Q$-Cartier $\Q$-Weil divisors and define numerically (semi) log canonical singularities.

\begin{definition}
    Let $X$ be a Noetherian normal integral scheme and $D$ be a $\Q$-Weil divisor on $X$.
    We say that $D$ is \emph{numerically $\Q$-Cartier} if there exists a proper birational morphism $f \colon Y \to X$ from a normal integral scheme $Y$ and a $\Q$-Cartier $\Q$-Weil divisor $D'$ on $Y$ such that $f_*D'=D$ and $D'$ is numerically trivial over $X$.
\end{definition}

Let $D$ be a numerically $\Q$-Cartier $\Q$-Weil divisor on a Noetherian normal integral scheme $X$.
The numerically trivial $\Q$-Cartier divisor $D'$ in the above definition is uniquely determined for every $f \colon Y \to X$ if it exists.

For a proper birational morphism $g \colon Z \to X$ from a normal scheme $Z$, there is a proper birational morphism $h \colon W \to Z$ from a normal scheme $W$ with a $\Q$-Cartier $\Q$-Weil divisor $D'$ on $W$ which is numerically trivial over $X$ and $(g \circ h)_*D'=D$.
We define the \emph{numerical pullback} of $D$ to $Z$ by 
    \[
    g^*_{\num}D \coloneqq h_*D'.
    \]
This is independent of the choice of $h$.

\begin{remark}
    Let $X$ be an excellent $2$-dimensional normal integral scheme.
    Then every $\Q$-Weil divisor $D$ on $X$ is numerically $\Q$-Cartier, and the numerical pullback $f^*_{\num}D$ by a resolution $f \colon Y \to X$ coincides with Mumford’s numerical pullback.
\end{remark}

Let $X$ be an excellent normal integral scheme with a dualizing complex $\omega_X^\bullet$. 
We fix a canonical divisor $K_X$ of $X$ associated to $\omega_X^\bullet$.
Given a proper birational morphism $\pi:Y \to X$ from a normal integral scheme $Y$, we always choose a canonical divisor $K_Y$ of $Y$ that is associated to $\pi^! \omega_X^{\bullet}$ and coincides with $K_X$ outside the exceptional locus $\mathrm{Exc}(\pi)$ of $\pi$. 

\begin{definition}\label{defn:sing in mmp}
With the above notation, suppose that $\Delta$ is an effective $\Q$-Weil divisor on $X$ such that $K_X+\Delta$ is numerically $\Q$-Cartier.
\begin{enumerate}[label=(\roman*)]
\item 
Given a proper birational morphism $f:Y \to X$ from a normal integral scheme $Y$, we define the $\Q$-Weil divisor $\Delta_Y^{\num}$ on $Y$ as   
\[
\Delta_Y^{\num} : = f^*_{\num}(K_X+\Delta)-K_Y.
\]
The \textit{discrepancy} $a_{E}(X, \Delta)$ of the pair $(X, \Delta)$ with respect to a prime divisor $E$ on $Y$ is defined as the coefficient of $E$ in $-\Delta_Y^{\num}$. 
\item The pair $(X, \Delta)$ is said to be \textit{numerically log canonical} (resp.~\textit{numerically log terminal}) (or \textit{numerically lc} (resp.~\textit{numerically klt}) for short) if $a_E(X, \Delta) \ge -1$ (resp.~$>-1$) for every prime divisor $E$ on a normal scheme $Y$ proper birational to $X$.
\item The pair $(X,\Delta)$ is \emph{log canonical} (resp.~log terminal) (or \textit{lc} (resp.~\textit{klt}) for short) if it is numerically lc (resp.~numerically klt) and $K_X+\Delta$ is $\Q$-Cartier.
\end{enumerate}
\end{definition}

A two-dimensional normal excellent scheme $X$ has a rational singularity if we have $R^1f_*\sO_Y=0$ for some (equivalently any) resolution $f \colon Y \to X$.
\begin{lem}\label{lem:num lc surface}
    Let $(X=\Spec R,\Delta)$ be a two-dimensional numerically lc pair with $R$ local.
    \begin{enumerate}[label=\textup{(\arabic*)}]
    \item If $X$ is not a rational singularity, then $\Delta=0$ and $K_X$ is Cartier.
    \item Assume that $X$ has a rational singularity.
    Let $f \colon Y \to X$ be a resolution of singularities, $D$ be a $\Q$-Weil divisor on $X$ and $r \ge 1$ be an integer.
    Then $rD$ is Cartier if and only if $rf_{\num}^*D$ is a $\Z$-Weil divisor.
    \end{enumerate}
\end{lem}

\begin{proof}
    For (1), we assume that $X$ is not a rational singularity.
    It then follows from \cite{kollar13}*{Proposition 2.28} that we have $\Delta=0$.
    By the proof of \cite{Tanaka18}*{Theorem 4.13}, we also conclude that $K_X$ is Cartier.
    The assertion in (2) follows from \cite{kollar13}*{Proposition 10.9 (2)}.
\end{proof}

\begin{definition}\label{defn:multiplier}
    Let $X$ be an excellent normal integral scheme with a dualizing complex $\omega_X^{\bullet}$.
    \begin{enumerate}[label=\textup{(\arabic*)}]
    \item For a $\Q$-Weil divisor $\Delta \ge 0$ on $X$ with $K_X+\Delta$ numerically $\Q$-Cartier, 
    the \textit{multiplier ideal sheaf} $\mathcal{J}(X,\Delta)$ associated to $(X,\Delta)$ is defined as
    \[
    \cJ(X, \Delta) \coloneqq \bigcap_{f \colon Y \to X} f_*\sO_Y(-\rdown{\Delta_Y^{\num}}) \subseteq \sO_X,
    \]
    where $f: Y \to X$ runs through all proper birational morphisms from a normal integral scheme $Y$.
    \item For a numerically $\Q$-Cartier $\Q$-Weil divisor $\Gamma$ on $X$, the \textit{multiplier submodule} $\mathcal{J}(\omega_X,\Gamma)$ associated to $(X,\Gamma)$ is defined as
    \begin{align*}
    \cJ(\omega_X, \Gamma) & \coloneqq \bigcap_{f \colon Y \to X} f_*(\omega_Y(-\rdown{f^*_{\num} \Gamma })) \\ 
    & = \bigcap_{f \colon Y \to X} f_*\cHom_{\sO_Y}(\sO_Y(\rdown{f^*_{\num} \Gamma }), \omega_Y)
    \subseteq \omega_X \otimes_{\sO_X} \mathscr{K}_X,
    \end{align*}
    where $f: Y \to X$ runs through all proper birational morphisms from a normal integral scheme $Y$.
    \end{enumerate}
    \end{definition}

\begin{remark}\label{rem:cJ basic}
Let $(X,\Delta)$ be as in \cref{defn:multiplier}.
\begin{enumerate}[label=\textup{(\roman*)}]
\item If we fix an isomorphism $\alpha \colon \sO_X(K_X) \xrightarrow{\sim} \omega_X$, then by identifying $\omega_X \otimes_{\sO_X} \mathscr{K}_X$ with $\mathscr{K}_X$, the multiplier submodules are considered as fractional ideals.
By this identification, we have
\[
\cJ(X, \Delta) = \cJ(\omega_X, K_X+\Delta).
\]
\item If $f: Y \to X$ is a log resolution of $(X,\Delta)$, then we have 
\[
\mathcal{J}(X,\Delta) = f_* \sO_Y(-\lfloor \Delta_Y^{\num} \rfloor).
\]
In particular, if such $f$ exists, then $\mathcal{J}(X,\Delta)$ is coherent.
\end{enumerate}
\end{remark}

We next consider the singularities on non-normal schemes.
Suppose that $X$ is an excellent reduced scheme satisfying Serre's condition $\Serre{2}$. 
Let $\nu : X^n \to X$ be the normalization of $X$.
Since the coherent ideal sheaf
\[
\nu^{-1} (\mathcal{H}om_X(\nu_* \sO_{X^n}, \sO_X)) \subseteq \sO_{X^n}
\]
satisfies $\Serre{2}$ condition, there is a unique effective Weil divisor $C$ on $X^n$ such that 
\[
\sO_{X^n}(-C) = \nu^{-1} (\mathcal{H}om_X(\nu_* \sO_{X^n}, \sO_X)).
\]
We call it the \emph{conductor divisor} of $\nu$.
For an element $D \in \MDiv_{\Q}(X)$, we define the pullback $\nu^*D \in \WDiv_{\Q}(X^n)$ of $D$ as the strict transform of $D$ by the birational morphism $\nu$.
 
\begin{definition}\label{defn:slc}
Let $X$ be an excellent reduced $\Serre{2}$ and $\Gorenstein$ scheme with a dualizing complex $\omega_X^{\bullet}$.
Suppose that $\Delta$ is an effective Mumford $\Q$-divisor.
\begin{enumerate}[label=\textup{(\arabic*)}]
    \item The pair $(X, \Delta)$ is said to be \textit{numerically semi log canonical} (or \textit{numerically slc} for short) if the pair $(X^n, \nu^*\Delta+C)$ is numerically lc.
    \item We further assume that $X$ admits a canonical Mumford divisor $K_X$ associated to a dualizing complex $\omega_X^{\bullet}$.
    Then the pair $(X, \Delta)$ is said to be \textit{semi log canonical} (or \textit{slc} for short) if it is numerically slc and $K_X+\Delta$ is $\Q$-Cartier.
\end{enumerate}
\end{definition}

\begin{remark}\label{remark on slc} 
If $(X,\Delta)$ is a 2-dimensional pair with $X$ normal, then $(X,\Delta)$ is numerically lc if and only if it is lc (\cref{lem:num lc surface}).
On the other hand, there exists an example of a 2-dimensional non-$\Q$-Gorenstein numerically slc scheme (see \cite{kollar13}*{Example 5.16}).
\end{remark}

\subsection{Rings of Witt vectors}\label{Witt ring}
For the convenience of the reader, we recall the notion of Witt vectors.
See \cite{KTTWYY1}*{Subsection 2.2}, \cite{illusie_de_rham_witt}*{Ch. 0, Section 1} or \cite{Serre79}*{Ch. II, Section 6} for more details.

For an $\F_p$-algebra $A$ and an integer $n > 0$, we denote by $W_nA$ the ring of Witt vectors of length $n$, that is, $W_n A$ is the set defined as
\[
W_n A \coloneqq A^{\oplus n} = \{(a_0, \dots, a_{n-1}) \mid a_i \in A\}
\]
with the suitable ring structure (see for example \cite{KTTWYY1}*{Definition 2.2}).
For an element $a \in A$, we denote by $[a]$ the Teichm\"{u}ller lift
\[
[a] \coloneqq (a,0,0, \dots, 0) \in W_nA
\]
of $a$ in $W_nA$.

\begin{example}\label{eg:multiple of Witt}
    In general, the identity map 
    \[
    \mathrm{id} : W_nA \to A^{\oplus n}
    \]
    preserves neither addition nor multiplication.
    For example, for any elements $a \in A$ and $(b_0, \dots, b_{n-1}) \in W_nA$, we have
    \[
    [a] (b_0, \dots, b_{n-1}) = (ab_0,a^pb_1, a^{p^2}b_2, \dots, a^{p^{n-1}}b_{n-1}).
    \]
\end{example}

For a ring homomorphism $f \colon A \to B$ of $\F_p$-algebras, the map 
\[
W_n f \colon W_n A \to W_n B \ ; \ (a_0,\dots, a_{n-1}) \mapsto (f(a_0), \dots, f(a_{n-1}))
\]
is a ring homomorphism.
In particular, if $A$ is a subring of $B$, then $W_nA$ is a subring of $W_nB$.
For a sheaf $\mathcal{A}$ of $\F_p$-algebras on a topological space $T$, the sheaf $W_n\mathcal{A}$ of rings on $T$ is defined by the rule
\[
\Gamma(U, W_n\mathcal{A}) \coloneqq W_n(\Gamma(U, \mathcal{A})).
\]

For a multiplicatively closed subset $S \subseteq A$, the set
\[
[S] \coloneqq \{[s] \in W_nA \mid s \in S\}
\]
of Teichm\"{u}ller lifts is multiplicatively closed and satisfies
\begin{align}\label{eq: localization and Witt}
[S]^{-1}W_n A \xrightarrow{\sim} W_n(S^{-1}A).
\end{align}
In particular, for an element $a \in A$, we have the natural isomorphism 
\begin{align*}
(W_n A)_{[a]} \xrightarrow{\sim} W_n(A_a)
\end{align*}
as rings.
This shows that for an $\F_p$-scheme $X=(X,\sO_X)$, the ringed space $(X, W_n\sO_X)$ is a scheme, which we denote by $W_nX$.
Moreover, if $X$ is an $F$-finite Noetherian separated scheme, then $W_nX$ is Noetherian and separated (\cite{LZ}*{Proposition A.1 and Proposition A.4}).

For an $\F_p$-scheme $X$, the surjective ring homomorphism 
\[
R \colon W_{n+1} {\sO_X} \longrightarrow W_n {\sO_X}  \ ; \  (a_0,a_1,\dots ,a_{n}) \mapsto (a_0,a_1,\dots, a_{n-1})
\]
defines the closed immersion $W_{n} X \hookrightarrow W_{n+1}X$.
In particular, for every integer $ 1 \le m \le n$, a $W_m \sO_X$-module is naturally considered as a $W_n \sO_X$-module.
Let $\mathcal{A}$ be a sheaf of $\sO_X$-algebras on an $\F_p$-scheme $X$. Then there exist $W_n\sO_X$-module homomorphisms
\begin{alignat*}{3}
&(\textrm{Frobenius}) &&F \colon 
W_n\mathcal{A} \longrightarrow F_*W_n\mathcal{A} & F(a_0,a_1,\dots ,a_{n-1}) &=(a_0^p,a_1^p,\dots, a_{n-1}^p) ,  \\ 
&(\textrm{Verschiebung}) \quad && V \colon 
F_*W_{n-1}\mathcal{A} \longrightarrow W_{n}\mathcal{A} \quad & V(a_0,a_1,\dots ,a_{n-2})&=(0,a_0,\dots, a_{n-2}), \\
&(\textrm{Restriction}) && R \colon 
W_{n}\mathcal{A} \longrightarrow W_{n-1}\mathcal{A} & R(a_0,a_1,\dots ,a_{n-1}) &=(a_0,a_1,\dots, a_{n-2}) ,
\end{alignat*}
where $F$ and $R$ are also the homomorphisms of $W_n\sO_X$-algebras.

By considering $(W_n\cA)_n$ as the inverse system via the restriction morphisms $R$, we define
\[
W\cA \coloneqq \varprojlim_n W_n \cA.
\]
For every open subset $U \subseteq X$, we have
\begin{align*}
\Gamma(U, W\cA) & = \varprojlim_n \Gamma(U, W_n\cA) \\
&= \varprojlim_n W_n\Gamma(U, \cA)\\
& = \{(a_0,a_1, \dots ) \mid a_i \in \Gamma(U,\cA)\}.
\end{align*}

\subsection{Witt divisorial sheaves}
Witt divisorial sheaves are defined in \cite{tanaka22}*{Section~3} on normal schemes.  
In this subsection, we generalize the construction to reduced $\Serre{2}$ schemes.  
Throughout this subsection, $X$ denotes an $F$-finite $\Serre{2}$ reduced Noetherian $\F_p$-scheme and $\mathscr{K}_X$ denotes the sheaf of total quotients of $X$.

For a Mumford $\Q$-divisor $D$ and an effective Mumford divisor $S$ (which may be zero), we define the $W_n\sO_{X}$-submodule $W_n\cI_S(D) \subseteq W_n \mathscr{K}_X$ by
\begin{align*}
\Gamma(U, W_n\cI_S(D)) &\coloneqq 
\{(\varphi_0,\ldots,\varphi_{n-1}) \mid 
\varphi_i \in \Gamma(U,\cO_X(-S+p^iD)) \text{ for every } i \}
\\
& \subseteq \Gamma(U, W_n \mathscr{K}_X)
\end{align*}
for every open subset $U \subseteq X$.
We also define the $W\sO_X$-module $W\cI_S(D)$ as 
\[
W\cI_S(D) \coloneqq \varprojlim_n W_n \cI_S(D).
\]
For every open subset $U \subseteq X$, we have
\[
\Gamma(U, W\cI_S(D)) = 
\{(\varphi_0, \varphi_1,\ldots ) \mid 
\varphi_i \in \Gamma(U,\cO_X(-S+p^iD)) \text{ for every } i \}.
\]
In the case of $S=0$, we write $W_n\cO_X(D)$ (resp.~$W\sO_X(D)$) instead of $W_n\cI_S(D)$ (resp.~$W\cI_S(D)$).

 By the same argument as in \cite{tanaka22}*{Subsection~3.1} and \cite{KTTWYY1}*{Subsection~2.6}, 
 the following hold for $n,m \in \Z_{>0}$. 
 \begin{enumerate}[label=\textup{(\arabic*)}]
 \item 
 $W_n\cI_S(D)$ is a coherent $W_n\cO_X$-submodule of $W_n\mathscr{K}_X$ (cf.~\cite{tanaka22}*{Lemma~3.5(1) and Proposition~3.8}).
\item The $W_n\sO_X$-module homomorphisms
\[
\begin{aligned}
F &\colon W_n \mathscr{K}_X \longrightarrow F_*W_n\mathscr{K}_X, \\
V &\colon F_*W_{n-1}\mathscr{K}_X \longrightarrow W_{n}\mathscr{K}_X, \\
R &\colon W_{n}\mathscr{K}_X \longrightarrow W_{n-1}\mathscr{K}_X
\end{aligned}
\]
induce the $W_n\cO_X$-module homomorphisms
\begin{alignat*}{2}
&(\textrm{Frobenius}) &&F \colon 
W_n\cI_S(D) \longrightarrow F_*W_n\cI_S(pD),  \\ 
&(\textrm{Verschiebung}) \qquad &&V \colon 
F_*W_{n-1}\cI_S(pD) \longrightarrow W_{n}\cI_S(D), \\
&(\textrm{Restriction}) &&R \colon 
W_{n}\cI_S(D) \longrightarrow W_{n-1}\cI_S(D).
\end{alignat*}
 \item 
 We have the following exact sequence (cf.~\cite{tanaka22}*{Proposition 3.7}): 
 \begin{equation} \label{eq:key-sequence-for-WnI}
 0\to F_*^nW_m\cI_S(p^nD) \xrightarrow{V^n}  
 W_{n+m}\cI_S(D) \xrightarrow{R^m} W_n\cI_S(D) \to 0.
 \end{equation}

 \end{enumerate}

\begin{lem}\label{lem:Witt divisorial S2}
    Let $X$ be an $F$-finite $\Serre{2}$ reduced Noetherian $\F_p$-scheme and $n \ge 1$ be an integer.
    For a Mumford $\Q$-divisor $D$ and an effective Mumford divisor $S$, the following hold.
    \begin{enumerate}[label=\textup{(\arabic*)}]
    \item The $W_n\sO_X$-module $W_n\cI_S(D)$ is $\Serre{2}$ with full support.
    \item If $E$ is a Cartier Mumford divisor, then $W_n\sO_X(E)$ is an invertible $W_n\sO_X$-module and we have
    \begin{align*}
    &W_n\cI_S(D+E) \simeq W_n\cI_S(D) \otimes_{W_n\sO_X} W_n\sO_X(E) \\
    &F^*W_n\sO_X(E)  \simeq W_n\sO_X(pE) \\
    &\iota^*W_n\cO_X(E)  \simeq W_m \cO_X(E),
    \end{align*}
    where $\iota \colon W_m X \into W_n X$ is the closed immersion for $n \geq m \geq 1$ and $e \geq 0$.
    \item For a Mumford divisor $E$, we have
    \begin{align*}
    \left((F^e_*\iota_*W_m\cI_S(D)) \otimes_{W_n\sO_X} W_n\sO_X(E)\right)^H \simeq F^e_*\iota_*W_m\cI_S(D+p^eE),
    \end{align*}
    where $\iota \colon W_m X \into W_n X$ is the closed immersion.
    \end{enumerate}
\end{lem}

\begin{proof}
    For (1), by the induction hypothesis on $n$, we may assume that $F_*W_{n-1}\cI_S(pD)$ is $\Serre{2}$ with full support.
    By \eqref{eq:key-sequence-for-WnI}, we have the exact sequence
    \[
    0 \to F_*W_{n-1}\cI_S(pD) \to W_{n}\cI_S(D) \to \sO_X(D-S) \to 0.
    \]
    Applying the depth lemma (\stacksproj{00LX}) to this sequence, we conclude that $W_n \cI_S(D)$ is $\Serre{2}$ with full support.
    
    For (2), after replacing $X$ by its open covering, we may assume that $E = -\div(\phi)$ for some $\phi \in \Gamma(X, \mathscr{K}_X^*)$ such that $\phi_x \in \sO_{X,x}^*$ for every codimension one point $x$ contained in the non-normal locus of $X$.
    Then the assertion follows from the equation
    \[
    W_n \sO_X( -\div(\phi)) = [\phi] W_n\sO_X,
    \]
    where $[\phi] \in \Gamma(X, W_n\mathscr{K}_X)$ is the Teichm\"{u}ller lift of $\phi$.

    For (3), we note that $F^e_*\iota_*W_m\cI_S(D+p^eE)$ is $\Serre{2}$ with full support.
    It then follows from (1) and \cref{rem:Serre condition basic} that we may replace $X$ by its big open subset and we may assume that $E$ is Cartier.
    The assertion now follows from (2).
\end{proof}

\begin{proposition}\label{etale-witt}
Let $f \colon Y \to X$ be a finite \'etale morphism of $F$-finite normal Noetherian $\F_p$-schemes.
Let $D$ be a $\Q$-Weil divisor and $S$ an effective Weil divisor on $X$.
Then, for every $e \in \Z_{\ge 0}$ and $n \in \Z_{\geq 1}$, there is a canonical isomorphism
\[
(W_nf)^*F^e_*W_n\cI_S(D) \;\xrightarrow{\ \sim\ }\; F^e_*W_n\cI_{f^*S}(f^*D),
\]
where $W_nf \colon W_n Y \to W_nX $ is the morphism induced by $f$.
\end{proposition}

\begin{proof}
The natural map $F^e_*\sK_X \to f_*F^e_*\sK_Y$ induces
\[
F^e_*\cO_X(p^eD-S) \longrightarrow f_*F^e_*\cO_Y(p^ef^*D-f^*S).
\]
Accordingly, the map $F^e_*W_n\sK_X \to f_*F^e_*W_n\sK_Y$ induces
\[
F^e_*W_n\cI_S(D) \longrightarrow f_*F^e_*W_n\cI_{f^*S}(f^*D).
\]
Since $(W_nf)^*$ is left adjoint to $f_*=(W_nf)_*$, we obtain
\[
\varphi^e_{D,n} \colon (W_nf)^*F^e_*W_n\cI_S(D) \longrightarrow F^e_*W_n\cI_{f^*S}(f^*D).
\]

We prove that $\varphi^e_{D,n}$ is an isomorphism for every $\Q$-Weil divisor $D$, $e\geq 0$ and $n \geq 1$ by induction on $n$.
Replacing $X$ by a big open subset if necessary, we may assume $\cO_X(p^eD-S)$ is invertible for all $e \geq 0$ (by \cref{lem:Witt divisorial S2}).

For $n=1$, noting that the \'{e}tale base change of the Frobenius morphism is the Frobenius morphism (\stacksproj{0EBS}), we have
\[
\varphi^e_{D,1} \colon f^*F^e_*\cO_X(D-S) 
  \;\simeq\; F^e_*f^*\cO_X(D-S)
  \;\simeq\; F^e_*\cO_Y(f^*D-f^*S).
\]

For $n \geq 2$, we first note that $W_nf$ is flat and the following commutative diagram 
\[
\begin{tikzcd}
    W_nY \arrow[r,"W_nf"] & W_nX \\
    W_mY \arrow[r,"W_mf"] \arrow[u,hook] & W_mX \arrow[u,hook]
\end{tikzcd}
\]
is Cartesian for every $m < n$ (\cite{LZ}*{Proposition A.8}).
Therefore, for every coherent $W_m\sO_X$-module $\cG$, we have the natural isomorphism
\[
(W_nf)^* \cG \cong (W_mf)^*\cG.
\]
We consider the following commutative diagram with exact rows:
 \footnotesize
\[
\begin{tikzcd}[column sep=0.55cm]
0 \arrow[r] & (W_nf)^* F^{e+1}_*W_{n-1}\cI_S(pD) \arrow[r] \arrow[d,"\sim",sloped] 
& (W_nf)^* F^e_*W_n\cI_S(D) \arrow[r] \arrow[dd,"\varphi^e_{D,n}"] 
& (W_nf)^* F^e_*\cO_X(D-S) \arrow[d,"\sim",sloped] \arrow[r] & 0 \\ 
& (W_{n-1}f)^* F^{e+1}_*W_{n-1}\cI_S(pD) \arrow[d,"\varphi^{e+1}_{pD,n-1}"] & & f^* F^e_*\cO_X(D-S) \arrow[d,"\varphi^e_{D,1}"] & \\
0 \arrow[r] & F^{e+1}_*W_{n-1}\cI_{f^*S}(pf^*D) \arrow[r]  
& F^e_*W_n\cI_{f^*S}(f^*D) \arrow[r]  
& F^e_*\cO_Y(f^*D-f^*S)  \arrow[r] & 0. \\ 
\end{tikzcd}
\]
\normalsize
By the induction hypothesis, $\varphi^{e+1}_{pD,n-1}$ and $\varphi^e_{D,1}$ are isomorphisms.
Hence $\varphi^e_{D,n}$ is also an isomorphism, as desired.
\end{proof}

\subsection{Witt dualizing modules}\label{subsec:Witt dualizing}

In this subsection, we summarize some basic properties of dualizing complexes of Witt rings.

\begin{definition}
Let $X$ be a Noetherian $\F_p$-scheme and $W_n\omega_X^{\bullet}$ be a dualizing complex of $W_nX$.
We say that the sequence $\{W_n \omega_X^{\bullet}\}_{n \ge 1}$ satisfies the condition $(*)$ if the following two properties hold:
\begin{itemize}
    \item For all integers $e \ge 0$, $n \ge m \ge 1$, there is an isomorphism
    \[
    \rho^e_{m,n} \colon W_m \omega_X^{\bullet} \xrightarrow{\sim} (\iota^e_{m,n})^! W_n \omega_X^{\bullet},
    \]
    where $\iota^e_{m,n}$ is the composite map
    \[
    \iota^e_{m,n} \colon W_mX \xrightarrow{F^e} W_mX \into W_n X.
    \]
    \item For all $e, e' \ge 0$ and $n \ge m \ge l \ge 1$, the following diagram is commutative:
    \[
    \begin{tikzcd}
        W_l\omega_X^{\bullet} \arrow[r,"\rho^{e+e'}_{l,n}"] \arrow[d, "\rho^{e}_{l,m}"] & (\iota^{e+e'}_{l,n})^!W_n \omega_X^{\bullet} \arrow[d, "\sim" sloped]\\
        (\iota^{e}_{l,m})^!W_m\omega_X^{\bullet} \arrow[r,"(\iota^{e}_{l,m})^!\rho^{e'}_{m,n}"] & (\iota^{e}_{l,m})^!((\iota^{e'}_{m,n})^!W_n \omega_X^{\bullet}) \\
    \end{tikzcd}
    \]
\end{itemize}
\end{definition}

\noindent
    {\scshape Convention:} 
    \begin{enumerate}[label=\textup{(\Roman*)}]
        \item Let $R$ be an $F$-finite Noetherian $\F_p$-algebra.
        In this paper, we always choose a dualizing complex $W_n\omega_R^{\bullet}$ on $\Spec W_n R$ so that the sequence $\{W_n \omega_R^{\bullet}\}_{n}$ satisfies the condition $(*)$ above.
        \item Moreover, for a scheme $X$ separated of finite type over $R$, we always define
        \[
        W_n\omega_X^{\bullet} \coloneqq (W_n\pi)^!W_n \omega_R^{\bullet},
        \]
        where $W_n \pi \colon W_n X \to \Spec W_n R$ is the natural morphism.
        \item When $R$ is a local ring, we also assume that $\omega_R^{\bullet} = W_1\omega_R^{\bullet}$ is normalized (see \stacksproj{0A7M} for the definition).
    \end{enumerate}

\vskip 6pt
\smallskip

\begin{remark}\label{rem:remark on convention}
Let $R$ be an $F$-finite Noetherian $\F_p$-algebra.
\begin{enumerate}[label=\textup{(\arabic*)}]
    \item By \cite{KTTWYY3}*{Theorem 9.1}, there is a sequence $\{W_n \omega_R^{\bullet}\}_n$ which satisfies the condition $(*)$.
    \item For a scheme $X$ separated of finite type over $R$, the sequence $\{W_n \omega_X^{\bullet}\}_n$ defined in Convention (II) also satisfies the condition $(*)$.
    \item When $R$ is local, noting that any shift of dualizing complex is again a dualizing complex, we are always able to attain Condition (III).
    \item We identify $W_1 X$ with $X$ and we write $\omega_X^{\bullet} \coloneqq W_1\omega_X^{\bullet}$. 
\end{enumerate}
\end{remark}

Let $X$ be a separated scheme of finite type over an $F$-finite Noetherian $\F_p$-algebra $R$ and $W_n \omega_X^{\bullet}$ be a dualizing complex as in Convention above.
Noting that $(\iota^e_{m,n})^!$ is the right adjoint of $(\iota^e_{m,n})_*=(F^e)_*$, we have the natural morphism
\[
T^e_{m,n} \colon F^e_*W_m\omega_X^{\bullet} \to W_n\omega_X^{\bullet}
\]
of elements in $D^b(W_n\sO_X)$.
It follows from \cref{rem:remark on convention} (2) and \stacksproj{0AU3} that for any integers $e \ge 0$, $n \ge m \ge 1$ and any element $K^{\bullet} \in D^{-}_{\mathrm{coh}}(W_m\sO_X)$, we have
\begin{equation}\label{eq:duality}
    F^e_*R\cHom_{W_m\sO_X}(K^{\bullet}, W_m\omega_X^{\bullet}) \simeq R\cHom_{W_n\sO_X}(F^e_*K^{\bullet}, W_n\omega_X^{\bullet}).
\end{equation}

We denote by $W_n\omega_X$ the first non-zero cohomology of $W_n \omega_X^{\bullet}$ and call it the \emph{dualizing $W_n\sO_X$-module} associated to $W_n\omega_X^{\bullet}$.
It follows from \cref{rem:dualizing module} below that the morphism $T^e_{m,n} \colon F^e_*W_m \omega_X^{\bullet} \to W_n\omega_X^{\bullet}$ induces the $W_n\sO_X$-module homomorphism
\begin{align}\label{eq:Trace map}
F^e_*W_m \omega_X \to W_n\omega_X,
\end{align}
which we also denote by the same symbol $T^e_{m,n}$.

\begin{remark}\label{rem:dualizing module}
With the above notation, let $\delta_n$ be the minimal integer such that $\cH^{\delta_n}(W_n\omega_X^{\bullet}) \neq 0$.
Then we have $\delta_n=\delta_1$ for every $n$.
In fact, by \eqref{eq:key-sequence-for-WnI}, we have the following exact sequence
\[
0 \to F_* W_{n-1}\sO_X \xrightarrow{V} W_n\sO_X \to \sO_X \to 0.
\]
Taking $R\cHom_{W_n\sO_X}(-,W_n \omega_X^{\bullet})$, combining with the isomorphism \eqref{eq:duality}, we have the exact triangle
\[
\omega_X^{\bullet} \to W_n \omega_X^{\bullet} \to F_*W_{n-1}\omega_X^{\bullet} \xrightarrow{+1}.
\]
By taking cohomology, we obtain the exact sequence
\[
0 \to \cH^{\delta_1-1}(W_n \omega_X^{\bullet}) \to F_*\cH^{\delta_1-1}(W_{n-1}\omega_X^{\bullet}) \to \cH^{\delta_1}(W_1\omega_X^{\bullet}) \to \cH^{\delta_1}(W_n \omega_X^{\bullet})
\]
and the isomorphisms
\[
\cH^i(W_n\omega_X^{\bullet}) \xrightarrow{\sim} F_*\cH^{i}(W_{n-1}\omega_X^{\bullet})
\]
for all $i <\delta_1-1$.
Therefore, by induction on $n$, we conclude that $\delta_n=\delta_1$.
\end{remark}

\begin{lem}\label{lem:adjoint and duality}
With the above notation, let $\cG$ be a coherent $W_m\sO_X$-module.
Then we have
\[
F^e_*\cHom_{W_m\sO_X}(\cG, W_m\omega_X) \simeq \cHom_{W_n\sO_X}(F^e_*\cG, W_n \omega_X).
\]
\end{lem}

\begin{proof}
By the spectral sequence for hypercohomology (\cite{Huybrechts}*{Proposition 2.66}), we have
\begin{align*}
E_2^{p,q} = \mathcal{E}xt_{W_n\sO_X}^p(F^e_*\cG, \mathcal{H}^q(W_n\omega_X^{\bullet})) \Rightarrow & \cH^{p+q}(R\cHom_{W_n\sO_X}(F^e_*\cG, W_n\omega_X^{\bullet})).
\end{align*}
Therefore, one has
\[
\cH^{\delta_n}(R\cHom_{W_n\sO_X}(F^e_*\cG, W_n\omega_X^{\bullet})) \simeq \cHom_{W_n\sO_X}(F^e_*\cG, W_n \omega_X).
\]
Similarly, we have
\[
\cH^{\delta_m}(R\cHom_{W_m\sO_X}(\cG, W_m\omega_X^{\bullet})) \simeq \cHom_{W_m\sO_X}(\cG, W_m \omega_X).
\]
Then the assertion follows from the isomorphism \eqref{eq:duality} and the equation $\delta_m=\delta_n$ (\cref{rem:dualizing module}).
\end{proof}

\begin{definition}
    With the above notation, we further assume that $X$ is reduced and $\Serre{2}$.
   Let $D$ be a Mumford divisor on $X$.
    We define the coherent $W_n\sO_X$-module $W_n\omega_X(D)$ as follows
\[
W_n\omega_X(D) \coloneqq \cHom_{W_n\sO_X}(W_n \sO_X(-D), W_n \omega_X).
\]
\end{definition}

\begin{prop}\label{prop:dualizing module basic}
Let $X$ be a connected reduced $\Serre{2}$ separated scheme of finite type over an $F$-finite Noetherian $\F_p$-algebra $R$ and $D$ be a Mumford divisor on $X$.
We further assume that the dualizing $\sO_X$-module $\omega_X = W_1\omega_X$ has full support.
Then the following hold.
    \begin{enumerate}[label=\textup{(\arabic*)}]
        \item $W_n\omega_X(D)$ is an $\Serre{2}$ coherent $W_n\sO_X$-module with full support.
        \item For all integers $e \ge 0$, $n \ge m \ge 1$ and every coherent $W_m\sO_X$-module $\cG$ with full support, we have 
        \[
        F^e_*\cHom_{W_m\sO_X}(\cG, W_m\omega_X(p^eD)) \simeq \cHom_{W_n\sO_X}(F^e_*\cG, W_n \omega_X(D)).
        \]
    \end{enumerate}
\end{prop}

\begin{proof}
    Let $\delta \colon X \to \Z$ be the dimension function associated to $\omega_X^{\bullet}$ defined in \stacksproj{0AWF}.
    Combining \stacksproj{0AWK} with the assumption $\Supp(\omega_X)=X$, the integer $\delta(\eta)$ is constant for every generic point $\eta \in X$.
    Noting that $W_nX \xrightarrow{\mathrm{id}} X \xrightarrow{\delta} \Z$ is a dimension function, it again follows from \stacksproj{0AWK} that $W_n\omega_X$ is $\Serre{2}$ with full support.
    The assertion (1) follows from \cref{lem:Hom S2}.

    For (2), since both sheaves are $\Serre{2}$ (\cref{lem:Hom S2}), after shrinking $X$, we may assume that $D$ is Cartier.
    Since $W_n\sO_X(D)$ and $W_m\sO_X(D)$ are invertible (\cref{lem:Witt divisorial S2} (2)), we have
    \begin{align*}
        F^e_*\cHom_{W_m\sO_X}(\cG,W_m\omega_X(p^eD)) & \simeq F^e_*\cHom_{W_m\sO_X}(\cG \otimes_{W_m\sO_X} W_m\sO_X(-p^eD), W_m \omega_X) \\
        & \simeq \cHom_{W_n\sO_X}(F^e_*(\cG \otimes_{W_m\sO_X} W_m\sO_X(-p^eD)), W_n \omega_X) \\
        & \simeq \cHom_{W_n\sO_X}(F^e_*\cG \otimes_{W_n\sO_X} W_n\sO_X(-D), W_n \omega_X) \\
        & \simeq \cHom_{W_n\sO_X}(F^e_*\cG, W_n \omega_X(D)),
    \end{align*}
    where the second line follows from \cref{lem:adjoint and duality} and the third line follows from \cref{lem:Witt divisorial S2} (2).
\end{proof}

In the following remark, we provide a sufficient condition ensuring that $\Supp(\omega_X) = X$, as assumed in \cref{prop:dualizing module basic} above.

\begin{remark}\label{rem:omega full support}
    Let $X$ be a connected reduced $\Serre{2}$ scheme and $\omega_X$ be a dualizing module.
    \begin{enumerate}
        \item If $X$ is irreducible, then $\omega_X$ has full support since $X$ has only one generic point (cf.~\stacksproj{0AWK}).
        \item If there is a proper morphism $\pi \colon X \to \Spec R$ to the spectrum of a Noetherian local ring $(R,\m)$, $X$ is equidimensional and there is a dualizing complex $\omega_R^{\bullet}$ of $R$ which satisfies $\omega_X^{\bullet} \simeq \pi^{!}\omega_R^{\bullet}$, then $\omega_X$ has full support by \stacksproj{0AWN}.
    \end{enumerate}
\end{remark}

\subsection{Local cohomology}
Let $(R,\m)$ be a Noetherian local ring and $f \colon X \to \Spec R$ be a morphism.
For a sheaf of rings $\cA$ on $X$, we denote by $\Mod_X(\cA)$ the category of sheaves of $\cA$-modules.

\begin{definition}
    For every sheaf $\cF \in \Mod_X(\cA)$, we define the $\Gamma(X,\cA)$-module $H^i_{\m}(\cF)$ by
    \[
    H^i_{\m}(\cF) \coloneqq H^i_{f^{-1}(\{\m\})} (X, \cF),
    \]
    where the right hand side is the local cohomology with support in $f^{-1}(\{\m\})$ (see \cite{hartshorne_local_cohomology}*{Section 1} for the definition).
\end{definition}
Noting that the higher local cohomology of flasque sheaves vanishes (\cite{hartshorne_local_cohomology}*{Proposition 1.10}), it follows from \stacksproj{015M} that one has the isomorphism 
\[
R\Gamma_{f^{-1}(\{\m\})}(X, -) = R (\Gamma_{\{\m\}} \circ f_*) \simeq R\Gamma_{\{\m\}} \circ Rf_*.
\]
Therefore, we have the isomorphism
\begin{equation}\label{eq:derived functor of composition}
H^i_{\m}(\cF) \simeq \cH^i\left( R\Gamma_{\{\m\}} \circ Rf_* (\cF) \right) 
\end{equation}
as $\Gamma(X, \cA)$-modules.

Suppose that $(R,\m)$ is an $\F_p$-algebra.
If $\cF$ is a $W_n\sO_X$-module, then $H^i_{\m}(\cF)$ is a $W_nR$-module.
In this case, identifying the underlying topological space of $\Spec W_n R$ with that of $\Spec R$, the isomorphism \eqref{eq:derived functor of composition} defines the isomorphism 
\begin{align*}
H^i_\m(\cF) \simeq  \cH^i (R\Gamma_{\m_n} R\Gamma(W_nX,\cF))
\end{align*}
as $W_nR$-modules, where $\m_n$ denotes the maximal ideal of $W_nR$.
Suppose that $E_n$ is the injective hull of $W_n R /\m_n \cong R/\m$ as a $W_nR$-module. 
For a $W_nR$-module $M$, we denote the Matlis dual by
    \[
    M^{\vee} \coloneqq \Hom_{W_n R}(M,E_n).
    \]
    We also denote by $M^{\wedge}$ the $\m_n$-adic completion of $M$.
    
\begin{lem}[\textup{\cite{TWY}*{Lemma 2.9}}]\label{lem:local duality}
    Let $(R, \m)$ be an $F$-finite Noetherian local ring of characteristic $p>0$.     Suppose that $f \colon X \to \Spec R$ is a proper morphism from a $d$-dimensional scheme $X$ and $\cF$ is a coherent $W_n\sO_X$-module.
    Then the following hold.
    \begin{enumerate}[label=\textup{(\arabic*)}]
        \item We have the isomorphism
        \[
        H^i_{\m}(\mathcal{F})^{\vee}\simeq \Ext^{-i}_{W_n\sO_X}(\mathcal{F},W_n\omega_X^{\bullet})^{\wedge} \coloneqq \cH^{-i}(R\Hom_{W_n\sO_X}(\mathcal{F},W_n\omega_X^{\bullet}))^{\wedge}
        \]
        as $W_nR$-modules.
        \item In particular, one has
        \[
        H^{d}_{\m}(\mathcal{F})^{\vee} \simeq \Hom_{W_n \sO_X}(\mathcal{F},W_n \omega_X)^{\wedge}.
        \]
    \end{enumerate}
\end{lem}

\begin{proof}
    The assertion in (1) follows from the local duality on $W_n R$ (cf.~the proof of \cite{TWY}*{Lemma 2.9} for more details).
    For (2), we first note that $W_n \omega_X \simeq \cH^{-d}(W_n \omega_{X}^{\bullet})$ (\stacksproj{0AWI}).
    Then the assertion follows from (1) and the spectral sequence 
    \[
    E_2^{p,q} = \Ext^p_{W_n \sO_X} (\cF, \cH^{q}(W_n\omega_X^{\bullet})) \Rightarrow \Ext^{p+q}_{W_n\sO_X}(\mathcal{F},W_n\omega_X^{\bullet})
    \]
    for hypercohomology (\cite{Huybrechts}*{Proposition 2.66}).
\end{proof}

\begin{remark}\label{rem:local coh is Artin}
    Let $X$ and $(R,\m)$ be as in \cref{lem:local duality}.    
    Then the local cohomology $H^i_{\m}(\cF)$ is an Artinian $W_nR$-module for every $i \ge 0$.
    In fact, since $E_2^{p,q}$ in the proof of \cref{lem:local duality} is a Noetherian $W_nR$-module, so is $\Ext^{p+q}_{W_n\sO_X}(\mathcal{F},W_n\omega_X^{\bullet})$.
    Therefore, its Matlis dual is Artinian.
\end{remark}

We next consider the compatibility of local cohomologies and inverse limits.

\begin{lem}[\textup{cf.~\cite{tanaka22}*{Lemma 4.1}}]\label{lem:Rlim}
    Let $(X,\cA)$ be a ringed space, $R$ be a ring and $M_{\bullet}=(M_n)_{n \ge 1}$ be an inverse system of sheaves of $\cA$-modules on $X$.
    Suppose that $\mathcal{B}$ is an open basis of $X$ and 
    \[
    G \colon \Mod_X(\cA) \to \Mod(R)
    \]
    is a left exact functor to the category of $R$-modules which commutes with inverse limits.
    We further assume that the following three conditions are satisfied:
    \begin{enumerate}[label=\textup{(\alph*)}]
        \item $H^i(U, M_n)=0$ for every $U \in \mathcal{B}$, $n>0$ and $i >0$.
        \item The inverse system $\{H^0(U,M_n)\}_{n \ge 1}$ satisfies the Mittag-Leffler condition for every $U \in \mathcal{B}$.
        \item The inverse system $\{R^iG(M_n)\}_{n \ge 1}$ satisfies the Mittag-Leffler condition for every $i \ge 0$.
    \end{enumerate}
    Then we have
    \[
    R^jG ( \varprojlim_n M_n) \cong \varprojlim_n R^jG(M_n).
    \]
\end{lem}

\begin{proof}
The proof is similar to that of \cite{tanaka22}*{Lemma 4.1}, but for the convenience of the reader, we provide some details.

Before the proof, we introduce some notation and briefly discuss the underlying theory.
For an abelian category $\cC$, we denote by $\cC^{\bN}$ the category of inverse systems with values in $\cC$, which is again an abelian category (\stacksproj{02MZ}).
If $N$ is an object of $\cC$, then we denote by $\tilde{N}_{\bullet}$ the inverse system $\{N\}_{n \ge 1}$ with all values $N$ and all transition maps $\mathrm{id}_N$.
For a sequence $J_1, J_2, \dots$ of objects in $\cC$, we write 
\[
J^{\oplus}_{\bullet} \coloneqq \{J_1 \oplus J_2 \oplus \cdots \oplus J_n\}_{n \ge 1},
\]
where the transition map
\[
J_1 \oplus J_2 \oplus \cdots \oplus J_n \to J_1 \oplus J_2 \oplus \cdots \oplus J_{n-1}
\]
is the projection.
We note that for every object $N$ in $\cC$, one has the bijection
\begin{align}\label{eq:inverse system}
\Mor_{\cC^{\bN}}(\tilde{N}_{\bullet}, J^{\oplus}_{\bullet}) \cong \prod_{n \ge 1} \Mor_{\cC}(N,J_n).
\end{align}
By the proof of \cite{CR12}*{Lemma 1.5.1}, for an inverse system $E_{\bullet}=(E_n)_{n \ge 1}$, the following conditions are equivalent:
\begin{itemize}
    \item $E_{\bullet}$ is an injective object in $\cC^{\bN}$.
    \item $E_n$ is injective and the transition map $E_{n+1} \to E_n$ is a splitting surjection for every $n \ge 1$.
    \item There exist injective objects $J_1, J_2 \dots $ such that $E_{\bullet} \cong J^{\oplus}_{\bullet}$.
    \end{itemize}
Combining this with \eqref{eq:inverse system}, if $\cC$ has enough injectives, then so does $\cC^{\bN}$.
Moreover, if the inverse limit $\varprojlim_n E_n$ of an injective object $\{E_n\}_{n \ge 1}$ exists, then it is an injective object of $\cC$.

Suppose that the inverse limit exists for every inverse system in $\cC$.
Then the functor $\varprojlim \colon \cC^{\bN} \to \cC$ is left exact since it is a right adjoint of the functor $N \mapsto \tilde{N}_{\bullet}$.
When we have $\cC=\mathfrak{Mod}(R)$, it follows from \cite{hartshorne77}*{II Proposition 9.1} that $R^i\varprojlim E =0$ for every $i>0$ and every inverse system $E$ in $\Mod(R)^{\bN}$ satisfying the Mittag-Leffler condition.

If $F \colon \cC \to \cD$ is a left exact functor, then so is the functor
\[
F^{\bN} \colon \cC^{\bN} \to \cD^{\bN} \ ; \ \{E_n\}_{n \ge 1} \mapsto \{F(E_n)\}_{n \ge 1}.
\]
Suppose that $\cC$ has enough injectives.
Noting that the collection of functors 
\[
(R^iF)^{\bN} \colon \cC^{\bN} \to \cD^{\bN} \ ; \ \{E_n\}_{n \ge 1} \mapsto \{R^iF(E_n)\}_{n \ge 1}
\]
is an effaceable $\delta$-functor, it follows from \cite{hartshorne77}*{III Theorem 1.3A and Corollary 1.4} that we have $(R^iF)^{\bN} \cong R^i(F^{\bN})$.

We now prove the lemma.
Consider the following commutative diagram:
\[
\begin{tikzcd}
    \Mod(\cA)^{\bN} \arrow[r,"\varprojlim"] \arrow[d,"G^{\bN}"] & \Mod(\cA) \arrow[d,"G"] \\
    \Mod(R)^{\bN} \arrow[r,"\varprojlim"] & \Mod(R)
\end{tikzcd}.
\]
Since $\varprojlim$ sends injective objects to injective objects, we have $R(G \circ \varprojlim) \cong RG \circ R\varprojlim$ (\stacksproj{015M}).
On the other hand, the assumptions (a) and (b) imply the vanishing
\[
R^i\varprojlim_n M_n =0
\]
for every $i>0$ (\cite{CR12}*{Lemma 1.5.1}).
Therefore, we have
\begin{align}\label{eq:Rlim}
R^i (G \circ \varprojlim) (M_{\bullet}) \cong R^iG (\varprojlim_n M_n).
\end{align}

Let $E_{\bullet}$ be an injective object in $\Mod(\cA)^{\bN}$.
Since every transition map of $E_{\bullet}$ is a splitting surjection, so is $G^{\bN}(E_{\bullet})$.
Since a Mittag-Leffler sequence is $\varprojlim$-acyclic in $\Mod(R)$, we have $R^i\varprojlim (G^{\bN}(E_{\bullet}))=0$ for $i>0$.
Therefore, we have $R(\varprojlim \circ G^{\bN}) \cong R\varprojlim \circ R G^{\bN}$, which induces the spectral sequence
\[
E_2^{p,q} = R^p\varprojlim R^qG^{\bN}(M_{\bullet}) \Rightarrow R^{p+q} (\varprojlim \circ G^{\bN})(M_{\bullet}).
\]
It then follows from the assumption (c) that we have 
\[
R^i(\varprojlim \circ G^{\bN}) (M_{\bullet}) \cong  \varprojlim_{n} R^iG^{\bN}(M_{\bullet}) = \varprojlim_n R^iG(M_n).
\]
Combining this with the equation \eqref{eq:Rlim}, we conclude the proof.
\end{proof}

\begin{proposition}\label{prop:lim and coh}
    Let $(R,\m)$ be an $F$-finite Noetherian local ring and $f \colon X \to \Spec R$ be a proper morphism from a reduced $\Serre{2}$-scheme $X$.
    For an effective Mumford divisor $S$ and a Mumford $\Q$-divisor $D$, we have
    \[
    H^i_{\m}(W\cI_S(D)) \cong \varprojlim_n H^i_{\m}(W_n\cI_S(D)).
    \]
\end{proposition}

\begin{proof}
    Since the functor $\Gamma_\m \colon \Mod_X(W\sO_X) \to \Mod(W(R))$ is left exact and commutes with inverse limits, it suffices to show that the assumptions (a), (b) and (c) in \cref{lem:Rlim} hold for $M_{\bullet}=\{W_n \cI_S(D)\}_{n \ge 1}$.

    The condition (a) holds for an open affine subset $U \subseteq X$ since $W_nU = (U , W_n\sO_U)$ is also affine and $W_n \cI_S(D)$ is a coherent $W_n \sO_U$-module.
    The condition (b) follows from the surjectivity of the map 
    \[
    H^0(U, R) \colon H^0(U, W_n \cI_S(D)) \to H^0(U, W_{n-1}\cI_S(D))
    \]
    for every open affine subset $U \subseteq X$.
    The condition (c) holds since $\{H^i_{\m}(W_n\cI_S(D))\}_{n \ge 1}$ is an inverse system of Artinian $W(R)$-modules (\cref{rem:local coh is Artin}) for every $i \geq 0$.
\end{proof}

\section{Pure quasi-\texorpdfstring{$F^e$}{Fe}-splitting}

\subsection{Definition}
In this subsection, we extend the notion of purely $n$-quasi-$F^e$-splitting, as introduced in \cite{TWY}*{Definition 3.34}, to the non-normal setting.
For the moment, we shall work under the following setup.

\begin{setting}\label{setting:define purely qFS}
    Let $R$ be an $F$-finite Noetherian ring of characteristic $p>0$.
    Suppose that $X$ is a reduced separated scheme of finite type over $\Spec R$ satisfying $\Serre{2}$ and $\Gorenstein$ conditions, and $\Delta$ is an effective Mumford $\Q$-divisor on $X$.
    Let $W_n\omega_R^{\bullet}$ and $W_n \omega_X^{\bullet}$ be dualizing complexes of $W_nR$ and $W_nX$ as in Convention in Subsection \ref{subsec:Witt dualizing}.
    We further assume that 
    \begin{enumerate}[label=\textup{(\roman*)}]
        \item $S \coloneqq \rdown{\Delta}$ is reduced (or $0$).
        \item $\omega_X$ has full support.
        \item There exists a canonical Mumford divisor $K_X$.
        We then fix an isomorphism $\alpha \colon \sO_X(K_X) \xrightarrow{\sim } \omega_X$.
    \end{enumerate}
\end{setting}

\begin{remark}\label{rem:on setting}
\
\begin{enumerate}[label=\textup{(\roman*)}]
\item If $X$ is a normal integral separated scheme of finite type over an $F$-finite Noetherian ring of characteristic $p>0$ and $\Delta$ is an effective divisor on $X$ with coefficients at most one, then the pair $(X,\Delta)$ satisfies all the assumptions in \cref{setting:define purely qFS}.
\item  In \cref{setting:define purely qFS}, we assume that $(R, \m)$ is local and $X$ is projective over $R$.
    \begin{enumerate}[label=\textup{(\arabic*)}]
        \item By \cref{rem:omega full support}, the assumption (ii) in \cref{setting:define purely qFS} follows if $X$ is equidimensional.
        \item By \cref{rem:AC divisor}, the assumption (iii) automatically follows from (ii) if $R/\m$ is infinite. 
\end{enumerate}
\item 
        We need not impose assumption (iii) once AC divisors are employed.
        To this end, we define $W_n\mathcal{I}_S(\mathcal{F})$ for an AC divisor $\mathcal{F}$ on $X$, and more generally, we define $W_n\mathcal{I}_S(\mathcal{F} + D)$ for an AC divisor $\mathcal{F}$ together with an element $D \in \MDiv_{\Q}(X)$  as follows:
        \begin{align*}
        \Gamma(U, W_n\cI_S(\cF+D)) &\coloneqq 
        \{(\varphi_0,\ldots,\varphi_{n-1}) \mid 
        \varphi_i \in \Gamma(U,\cO_X(p^i\mathcal{F}+\rdown{p^iD}-S))\}\\
        & \subseteq \Gamma(U, W_n \mathscr{K}_X),
        \end{align*}
        where the Mumford divisor $\rdown{p^iD}-S$ is naturally considered as an AC divisor (cf.~\cite{ST23}*{Subsection A.1}).
        Using this terminology, all results in this paper extend naturally to the setting of AC divisors.
        Observing that $\omega_X$ is an AC divisor under assumption (ii), we can consequently dispense with assumption (iii).
        For brevity, we omit the details.
\end{enumerate}
\end{remark}

With the notation as in \cref{setting:define purely qFS}, let $n \ge 1$ be an integer.
We consider the trace map 
\[
T_n \coloneqq T^0_{1,n} \colon \omega_X \to W_n\omega_X
\]
defined in \eqref{eq:Trace map}.
For a Mumford divisor $E \in \MDiv(X)$, it follows from \cref{rem:S2 hull basic}, \cref{lem:S2 hull} and \cref{lem:Witt divisorial S2} that the $\Serre{2}$-hull of $T_{n} \otimes_{W_n\sO_X} \mathrm{id}_{W_n\sO_X(E)}$ defines the $W_n\sO_X$-homomorphism
\[
T_n(E) \colon \omega_X(E) \to W_n\omega_X(E).
\]
On the other hand, the isomorphism $\alpha \colon \sO_X(K_X) \xrightarrow{\sim} \omega_X$ given in \cref{setting:define purely qFS} induces the isomorphism
\[
\alpha^* \colon \sO_X \xrightarrow{\sim} \omega_X(-K_X) = \cHom_{\sO_X}(\sO_X(K_X),\omega_X).
\]

\begin{definition}\label{defn:p qFs}
With the notation as in \cref{setting:define purely qFS}, let $e,n \geq 1$ be integers. 
We say that $(X,\Delta)$ is \emph{purely $n$-quasi-$F^e$-split} if there exists a $W_n\cO_X$-module homomorphism
\[
\varphi \colon F^e_*W_n\cI_S(p^e\Delta) \longrightarrow W_n\omega_X(-K_X)
\]
such that the following diagram commutes:
\begin{equation}\label{eq:def:p-qFs}
\begin{tikzcd}
W_n\cI_S(\Delta) \arrow[rrr,"F^e"] \arrow[d,"R^{n-1}"'] & & & F^e_*W_n\cI_S(p^e\Delta) \arrow[d,"\varphi"] \\
\cO_X \arrow[r,"\alpha^*"] & \omega_X(-K_X)  \arrow[rr,"T_n(-K_X)"]  & & W_n\omega_X(-K_X),
\end{tikzcd}
\end{equation}
This definition is independent of the choice of $\alpha$, $K_X$ and $\{W_m\omega_X^{\bullet}\}_m$.
\end{definition}

We note that if $(X,\Delta)$ is purely $n$-quasi-$F^e$-split, then it is both purely $(n+1)$-quasi-$F^e$-split and purely $n$-quasi-$F^{e-1}$-split when $e \geq 2$.

\begin{definition}
    Let $(X,\Delta)$ and $S$ be as in \cref{setting:define purely qFS}.
    \begin{enumerate}[label=\textup{(\arabic*)}]
    \item We say that $(X,\Delta)$ is \emph{purely quasi-$F^e$-split} for an integer $e \geq 1$ if there exists an integer $n \geq 1$ such that $(X,\Delta)$ is purely $n$-quasi-$F^e$-split.
    \item We say that $(X,\Delta)$ is \emph{purely quasi-$F^{\infty}$-split} if for every integer $e >0$, there exists an integer $n_e \geq 1$ such that $(X,\Delta)$ is purely $n_e$-quasi-$F^e$-split.
    \item If $S=0$ and $(X,\Delta)$ is purely $n$-quasi-$F^e$-split (resp.~purely quasi-$F^e$-split, purely quasi-$F^{\infty}$-split), then we simply say that $(X,\Delta)$ is $n$-quasi-$F^e$-split (resp.~quasi-$F^e$-split, quasi-$F^{\infty}$-split).
    \end{enumerate}
\end{definition}

\begin{remark}\label{rem:p-qFs-to-qFs}
Let $X$, $\Delta$, and $S$ be as in \cref{defn:p qFs}.  
If $(X,\Delta)$ is purely quasi-$F^e$-split, then $(X,\tfrac{p^e-1}{p^e}\Delta)$ is quasi-$F^e$-split.  

Indeed, since $(X,\Delta)$ is purely $n$-quasi-$F^e$-split for some integer $n \geq 1$, there exists a $W_n\cO_X$-module homomorphism
\[
\varphi \colon F^e_*W_n\cI_S(p^e\Delta) \longrightarrow W_n\omega_X(-K_X)
\]
fitting into the commutative diagram \eqref{eq:def:p-qFs}.  
We then obtain the following commutative diagram:
\[
\begin{tikzcd}
    W_n\cO_X\!\left(\tfrac{p^e-1}{p^e}\Delta\right) \arrow[r] \arrow[d,"\beta"] 
        & F^e_*W_n\cO_X((p^e-1)\Delta) \arrow[d,"\gamma"] \\
    W_n\cI_S(\Delta) \arrow[r] \arrow[d] 
        & F^e_*W_n\cI_S(p^e\Delta) \arrow[d,"\varphi"] \\
    \omega_X(-K_X) \arrow[r] & W_n\omega_X(-K_X),
\end{tikzcd}
\]
where $\beta$ and $\gamma$ are natural injections, which follow from the inequalities
\[
\lfloor p^m\tfrac{p^e-1}{p^e}S \rfloor \;\leq\; (p^m-1)S 
\quad\text{and}\quad 
p^m(p^e-1)S \;\leq\; (p^{e+m}-1)S
\]
for every integer $m \geq 0$.  
Therefore, $\varphi \circ \gamma$ gives the desired $W_n\cO_X$-module homomorphism.
\end{remark}

\begin{proposition}\label{proposition:no-gl-section}
With the notation as in \cref{setting:define purely qFS}, we assume that $S=\rdown{\Delta} = 0$.
If $(X,\Delta)$ is quasi-$F^e$-split for some integer $e \geq 1$, then 
\[
H^0(\cO_X((1-p^l)K_X-\rdown{p^l\Delta})) \neq 0
\]
for some integer $l \geq e$.
\end{proposition}

\begin{proof}
Let $n \geq 1$ be the minimal integer such that $(X,\Delta)$ is $n$-quasi-$F^e$-split.  
Then there exists a $W_n\cO_X$-module homomorphism 
\[
\varphi \colon F^e_*W_n\cO_X(p^e\Delta) \to W_n\omega_X(-K_X)
\]
that fits into the commutative diagram \eqref{eq:def:p-qFs}.  
Consider the following commutative diagram, where the top row is exact (cf.~\eqref{eq:key-sequence-for-WnI}):
\[
\begin{tikzcd}[column sep=small]
    0 \arrow[r] & F^{e+n-1}_*\cO_X(p^{e+n-1}\Delta) \arrow[r] \arrow[rd,"\beta"'] & 
    F^e_*W_n\cO_X(p^e\Delta) \arrow[r,"R"] \arrow[d,"\varphi"] & 
    F^e_*W_{n-1}\cO_X(p^e\Delta) \arrow[r] & 0 \\
    & & W_n\omega_X(-K_X). & &
\end{tikzcd}.
\]
We note that when $n=1$, we set $F^e_*W_{n-1}\cO_X(p^e\Delta) =0$.
By the minimality of $n$, the composite map $\beta$ is nonzero. 
Hence $\beta$ corresponds to a nonzero element of
\[
\Hom_{W_n\sO_X}\!\left(F^{e+n-1}_*\cO_X\bigl(p^{e+n-1}\Delta \bigr),\,W_n\omega_X(-K_X) \right).
\]
By \cref{prop:dualizing module basic}, this module is isomorphic to 
\begin{multline*}
F^{e+n-1}_*\Hom_{\cO_X}\!\left(\cO_X\bigl(p^{e+n-1}\Delta \bigr),\,\omega_X(-p^{e+n-1}K_X) \right) \\
\simeq\; F^{e+n-1}_*\,H^0\!\left(\cO_X\bigl((1-p^{e+n-1})K_X-\lfloor p^{e+n-1}\Delta\rfloor\bigr)\right)
\end{multline*}
Therefore, the assertion of the proposition holds for $l=e+n-1$.
\end{proof}

\begin{corollary}\label{peff-qfs}
With the notation as in \cref{setting:define purely qFS}, we further assume that $X$ is normal integral and $S=\rdown{\Delta} = 0$.
If $(X,\Delta)$ is quasi-$F^{\infty}$-split, then $-(K_X+\Delta)$ is pseudo-effective over $\Spec R$.
\end{corollary}

\begin{proof}
Take an integer $r \geq 1$ such that every coefficient of $r \Delta$ is larger than $1$. 
Then we have $ (m-r)\Delta \leq \rdown{m\Delta}$ for every integer $m \geq 1$.
By \cref{proposition:no-gl-section}, for every $e \geq 1$, there exists an integer $l_e \geq e$ such that
\[
(1-p^{l_e})(K_X+\frac{p^{l_e}-r}{p^{l_e}-1}\Delta)
\]
has a non-zero global section, thus
\[
-(K_X+\frac{p^{l_e}-r}{p^{l_e}-1}\Delta)
\]
is pseudo-effective.
Taking the limit, the divisor $-(K_X+\Delta)$ is pseudo-effective, as desired.
\end{proof}

\subsection{Equivalent conditions}
In this subsection, we present several equivalent characterizations of purely $n$-quasi-$F^e$-splitting (\cref{prop:p-qFS equivalent def}).
We begin with the notation that will be used repeatedly throughout this paper.

\begin{notation}\label{notation:Q and K}
    With the notation as in \cref{setting:define purely qFS}, let $e,n \ge 1$ be integers and $D$ be a Mumford $\Q$-divisor.
    \begin{enumerate}[label=\textup{(\roman*)}]
        \item We define the $W_n\cO_X$-modules $Q^{S,e}_{X,D,n}$ and the $W_n\cO_X$-module homomorphism $\Phi^{S,e}_{X,D,n}$ by the following pushout diagram:
\begin{equation}\label{e-def-Q^{e,S}_D}
\begin{tikzcd}
W_n\cI_S(D) \arrow[r,"F^e"] \arrow[d,"R^{n-1}"'] & F_*^e(W_n\cI_S(p^eD)) \arrow[d] \\
\cO_X(D-S) \arrow[r,"\Phi^{S,e}_{X,D,n}"] & Q^{S,e}_{X,D,n}. 
\end{tikzcd}
\end{equation}
\item We define the $W_n R$-homomorphism  
\[
(\Phi^{S,e}_{X,K_X+\Delta,n})^* \colon \Hom_{W_n\cO_X}\!\left(Q^{S,e}_{X,K_X+\Delta,n},\, W_n\omega_X\right) \to H^0(X,\sO_X)
\]
as the composite of the dual 
\[
\Hom_{W_n\cO_X}\!\left(Q^{S,e}_{X,K_X+\Delta,n},\, W_n\omega_X \right) \to \Hom_{W_n\cO_X}\!\left(\sO_X(K_X),\, W_n\omega_X \right)
\]
of $\Phi^{S,e}_{X,K_X+\Delta,n}$ and the isomorphism 
\begin{align*}
\Hom_{W_n\cO_X}\!\left(\sO_X(K_X),\, W_n\omega_X \right)
&\simeq \Hom_{\cO_X}(\cO_X(K_X),\omega_X) \\
& \simeq H^0(X,\sO_X),
\end{align*}
where the first line follows from \cref{lem:adjoint and duality}.
        \item When the base ring $R$ is local, we define the $W_nR$-module $K^{S,e}_{X,D,n}$ by
        \[
        K^{S,e}_{X,D,n} := \Ker\!\left(H^d_\m(W_n\cI_S(D)) \xrightarrow{F^e} F^e_*H^d_\m(W_n\cI_S(p^eD))\right),
        \]
        where $\m$ is the maximal ideal of $R$ and $d$ is the dimension of $X$.
        When there is no confusion, we simply denote it by $K^{e}_{D,n}$.
        For an integer $m \le n$, the homomorphism 
        \[
        H^d_{\m}(R^{n-m}) \colon H^d_{\m} (W_n\cI_S(D)) \to H^d_{\m}(W_m\cI_S(D))
        \]
        induces $H^d_{\m}(R^{n-m}) \colon K^{S,e}_{X,D,n} \to K^{S,e}_{X,D,m}$ which fits into the following commutative diagram:
        \[
        \begin{tikzcd}
            0 \arrow[r] & K^{e}_{D,n} \arrow[r] \arrow[d,"H^d_{\m}(R^{n-m})"] & H^d_{\m}(W_n\cI_S(D)) \arrow[r,"F^e"] \arrow[d,"H^d_{\m}(R^{n-m})"] & F^e_*H^d_{\m}(W_n\cI_S(p^eD)) \arrow[d,"H^d_{\m}(R^{n-m})"]\\
            0 \arrow[r] & K^{e}_{D,m} \arrow[r] & H^d_{\m}(W_m\cI_S(D)) \arrow[r,"F^e"]  & F^e_*H^d_{\m}(W_m\cI_S(p^eD))
        \end{tikzcd}
        \]
    \end{enumerate}
\end{notation}

\begin{prop}\label{prop:p-qFS equivalent def}
    With the notation as in \cref{setting:define purely qFS} and \cref{notation:Q and K}, let $e, n \ge 1$ be any integers.
    Then the following are equivalent.
    \begin{enumerate}[label=\textup{(\alph*)}]
        \item $(X,\Delta)$ is purely $n$-quasi-$F^e$-split.
        \item The $W_nR$-homomorphism 
        \[
(\Phi^{S,e}_{X,K_X+\Delta,n})^* \colon \Hom_{W_n\cO_X}\!\left(Q^{S,e}_{X,K_X+\Delta,n},\, W_n\omega_X\right) \to H^0(X,\sO_X)
\]
defined in \cref{notation:Q and K} (ii) is surjective.
    \end{enumerate}
    If moreover the base ring $(R,\m)$ is local and $X$ is proper over $R$, then the above are equivalent to the following conditions.
    \begin{enumerate}[label=\textup{(\alph*)}]
    \setcounter{enumi}{2}
        \item The $W_nR$-homomorphism 
        \[
        H^d_{\m}(\Phi^{S,e}_{X,K_X+\Delta,n})\colon H^d_{\m}(\sO_X(K_X)) \to H^d_{\m}(Q^{S,e}_{X,K_X+\Delta,n})
        \]
        is injective.
        \item The $W_nR$-homomorphism
        \[
        H^d_{\m}(R^{n-1}) \colon K^{e}_{K_X+\Delta, n} \to K^{e}_{K_X+\Delta,1}
        \] defined in \cref{notation:Q and K} (iii) is a zero map.
    \end{enumerate}
\end{prop}

\begin{proof}
        After tensoring the diagram \eqref{eq:def:p-qFs} with $W_n\sO_X(E)$ and taking the $\Serre{2}$-hulls, it follows from \cref{lem:Witt divisorial S2} (3) that the condition (a) is equivalent to the following condition:
    \begin{enumerate}
        \item[\textup{(a')}] For some (equivalently, for all) $E \in \MDiv(X)$, there exists a $W_n\cO_X$-module homomorphism
        \[
        \varphi \colon F^e_*W_n\cI_S(p^e (\Delta+E)) \longrightarrow W_n\omega_X(-K_X+E)
        \]
        such that the following diagram commutes:
        \[
        \begin{tikzcd}
        W_n\cI_S(\Delta+E) \arrow[rrr,"F^e"] \arrow[d,"R^{n-1}"'] & & & F^e_*W_n\cI_S(p^e(\Delta+E)) \arrow[d,"\varphi"] \\
        \cO_X(E) \arrow[r,"\alpha^*(E)"] & \omega_X(-K_X+E)  \arrow[rr,"T_n(-K_X+E)"]  & & W_n\omega_X(-K_X+E),
        \end{tikzcd}
        \]
    \end{enumerate} 
    By the universality of pushout, the condition (a') for $E=K_X$ is equivalent to the following condition:
    \begin{enumerate}
        \item[\textup{(b')}] the image of $(\Phi^{S,e}_{X,K_X+\Delta,n})^*$ contains $1$.
    \end{enumerate}
Since $H^0(X,\sO_X)$ is generated by $1$ as the $H^0(X,W_n \sO_X) =W_nH^0(X,\sO_X)$-module, this is equivalent to (b).

From now on, we assume that $R$ is local and $X$ is proper over $R$.
The equivalence (b) $\Leftrightarrow$ (c) follows from the local duality (\cref{lem:local duality}).

For the equivalence (c) $\Leftrightarrow$ (d), we consider the following commutative diagram:
{\small
\[
\begin{tikzcd}
    0 \arrow[r] & W_n\cI_S(K_X+\Delta) \arrow[r,"F^e"] \arrow[d,"R^{n-1}"] & F^e_*W_n \cI_S(p^e ( K_X+\Delta)) \arrow[r] \arrow[d] & \Coker(F^e) \arrow[r] \arrow[d,"\sim",sloped] & 0 \\
    0 \arrow[r] & \sO_X(K_X) \arrow[r,"\Phi^{S,e}_{X,K_X+\Delta,n}"] & Q^{S,e}_{X,K_X+\Delta,n} \arrow[r] & \Coker(\Phi^{S,e}_{X,K_X+\Delta,n}) \arrow[r] & 0.
\end{tikzcd}
\]
}
Taking local cohomology, we obtain 
\[
\begin{tikzcd}
H^{d-1}_{\m}(\Coker(F^e)) \arrow[r,twoheadrightarrow] \arrow[d,"\sim",sloped]& K^{e}_{K_X+\Delta,n} \arrow[r,hookrightarrow] \arrow[d] & H^d_{\m}(W_n\cI_S(K_X+\Delta)) \arrow[d,"H^d_{\m}(R^{n-1})"]\\
H^{d-1}_{\m}(\Coker(\Phi^{S,e}_{X,K_X+\Delta,n})) \arrow[r,twoheadrightarrow] & \Ker(H^d_{\m}(\Phi^{S,e}_{X,K_X+\Delta,n})) \arrow[r,hookrightarrow] & H^d_{\m}(\sO_X(K_X)),
\end{tikzcd}
\]
which implies the equation 
\[
\Ker(H^d_{\m}(\Phi^{S,e}_{X,K_X+\Delta,n})) = H^d_{\m}(R^{n-1}) (K^e_{K_X+\Delta,n}).
\]
This completes the proof of the proposition.

\end{proof}

In the following lemma, we also prove a variant of the condition (d) in \cref{prop:p-qFS equivalent def}.

\begin{lem}\label{lem:pqFs to zero map}
With the notation as in \cref{setting:define purely qFS} and \cref{notation:Q and K}, let $e,n \ge 1$ be integers.
We further assume that the base ring $(R,\m)$ is local and $X \to \Spec{R}$ is proper.
If $(X,\Delta)$ is purely $n$-quasi-$F^e$-split, then for every element $E \in \MDiv(X)$, the $W_nR$-homomorphism
\[
H^{d}_{\m}(R^{n-1}) \colon K^{S, e}_{X, \Delta+E,n} \to K^{S, e}_{X, \Delta+E,1}
\]
is a zero map.
\end{lem}

\begin{proof}
Since $(X,\Delta)$ is purely $n$-quasi-$F^e$-split, there is a homomorphism
\[
\varphi \colon F^e_*W_n\cI_S(p^e\Delta) \to W_n\omega_X(-K_X)
\]
such that the following diagram commutes:
\[
\begin{tikzcd}
    W_n\cI_S(\Delta) \arrow[rr,"F^e"] \arrow[d,"R^{n-1}"] 
        && F^e_*W_n\cI_S(p^e\Delta) \arrow[d] \\
    \sO_X \arrow[r,"\sim"] & \omega_X(-K_X) \arrow[r] & W_n\omega_X(-K_X)
\end{tikzcd} 
\]
Tensoring with $W_n\cO_X(E)$ and taking the $\Serre{2}$ hulls, it follows from \cref{lem:S2 hull} (1) that we obtain the following commutative diagram.
\[
\begin{tikzcd}
    W_n\cI_S(\Delta+E) \arrow[rr,"F^e"] \arrow[d,"R^{n-1}"] 
        && F^e_*W_n\cI_S(p^{e}(\Delta+E)) \arrow[d] \\
    \sO_X(E) \arrow[r,"\sim"] & \omega_X(-K_X+E) \arrow[r] & W_n\omega_X(-K_X+E)
\end{tikzcd} .
\]
Applying $H^d_\m(-)$, we get
\[
\begin{tikzcd}
    H^d_\m(W_n\cI_S(\Delta+E)) \arrow[rr,"F^e"] \arrow[d,"H^{d}_{\m}(R^{n-1})"] 
        && F^e_*H^d_\m(W_n\cI_S(p^e(\Delta+E))) \arrow[d] \\
    H^d_{\m}(\sO_X(E)) \arrow[r,"\sim"] & H^d_\m(\omega_X(-K_X+E)) \arrow[r,"(\star_1)"] & H^d_\m(W_n\omega_X(-K_X+E)).
\end{tikzcd} 
\]
By the local duality (\cref{lem:local duality}) and \cref{lem:Hom S2}, the Matlis dual of $(\star_1)$ as $W_nR$-module is the completion of 
\[
H^0(X, R^{n-1}) \colon H^0\left(X,W_n\cO_X(K_X-E)\right)\to H^0\left(X, \cO_X(K_X-E)\right).
\]
Since $H^0(X,R^{n-1})$ is surjective, the dual $(\star_1)$ is injective.
Thus we have 
\[
H^d_{\m}(R^{n-1}) (K^{S,e}_{X,\Delta+E,n})=0,
\]
as desired.
\end{proof}

\begin{lemma}\label{completion}
With the notation as in \cref{setting:define purely qFS}, we further assume that $X=\Spec R$ and $(R,\m)$ is local.
Then $(\Spec{R},\Delta)$ is purely $n$-quasi-$F^e$-split if and only if so is the pair $(\Spec{\widehat{R}},\iota^*\Delta)$, where $\iota \colon \Spec{\widehat{R}} \to \Spec{R}$ is the completion with respect to $\m$.
\end{lemma}

\begin{proof}
Let $D$ be a Mumford $\Q$-divisor on $X=\Spec R$.
We first note that 
\[
F^l_*R(D) \otimes_R \widehat{R} \simeq F^l_*(R(D) \otimes_R \widehat{R}) \simeq F^l_*\widehat{R}(\iota^*D)
\]
for every integer $l \ge 0$.
On the other hand, since $W_m\widehat{R}$ is the completion of $W_mR$ (\cite{KTY25}*{Proposition 2.2}), we have
\[
W_m\widehat{R} \cong W_m R \otimes_{W_nR} W_n\widehat{R}
\]
for all integers $1 \le m \le n$.
Therefore, as in the proof of \cref{etale-witt},
\[
F^l_*W_m\mathcal{I}_S(D)\otimes_{W_mR} W_m\widehat{R}
\simeq
F^l_*W_m\mathcal{I}_{\iota^*S}(\iota^*D)
\]
for every $l\ge 0$ and $m\ge 1$.
Combining this with the exact sequence
\[
0 \to F_*W_{n-1}\mathcal{I}_S(pD)
\xrightarrow{VF^e}
F^e_*W_n\mathcal{I}_S(p^eD)
\to Q^{S,e}_{R,D,n} \to 0,
\]
we obtain
\[
Q^{S,e}_{R,D,n}\otimes_{W_nR} W_n\widehat{R}
\simeq
Q^{\iota^*S,e}_{\widehat{R},\iota^*D,n}.
\]
Finally, it also follows from $W_m\widehat{R} \cong \widehat{W_mR}$ (\cite{KTY25}*{Proposition~2.2}) that for every $W_nR$-module $M$, one has 
\[
H^d_{\widehat{\m}}(M \otimes_{W_nR} W_n \widehat{R}) \cong H^d_{\m}(M ) \otimes_{W_nR} W_n\widehat{R}.
\]
The assertion now follows from \cref{prop:p-qFS equivalent def}~(c).
\end{proof}

\begin{remark}[\textup{cf.~\cite{TWY}*{Introduction}}]\label{rem:compare with original definition}
    With the notation as in \cref{setting:define purely qFS} and \cref{notation:Q and K}, we consider the case of $e=1$.
    In this case, as in the proof of \cite{TWY}*{Proposition 3.12}, we see that $Q^{S,1}_{X,\Delta,n}$ is an $\sO_X$-module.
    It then follows from \cref{lem:adjoint and duality} that one has
    \begin{align*}
    \Hom_{W_n\sO_X}(Q^{S,1}_{X,\Delta,n}, W_n\omega_X(-K_X)) & \cong \Hom_{\sO_X}(Q^{S,1}_{X,\Delta,n}, \omega_X(-K_X)) \\
    & \cong \Hom_{\sO_X}(Q^{S,1}_{X,\Delta,n}, \sO_X).
    \end{align*}
    This implies that $(X,\Delta)$ is purely $n$-quasi-$F^1$-split if and only if we have a $W_n\sO_X$-homomorphism $\phi$ which fits into the following commutative diagram:
    \[
    \begin{tikzcd}
        W_n\cI_S(\Delta) \arrow[r,"F"] \arrow[d,"R^{n-1}"'] & F_*W_n\cI_S(p\Delta) \arrow[ld,"\phi"] \\
        \sO_X & 
    \end{tikzcd}
    \]
\end{remark}

\subsection{Pure quasi-\texorpdfstring{$F^e$}{Fe}-splitting criterion via the injectivity of Frobenius}
In this subsection, we establish a sufficient condition for purely quasi-$F^e$-splitting in terms of the injectivity of the Frobenius action on local cohomology (\cref{prop:infty-qFs}).
We further examine the converse direction (\cref{thm:qFs-to-qF^es} and \cref{cor:qFs-to-qF^eS}).
Henceforth, we will adopt the following notation.
\begin{setting}\label{setting:global setting}
Let $X,\Delta, S$ and $R$ be as in \cref{setting:define purely qFS}.
We further assume that $(R,\m)$ is local and $X$ is proper over $R$ of dimension $d$.
\end{setting}

With the notation as in \cref{setting:global setting} and \cref{notation:Q and K}, let $e \geq 1$ be an integer and $D$ be a Mumford $\Q$-divisor.
We set 
\[
K^{S,e}_{X,D, \infty} \coloneqq \Ker\left(H^d_{\m} \left(W\cI_S(D) \right) \xrightarrow{F^e} F^e_* H^d_{\m}\left( W\cI_S(p^eD) \right)\right).
\]
When there is no confusion, we simply denote it by $K^e_{D,\infty}$.
Since the inverse limit is left exact, it follows from \cref{prop:lim and coh} that we have
\[
K^{e}_{D,\infty} \simeq \varprojlim_{n} K^{e}_{D,n}.
\]
\begin{proposition}\label{prop:infty-qFs}
With the above notation, we set $D \coloneqq K_X+\Delta$.
\begin{enumerate}[label=\textup{(\arabic*)}]
\item $(X,\Delta)$ is purely quasi-$F^e$-split if and only if the projection 
\[
K^{e}_{D,\infty} \to K^{e}_{D,1}
\]
is a zero map.
\item In particular, if the map
\[
H^d_{\m} \left(W\cI_S(D) \right) \xrightarrow{F^e} F^e_* H^d_{\m}\left( W\cI_S(p^eD) \right)
\]
is injective, then $(X,\Delta)$ is purely quasi-$F^e$-split.
\end{enumerate}
\end{proposition}

\begin{proof}
Since $K^e_{D,n}$ is an Artinian $W_nR$-module (\cref{rem:local coh is Artin}) for every $n$, the inverse system $\{K^{e}_{D,n}\}_{n \ge 1}$ satisfies the Mittag-Leffler condition.
Therefore, the projection $\varprojlim_{n} K^{e}_{D,n} \to K^{e}_{D,1}$ is a zero map if and only if so is the map $K^{e}_{D,n} \to K^{e}_{D,1}$ for some $n \ge 1$.
Then the assertion of (1) follows from \cref{prop:p-qFS equivalent def}.
\end{proof}

We next give a sufficient condition for the injectivity of the Frobenius action on the local cohomology.

\begin{theorem}\label{thm:qFs-to-qF^es}
With the notation as in \cref{setting:global setting}, we set $D \coloneqq K_X+\Delta$.
We further assume the following conditions are satisfied:
\begin{enumerate}[label=\textup{(\roman*)}]
\item $H^{d-1}_{\m}(\sO_X(p^lD-S)) =0$ for every integer $l \ge 0$, and
\item $(X,S+\{p^l\Delta\})$ is purely quasi-$F$-split for every $l \geq 0$.
\end{enumerate}
Then for every $e \ge 1$, the map
\[
H^d_{\m} \left(W\cI_S(D) \right) \xrightarrow{F^e} F^e_* H^d_{\m}\left( W\cI_S(p^eD) \right)
\]
is injective, and in particular, $(X,\Delta)$ is purely quasi-$F^{\infty}$-split (\cref{prop:infty-qFs}).
\end{theorem}

\begin{proof}
For an integer $m \ge 1$, we denote by
\[
\pi^l_{m} \colon H^d_{\m}( W\cI_S(p^lD))  \to H^d_{\m}(W_m\cI_S(p^lD))
\]
the natural projection.

We show, by induction on $m \ge 1$, that for every $l \geq 0$, we have $\pi^l_{m}(K^1_{p^lD,\infty})=0$.
When $m=1$, noting that we have
\[
p^l D = \{p^lD\} + \rdown{p^lD} = (S+\{p^lD\}) + (\rdown{p^lD} -S),
\]
it follows from \cref{lem:pqFs to zero map} that we have $\pi^l_{1}(K^1_{p^lD,\infty})=0$.

We next assume $m \geq 2$.
By the assumption $H^{d-1}_\m(\cO_X(p^lD-S))=0$ and the left exactness of the inverse limit, we have the following exact sequence
\small
\[
\begin{tikzcd}
0 \arrow[r] 
    & \varprojlim_n F_* H^d_\m(W_n\cI_S(p^{l+1}D)) \arrow[r,"V"] 
    & \varprojlim_n H^d_\m(W_n\cI_S(p^lD)) \arrow[r,"\pi^l_{1}"] 
    & H^d_\m(\cO_X(p^lD-S)).
\end{tikzcd}
\]
\normalsize
Combining this with \cref{prop:lim and coh}, we obtain the exact sequence
\[
\begin{tikzcd}
0 \arrow[r] 
    &  F_* H^d_\m(W\cI_S(p^{l+1}D)) \arrow[r,"V"] 
    & H^d_\m(W\cI_S(p^lD)) \arrow[r,"\pi^l_{1}"] 
    & H^d_\m(\cO_X(p^lD-S)).
\end{tikzcd}
\]
Take an element $\alpha \in K^1_{p^lD,\infty}$.
By the case $m=1$, we have $\pi^l_{1}(\alpha)=0$, hence there exists $\beta \in H^d_\m(W\cI_S(p^{l+1}D))$ with $V(F_*\beta)=\alpha$.
Combining the injectivity of $F_*V$ with the following commutative diagram
\[
\begin{tikzcd}
 F_* H^d_\m(W\cI_S(p^{l+1}D)) \arrow[r,"V"] \arrow[d,"F"] 
    & H^d_\m(W\cI_S(p^lD)) \arrow[d,"F"] \\
    F^2_* H^d_\m(W\cI_S(p^{l+2}D)) \arrow[r,"F_*V"]
    & F_* H^d_\m(W\cI_S(p^{l+1}D)),
\end{tikzcd}
\]
we have $\beta \in K^1_{p^{l+1}D,\infty }$.
By the induction hypothesis, we have $\pi^{l+1}_{m-1}(\beta)=0$.  
It then follows from the following commutative diagram
\[
\begin{tikzcd}
     F_* H^d_\m(W\cI_S(p^{l+1}D)) \arrow[r,"V"] \arrow[d,"F_*\pi^{l+1}_{m-1}"] 
    & H^d_\m(W\cI_S(p^lD)) \arrow[d,"\pi^l_{m}"] \\
     F_*H^d_\m(W_{m-1}\cI_S(p^{l+1}D)) \arrow[r,"V"] 
    & H^d_\m(W_m\cI_S(p^lD))
\end{tikzcd}
\]
that we have 
\[
\pi^{l}_{m}(\alpha)=V(F_*\pi^{l+1}_{m-1}(\beta))=0,
\]
which shows the equality $\pi^{l}_{m}(K^1_{p^lD, \infty})=0$, as desired.

Therefore, we have $K^1_{p^lD, \infty} = \varprojlim_m K^1_{p^lD,m} =0$ and so the Frobenius map
\[
F \colon H^d_\m(W\cI_S(p^lD)) \to F_*H^d_\m(W\cI_S(p^{l+1}D))
\]
is injective for all $l \ge 0$.
Consequently, the composite map
\[
F^e \colon H^d_\m(W\cI_S(D)) \to F_*H^d_\m(W\cI_S(pD)) \to \cdots \to F^e_*H^d_\m(W\cI_S(p^eD))
\]
is injective for every $e \geq 1$.
\end{proof}

\begin{cor}\label{cor:qFs-to-qF^eS}
With the notation as in \cref{setting:define purely qFS}, we further assume that the base ring $(R,\m)$ is local and $X = \Spec R$. 
Set $D:=K_X+\Delta$ and $d:=\dim R$.
Moreover, suppose that 
\begin{enumerate}[label=\textup{(\roman*)}]
\item $\sO_X(p^lD-S)$ is Cohen-Macaulay for every integer $l \ge 0$.
\item $\Delta$ has standard coefficients.
\end{enumerate}
Then the following conditions are equivalent:
\begin{enumerate}[label=\textup{(\arabic*)}]
\item $(X,\Delta)$ is purely quasi-$F$-split.
\item $(X,\Delta)$ is purely quasi-$F^{\infty}$-split.
\item The map
\[
H^d_{\m} \left(W\cI_S(D) \right) \xrightarrow{F} F_* H^d_{\m}\left( W\cI_S(pD) \right)
\]
is injective.
\item The map
\[
H^d_{\m} \left(W\cI_S(D) \right) \xrightarrow{F^e} F^e_* H^d_{\m}\left( W\cI_S(p^eD) \right)
\]
is injective for every $e \ge 1$.
\end{enumerate}
\end{cor}
\begin{proof}
    The implications (4) $\Rightarrow$ (3) and (2) $\Rightarrow$ (1) are obvious.
    The implications (3) $\Rightarrow$ (1) and (4) $\Rightarrow$ (2) follow from \cref{prop:infty-qFs}.
    Finally, we will verify the implication (1) $\Rightarrow$ (4).

    Since $\sO_X(p^lD-S)$ is Cohen-Macaulay, we have the vanishing
    \[
    H^{d-1}_{\m}(\sO_X(p^lD-S))=0
    \]
    for every $l$.
    On the other hand, since $\Delta$ has standard coefficients, we have
    \[
    S+\{p^l \Delta\} \le \Delta 
    \]
    for every $l$.
    Therefore, the assumption (ii) in \cref{thm:qFs-to-qF^es} follows from the condition (1).
    Then applying \cref{thm:qFs-to-qF^es}, the condition (4) holds. 
\end{proof}

\subsection{Finite covers}

In this subsection, we consider the behavior of quasi-$F$-splitting under finite \'etale morphisms (\cref{etale-cov}), base changes (\cref{cor:asc-des ind-etale}, \cref{prop:base change descent}) and index one covers (\cref{prop:index-one-cov}).

\begin{proposition}\label{etale-cov}
With the notation as in \cref{setting:global setting}, we further assume that $X$ is normal and integral.
Let $f \colon Y \to X$ be a finite \'etale morphism from a normal integral scheme $Y$ and we set $\Delta_Y := f^*\Delta$.
Then the following hold
\begin{enumerate}[label=\textup{(\arabic*)}]
    \item If $(X,\Delta)$ is purely $n$-quasi-$F^e$-split, then so is $(Y,\Delta_Y)$.
    \item 
    If the morphism $\cO_X \to f_*\cO_Y$ splits as a morphism of $\cO_X$-modules, then the converse implication also holds.
\end{enumerate}
\end{proposition}

\begin{proof}
Since $f$ is \'etale, we have $f^*S = \rdown{\Delta_Y}$.
It follows from \stacksproj{0ATX} and \stacksproj{0FWI} that we have
\[
\omega_Y \cong f^* \omega_X.
\]
In particular, $K_Y \coloneqq f^*K_X$ is a canonical divisor on $Y$.
Similarly, since $W_nf$ is \'etale by \cite{LZ}*{Proposition~A.8}, we have the natural isomorphism 
\[
W_n\omega_Y \cong (W_nf)^*W_n\omega_X 
\]
which fits into the following commutative diagram:
\[
\begin{tikzcd}
    f^*\omega_X \arrow[r,"\sim"] \arrow[d,"\sim" sloped] & (W_nf)^*\omega_X \arrow[r,"(W_nf)^*T_n"] & (W_nf)^*W_n\omega_X \arrow[d,"\sim" sloped] \\
    \omega_Y \arrow[rr,"T_n"] && W_n \omega_Y.
\end{tikzcd}
\]

For (1), since we assume that $(X,\Delta)$ is purely $n$-quasi-$F^e$-split, there exists a $W_n\cO_X$-module homomorphism
\[
\varphi \colon F^e_*W_n\cI_S(p^e\Delta) \longrightarrow W_n\omega_X(-K_X)
\]
such that the following diagram commutes:
\begin{equation*}
\begin{tikzcd}
W_n\cI_S(\Delta) \arrow[rrr,"F^e"] \arrow[d,"R^{n-1}"'] & & & F^e_*W_n\cI_S(p^e\Delta) \arrow[d,"\varphi"] \\
\cO_X \arrow[r,"\alpha^*"] & \omega_X(-K_X)  \arrow[rr,"T_n(-K_X)"]  & & W_n\omega_X(-K_X).
\end{tikzcd}
\end{equation*}
Taking pullback by $W_nf$, we obtain the commutative diagram
\begin{equation*}
\begin{tikzcd}
W_n\cI_{f^*S}(f^*\Delta) \arrow[rrr,"F^e"] \arrow[d,"R^{n-1}"'] & & & F^e_*W_n\cI_{f^*S}(p^ef^*\Delta) \arrow[d,"f^*\varphi"] \\
\cO_Y \arrow[r,"\alpha^*"] & \omega_Y(-K_Y)  \arrow[rr,"T_n(-K_Y)"]  & & W_n\omega_Y(-K_Y)
\end{tikzcd}
\end{equation*}
by \cref{etale-witt}.
Therefore, the pair $(Y,\Delta_Y)$ is purely $n$-quasi-$F^e$-split, as desired.

For (2), we set $D:=K_X+\Delta$ and $D_Y:=K_Y+\Delta_Y$.
Consider the following commutative diagram
\[
\begin{tikzcd}
    H^d_\m(\cO_X(K_X)) \arrow[r,"\Phi^{S,e}_{X,D,n}"] \arrow[d,"\beta"] 
        & H^d_\m(Q^{S,e}_{X,D,n}) \arrow[d,"\gamma"] \\
    H^d_\m(\cO_Y(K_Y)) \arrow[r,"\Phi^{f^*S,e}_{Y,D_Y,n}"] 
        & H^d_\m(Q^{f^*S,e}_{Y,D_Y,n}).
\end{tikzcd}
\]
Since $\cO_X \to f_*\cO_Y$ splits, the homomorphism $\beta$ is injective.
In particular, if $(Y,\Delta_Y)$ is purely $n$-quasi-$F^e$-split, then so is $(X,\Delta)$.
\end{proof}

A Noetherian local ring $(R,\m)$ is said to be a \emph{$G$-ring} if the completion $R \to \widehat{R}$ is a regular map (\stacksproj{07PT}).

\begin{proposition}\label{ind-etale}
Let $(R,\m)$ be a Noetherian local $G$-ring and $\{(R_n,\m_n)\}_{n \geq 1}$ be a direct system of local $R$-algebras such that the ring homomorphism $\varphi_{m,n} \colon R_m \to R_n$ is a localization of an \'etale homomorphism with $\m_n=\m_m R_n$ for every $ n \ge m  \ge 1$.
Assume moreover that $R$ has characteristic $p>0$ and that $R\to R_1$ is a local homomorphism essentially of finite type.
We set $R_{\infty}:=\varinjlim_{n } R_n$ and $\m_{\infty}:=\varinjlim_{n} \m_n$.
\begin{enumerate}[label=\textup{(\arabic*)}]
    \item $(R_\infty,\m_{\infty})$ is a Noetherian local $G$-ring.
    \item If the residue field $R_\infty/\m_{\infty}$ is $F$-finite, then so is $R_{\infty}$.
\end{enumerate}
\end{proposition}

\begin{proof}
We first note that for every $n \ge 1$, $R_n$ is also a $G$-ring by \stacksproj{07PV}.
The assertion (1) follows from a similar argument as in the proof of \stacksproj{06LJ} and \stacksproj{07QR}.

For (2), it follows from Cohen's structure theorem that the $\m_{\infty}$-adic completion of $R_{\infty}$ is $F$-finite.
Since $R_{\infty}$ is a $G$-ring, the $F$-finiteness of $R_{\infty}$ follows from \cite{Hashimoto15}*{Corollary~20}.
\end{proof}

\begin{cor}\label{cor:asc-des ind-etale}
    With the notation as in \cref{setting:global setting}, we further assume that $X$ is normal and integral.
    Let $\{(R_l,\m_l)\}_{l \ge 1}$ be a direct system of local rings which are localizations of finite \'etale $R$-algebras by their maximal ideals and $(R_{\infty},\m_{\infty})$ be its direct limit.
    We further assume that the residue field $R_{\infty}/\m_{\infty}$ is $F$-finite.
    Let $f \colon Y \coloneqq X \times_R \Spec R_{\infty} \to X$ be the projection and $\Delta_Y := f^*\Delta$ be the flat pullback of $\Delta$.
    Then the following hold:
\begin{enumerate}[label=\textup{(\arabic*)}]
    \item If $(X,\Delta)$ is purely $n$-quasi-$F^e$-split, then so is $(Y,\Delta_Y)$.
    \item 
    If the morphism $R \to R_{\infty}$ splits as $R$-modules, then the converse implication also holds.
\end{enumerate}
\end{cor}
\begin{proof}
    Since $R_{l+1}$ is the localization of an \'etale $R_l$-algebra at a maximal ideal (cf.~\stacksproj{00U7}), it follows from \cref{ind-etale} that $R_{\infty}$ is Noetherian and $F$-finite.
    We first prove that 
    \[
    W_m\omega_{R_{\infty}}^{\bullet} \coloneqq (W_mg)^*W_m\omega_R^{\bullet}
    \]
    is a dualizing complex of $W_m R_{\infty}$, where $g \colon \Spec R_{\infty} \to \Spec R$ is the natural morphism.
    Since $W_mR_l$ is the localization by a maximal ideal of an \'etale $W_mR$-algebra (\cite{LZ}*{Proposition A.8} and \eqref{eq: localization and Witt}), the maximal ideal $\m_{m,l}$ of the local ring $W_mR_l$ coincides with $\m_m \cdot W_mR_l$, where $\m_m$ is the maximal ideal of $W_mR$.
    Therefore, the maximal ideal 
    \[
    \m_{m,\infty} = \varinjlim_l \m_{m,l}
    \]
    of $W_mR_{\infty}$ coincides with $\m_{m} \cdot W_mR_{\infty}$.
    It then follows from \stacksproj{0AWD} that the flat pullback $W_m\omega_{R_{\infty}}^{\bullet} \coloneqq (W_mg)^*W_m\omega_R^{\bullet}$ is a dualizing complex, as claimed.
    We also note that the sequence $\{W_m\omega_{R_{\infty}}^{\bullet}\}_m$ satisfies the condition $(*)$ in Subsection \ref{subsec:Witt dualizing} by \stacksproj{0AA8} and \stacksproj{0ATQ}.

    We next show that the following diagram is Cartesian:
    \[
    \begin{tikzcd}
        W_m Y \arrow[r,"W_mf"] \arrow[d,"W_m\pi_Y"] & W_m X \arrow[d, "W_m\pi_X"] \\
        \Spec W_m R_{\infty} \arrow[r, "W_mg"] & \Spec W_m R,
    \end{tikzcd}
    \]
    where $\pi_X \colon X \to \Spec R$ and $\pi_Y \colon Y \to \Spec R_{\infty}$ are the natural morphisms.
    In order to prove this, we may assume that $X = \Spec A$ for some ring $A$.
    Then the above diagram is Cartesian by the following isomorphism 
    \begin{align*}
        W_m(A \otimes_R R_{\infty}) & \cong \varinjlim_l W_m(A \otimes_R R_l) \\
        & \cong \varinjlim_l (W_m A \otimes_{W_mR} W_mR_l) \\
        & \cong W_m A \otimes_{W_m R} W_m R_{\infty},
    \end{align*}
    where the second line follows from \cite{LZ}*{Corollary A.12} and the isomorphism \eqref{eq: localization and Witt}.
    
    As in \cref{setting:define purely qFS}, we define the dualizing complex $W_m\omega_Y^{\bullet}$ of $W_m Y$ by $W_m\omega_Y^{\bullet} \coloneqq (W_m\pi_Y)^{!}W_m\omega_{R_{\infty}}^{\bullet}$.
    It follows from \stacksproj{0AA8} that we have the natural isomorphism 
    \[
    W_m\omega_Y^{\bullet} \cong (W_mf)^*W_m \omega_{X}^{\bullet}.
    \]
    In particular, $K_Y \coloneqq f^*K_X$ is the canonical divisor of $Y$ associated to $W_1\omega_Y^{\bullet}$.
    
    Finally, we write $Y \xrightarrow{g_l} Y_l \coloneqq X \times_R \Spec R_l \xrightarrow{f_l} X$. 
    Then we have $(W_mf)^*\cF \cong \varinjlim_l (W_mg_l)^{-1}(W_mf_l)^* \cF$ for a coherent $W_mX$-module $\cF$ since one has 
    \[
    W_m Y \cong W_m X \times_{\Spec W_mR} \Spec (\varinjlim_l W_mR_l).
    \]

    For (1), since $(X, \Delta)$ is purely $n$-quasi-$F^e$-split, there is a homomorphism
    \[
    \varphi \colon F^e_*W_n\cI_S(p^e\Delta) \to W_n\omega_X(-K_X)
    \]
    which fits into the commutative diagram \eqref{eq:def:p-qFs}.
    Then, by \cref{etale-cov} (1), we obtain the following commutative diagram:
\[
\begin{tikzcd}
W_n\cI_{f_l^*S}(f_l^*\Delta) \arrow[rrr,"F^e"] \arrow[d,"R^{n-1}"'] & & &
F^e_*W_n\cI_{f_l^*S}(p^ef_l^*\Delta) \arrow[d,"(W_nf_l)^*\varphi"] \\
\cO_{Y_l} \arrow[r,"\alpha^*"] &
\omega_{Y_l}(-K_{Y_l}) \arrow[rr,"T_n(-K_{Y_l})"] & &
W_n\omega_{Y_l}(-K_{Y_l})
\end{tikzcd}
\]
    Applying the functor $(W_ng_l)^{-1}$ and taking the direct limit, we conclude that $(Y,\Delta_Y)$ is purely $n$-quasi-$F^e$-split, as desired.

    For (2), noting that $\m R_{\infty}$ is the maximal ideal $\m_{\infty}$ of $R_{\infty}$, we have $H^d_{\m_{\infty}}(\cF) \cong H^d_{\m}(\cF)$ for a sheaf of modules $\cF$ on $Y$.
    Then the assertion follows from a similar argument as in the proof of \cref{etale-cov} (2).
\end{proof}

\begin{prop}\label{prop:base change descent}
    Let $X$ be an equidimensional reduced projective scheme over an $F$-finite infinite field $k$ satisfying $\Serre{2}$ condition, and $\Delta \ge 0$ be a Mumford $\Q$-divisor on $X$.
    Suppose that $l$ is a field extension of $k$ which is $F$-finite.
    We further assume that 
    \begin{enumerate}[label=\textup{(\roman*)}]
    \item $X$ is a locally complete intersection over $k$ in codimension one,
    \item the base change $X_l \coloneqq X \times_k \Spec l$ is reduced, and 
    \item the flat pullback $\Delta_l \coloneqq f^*\Delta$ of $\Delta$ is a Mumford $\Q$-divisor with coefficients less than $1$, where $f \colon X_l \to X$ is the projection.
    \end{enumerate}
    If $(X_{l},\Delta_l)$ is $n$-quasi-$F^e$-split, then so is $(X,\Delta)$.
\end{prop}

\begin{proof}
    It follows from \cref{rem:on setting} that $(X,\Delta)$ and $(X_l, \Delta_l)$ satisfy the assumptions in \cref{setting:define purely qFS}.
    Since $k$ is infinite, we may choose a canonical Mumford divisor $K_X$ such that the flat pullback $f^*K_X$ is also a Mumford divisor (\cite{ST23}*{Lemma A.17}).
    We also note that $K_{X_l} \coloneqq f^*K_X$ is a canonical divisor on $X_l$ (\stacksproj{0EA0} and \cite{hartshorne77}*{Proposition II.8.10}).
    
    Consider the following commutative diagram:
    \[
        \begin{tikzcd}
            H^d(X, \sO_X(K_X)) \arrow[rr,"\Phi^{0,e}_{X,K_X+\Delta,n}"] \arrow[d,"\beta"] && H^d(X, Q^{0, e}_{X,K_X+\Delta,n}) \arrow[d]\\
            H^d(X_l, \sO_{X_l}(K_{X_l})) \arrow[rr,"\Phi^{0,e}_{X_l,K_{X_l}+\Delta_l,n}"] && H^d( X_l, Q^{0,e}_{X_l, K_{X_l}+\Delta_l,n}).
        \end{tikzcd}
    \]
    Noting that the inclusion $k \into l$ splits as $k$-modules, the natural morphism 
    \[
    \sO_X(K_X) \to f_* \sO_{X_l}(K_{X_l}) \cong \sO_X(K_X) \otimes_k l
    \]
    splits as $\sO_X$-modules.
    This implies that the morphism $\beta$ is injective.
    Combining this with \cref{prop:p-qFS equivalent def} (c), we complete the proof of the proposition.
\end{proof}

\begin{proposition}\label{prop:index-one-cov}
With the notation as in \cref{setting:global setting}, we further assume that
\begin{enumerate}[label=\textup{(\alph*)}]
    \item $X$ is normal and integral,
    \item $\Delta$ has standard coefficients, and 
    \item $N(K_X+\Delta)$ is principal, where $N$ is the Cartier index of $K_X+\Delta$.
\end{enumerate}
Let $a,r \ge 0$ be the integers such that $r$ is not divisible by $p$ and $N=p^ar$.
Then there exists a finite separable morphism $f \colon Y \to X$ from a normal integral scheme $Y$ which satisfies the following properties:
\begin{enumerate}[label=\textup{(\arabic*)}]
    \item $\Delta_Y \coloneqq f^*\Delta - \mathrm{Ram}_f$ has standard coefficients,
    \item the principal index of $K_Y+\Delta_Y$ (that is, the smallest number $n>0$ such that $n(K_Y+\Delta_Y)$ is principal) is $p^a$, and
    \item $(X,\Delta)$ is purely $n$-quasi-$F^e$-split if and only if so is $(Y,\Delta_Y)$,
\end{enumerate}
where $\mathrm{Ram}_f$ is the ramification divisor of $f$ (see \cite{schwedetucker14}*{Definition 4.5} for the definition).
\end{proposition}

\begin{proof}
Let $\F_{q}$ be a finite field which contains a primitive $r$-th root of unity and $X'$ be a connected component of $X \times_{\F_p} \Spec \F_{q}$.
Noting that the natural morphism $\phi \colon X' \to X$ is finite and \'etale, the pair $(X',\phi^*\Delta)$ also satisfies the assumptions of the proposition.
Moreover, it follows from \cref{etale-cov} that $(X,\Delta)$ is purely $n$-quasi-$F^e$-split if and only if so is $(X',\phi^*\Delta)$.
Since $\phi^*(K_X+\Delta) = K_{X'} + \phi^*\Delta$, it follows from \cref{etale-cover-principal-index} that the principal index of $K_{X'}+\phi^*\Delta$ is also $p^ar$.
After replacing $X$ by $X'$, we may assume that $H^0(\cO_X)$ contains a primitive $r$-th root of unity.

We will prove that the index-one cover 
\[
f \colon Y=\mathbf{Spec}_X\left(\bigoplus_{m=0}^{r-1}\sO_X(mp^a(K_X+\Delta))t^m \right) \to X
\]
of $p^a(K_X+\Delta)$ satisfies the properties (1), (2) and (3).
The assertion (2) follows from
\[
p^a(K_Y+\Delta_Y) = p^af^*(K_X+\Delta) = \div_Y(t).
\]

For (1), we write $\Delta=\Delta^{=1}+\sum (1-\tfrac{1}{m_i})E_i$.
Take integers $a_i,r_i \ge 0$ so that $m_i=p^{a_i}r_i$ and $r_i$ is not divisible by $p$. 
We compute the coefficients of $\Delta_Y$.
Since the denominator of $\operatorname{ord}_{E_i}(p^a\Delta)$ is $r_i$, we have
\[
f^*E_i = r_i(F_{i,1}+\cdots+F_{i,l_i})
\]
for some prime divisors $F_{i,1},\ldots,F_{i,l_i}$ (cf. \cite{Yoshikawa22}*{Lemma~4.12}).
As in the proof of \cite{hartshorne77}*{Proposition IV 2.2}, if the ramification index $e_F$ at a prime divisor $F \subseteq Y$ is coprime to $p$, then we have $\mathrm{ord}_F(\mathrm{Ram}_f) = e_F-1$.
Therefore, if $m_i \neq 1$, then
\[
\operatorname{ord}_{F_{i,j}}(\Delta_Y)
= - (r_i-1) + \Bigl(1-\tfrac{1}{m_i}\Bigr) r_i
= 1 - \tfrac{1}{p^{a_i}}.
\]
Furthermore, since $f$ is \'etale at the generic points of $\Delta^{=1}$, we conclude that $\Delta_Y$ has standard coefficients.

For (3), we set $S_X = \rdown{\Delta}$ and $S_Y = \rdown{\Delta_Y}$.
It follows from the proof of (1) that we have $S_Y=f^*S_X$.
We also set $D_X=K_X+\Delta$ and $D_Y=K_Y+\Delta_Y$.
Since we have $K_Y=f^*K_X + \mathrm{Ram}_f$ (\cite{schwedetucker14}*{Proposition 4.8}), one has $D_Y = f^*D_X$.

Consider the following commutative diagram
\[
\begin{tikzcd}
    H^d_\m(\cO_X(D_X-S_X)) \arrow[r] \arrow[d,"\alpha"'] &
        H^d_\m(Q^{S_X,e}_{X,D_X,n}) \arrow[d,"\beta"] \\
    H^d_\m(\cO_Y(D_Y-S_Y)) \arrow[r] &
        H^d_\m(Q^{S_Y,e}_{Y,D_Y,n}).
\end{tikzcd}
\]
Noting that $f$ is a Galois cover of degree not divisible by $p$, it follows from a similar argument as in \cite{TWY}*{Proposition~3.21} that the natural maps
\[
\cO_X(D_X-S_X) \to f_*\cO_Y(D_Y-S_Y) \ \textup{and }
Q^{S_X,e}_{X,D_X,n} \to f_* Q^{S_Y,e}_{Y,D_Y,n}
\]
split.
Thus the vertical maps $\alpha$ and $\beta$ in the above commutative diagram are injective.
In particular, if $(Y,\Delta_Y)$ is purely $n$-quasi-$F^e$-split, then so is $(X,\Delta)$.

For the converse implication, we assume that $(X,\Delta)$ is purely $n$-quasi-$F^e$-split.
For a maximal ideal $\n \subseteq A \coloneqq H^0(X,\sO_X)$, we write $X_\n \coloneqq X \times_{A} \Spec A_\n$ and $Y_\n \coloneqq Y \times_A \Spec A_\n$.
We note that the induced morphism $Y_\n \to X_\n$ is also the index one covering of $p^a(K_{X_{\n}}+\Delta|_{X_\n})$.
Moreover, it follows from \cref{prop:p-qFS equivalent def} (b) that $(X, \Delta)$ (resp.~$(Y,\Delta_Y)$) is purely $n$-quasi-$F^e$-split if and only if its localization by every maximal ideal $\n \subseteq A$ is purely $n$-quasi-$F^e$-split.
Therefore, after replacing $X$, $Y$ and $R$ by $X_{\n}$, $Y_{\n}$ and $A_{\n}$, respectively, we may assume that $H^0(X,\sO_X)=R$.

Since $(X,\Delta)$ is purely $n$-quasi-$F^e$-split, the composition
\[
H^d_\m(\cO_X(K_X)) \xrightarrow{\alpha} H^d_\m(\cO_Y(K_Y))
  \xrightarrow{\ \Phi^{S_Y,e}_{Y,D_Y,n}\ } H^d_\m(Q^{S_Y,e}_{Y,D_Y,n})
\]
is injective.
Taking Matlis duals, we obtain a surjection
\[
\Hom_{W_nY}\!\bigl(Q^{S_Y,e}_{Y,D_Y,n},\,W_n\omega_Y \bigr)
\longrightarrow H^0(Y, \cO_Y) \xrightarrow{\ \alpha^*\ } H^0(X, \cO_X).
\]
Thus it suffices to show that for every maximal ideal $\n \subset H^0(Y, \cO_Y)$ we have $\alpha^*(\n)\neq H^0(X, \cO_X)$.

Since one has
\[
H^0(Y, \cO_Y)=\bigoplus_{m=0}^{r-1} H^0(X, \cO_X(mp^aD_X))\,t^m,
\]
it is a local ring with the maximal ideal
\[
\n = \m \oplus H^0(X, \cO_X(p^aD_X))\,t \oplus \cdots \oplus H^0(X, \cO_X((r-1)p^aD_X))\,t^{r-1}.
\]
Since $\alpha^*$ coincides with the trace map for the finite extension $H^0( \cO_X) \to H^0(\cO_Y)$, 
the claim follows from \cite{CST}*{Lemma~2.10}.
Hence $(Y,\Delta_Y)$ is purely $n$-quasi-$F^e$-split, as required.
\end{proof}

\subsection{Normalization}
A reduced ring $A$ is said to be \emph{weakly normal} (resp.~\emph{seminormal}) if the weak normalization (resp.~seminormalization) of $A$ in its total ring of fractions coincides with $A$ itself (see \cite{RRS96}*{Introduction} for the definition of the weak normalization and seminormalization).

\begin{proposition}\label{prop:normalization-qFs}
With the notation as in \cref{setting:define purely qFS}, we further assume that $X=\Spec A$ for some ring $A$.
If $X$ is quasi-$F$-split, then $A$ is weakly normal.
\end{proposition}

\begin{proof}
Since $A$ is quasi-$F$-split, it follows from \cref{rem:compare with original definition} that there exists an integer $n \geq 1$ and a $W_nA$-module homomorphism 
\[
\varphi \colon F_*W_nA \to A
\]
such that $\varphi(F_*1)=1$.
In order to prove that $A$ is weakly normal, it suffices to show that if $x \in A^n$ satisfies $x^p \in A$, then $x \in A$ by \cite{RRS96}*{Theorem~4.3,~6.8}.

Let $x \in A^n$ with $x^p \in A$.
Take $a \in A^\circ$ such that $ax \in A$.
Then
\[
a \cdot \varphi(F_*[x]^p) = \varphi(F_*[ax]^p) = ax.
\]
Since $a$ is a non-zero divisor, it follows that $\varphi(F_*[x]^p) = x \in A$, as desired.
 \qedhere
\end{proof}

\begin{proposition}\label{prop:normalization-qFs-pair}
With the notation as in \cref{setting:define purely qFS}, let $\nu \colon X^n \to X$ be the normalization of $X$.
If $(X,\Delta)$ is purely quasi-$F^e$-split, then so is $(X^n,C+\nu^*\Delta)$, where $C$ is the conductor divisor.
\end{proposition}

\begin{proof}
Since $X$ is weakly normal (\cref{prop:normalization-qFs}), it is seminormal, and thus the conductor divisor $C$ is reduced (\cite{Tra70}*{Lemma 1.3}).
Set $S:=\rdown{\Delta}$ and $S':=\nu^*S$.
Since $(X,\Delta)$ is purely quasi-$F^e$-split, there exists a $W_n\cO_X$-module homomorphism
\[
\varphi \colon F^e_*W_n\cI_S(p^e\Delta) \to W_n\omega_X(-K_X)
\]
fitting into the commutative diagram \eqref{eq:def:p-qFs}.
Noting that we have $\nu_*\sO_{X^n}(-C)\subset \sO_X$, we obtain a $W_n\cO_X$-module homomorphism
\[
\psi' \colon 
\nu_*F^e_*W_n\cI_{C+S'}(p^e\nu^*\Delta)
   \to F^e_*W_n\cI_S(p^e\Delta)
   \to W_n\omega_X(-K_X).
\]

Next, observe that
\[
\begin{aligned}
\Hom_{W_n\cO_X}&\!\left(\nu_*F^e_*W_n\cI_{C+S'}(p^e\nu^*\Delta),W_n\omega_X(-K_X)\right) \\
&\simeq \Hom_{W_n\cO_X}\!\left(\nu_*F^e_*W_n\cI_{C+S'}(p^e\nu^*(K_X+\Delta)),W_n\omega_X\right) \\
&\simeq \Hom_{W_n\cO_{X^n}}\!\left(F^e_*W_n\cI_{C+S'}(p^e\nu^*(K_X+\Delta)),W_n\omega_{X^n}\right) \\
&\simeq \Hom_{W_n\cO_{X^n}}\!\left(F^e_*W_n\cI_{C+S'}(p^e(C+\nu^*\Delta)),W_n\omega_{X^n}(-K_{X^n})\right).
\end{aligned}
\]
Thus we obtain a $W_n\cO_{X^n}$-module homomorphism
\[
\psi \colon F^e_*W_n\cI_{C+S'}(p^e(C+\nu^*\Delta))
\to W_n\omega_{X^n}(-K_{X^n})
\]
such that the following diagram commutes:
\[
\begin{tikzcd}
    \nu_*F^e_*W_n\cI_{C+S'}(p^e\nu^*\Delta) \arrow[r,"\psi'"] \arrow[dd,hookrightarrow] 
      & W_n\omega_X(-K_X) \\
     & \nu_*W_n\omega_{X^n}(-\nu^*K_X)
        =\nu_*W_n\omega_{X^n}(-K_{X^n}-C) \arrow[u,"T_{\nu}"] \arrow[d,hookrightarrow] \\
    \nu_*F^e_*W_n\cI_{C+S'}(p^e(C+\nu^*\Delta)) \arrow[r,"\psi"] 
      & \nu_*W_n\omega_{X^n}(-K_{X^n}), 
\end{tikzcd}
\]
where $T_\nu$ is induced by the Grothendieck trace map $\nu_*W_n\omega_{X^n} \to W_n\omega_X$.
In particular, $\psi$ coincides with $\psi'$ in the total quotient ring of $\cO_X$.
Therefore, $\psi$ fits into the commutative diagram \eqref{eq:def:p-qFs}, as desired.
\end{proof}

\section{Quasi-test ideals for numerically \texorpdfstring{$\Q$}{Q}-Cartier pairs}

\begin{notation}\label{notation:test-section}
Let $(R,\m)$ be an $F$-finite Noetherian normal local ring of characteristic $p>0$,
and set $d:=\dim R$.
For a $\Q$-Weil divisor $D$ on $\Spec R$, we define the following homomorphisms:
\begin{align*}
    F^e_{D,n} &\colon H^d_\m(W_nR(D)) \longrightarrow F^e_*H^d_\m(W_nR(p^eD)),
        && (n \geq 1), \\
    V_{D,n} &\colon F_*H^d_\m(W_{n-1}R(pD)) \longrightarrow H^d_\m(W_nR(D)),
        && (n \geq 2), \\
    R_{D,n} &\colon H^d_\m(W_nR(D)) \longrightarrow H^d_\m(R(D)),
        && (n \geq 1),
\end{align*}
where the maps are induced by $F^e$, $V$, and $R^{n-1}$, respectively.
\end{notation}

\begin{definition}\label{defn:t-c 0}
We use \cref{notation:test-section}.
Let $D$ be a $\Q$-Weil divisor on $\Spec R$.
Take integers $n>0$ and $e\geq 0$. 
\begin{enumerate}
\item For $c \in R^{\circ}$, we define the $W_nR$-modules 
$Q^e_{R,D,n}$ and $Q^{e,c}_{R,D,n}$,
together with the homomorphisms 
$\Phi^e_{R,D,n}$ and $\Phi^{e,c}_{R,D,n}$,
by the following pushout diagram:
\begin{equation}\label{e-def-Q^{e,c}_D}
\begin{tikzcd}
W_nR(D) \arrow[r,"F^e"] \arrow[d,"R^{n-1}"'] 
  & F_*^e(W_nR(p^eD)) \arrow[d] \arrow[r,"{\cdot\,F_*^e[c]}"] 
  & F_*^e(W_nR(p^eD)) \arrow[d] \\
R(D) \arrow[r,"\Phi^e_{R,D,n}"] \arrow[rr,"\Phi^{e,c}_{R,D,n}"',bend right]
  & Q^e_{R,D,n} \arrow[r] & Q^{e,c}_{R,D,n}.
\end{tikzcd}
\end{equation}

\item For $c \in R^{\circ}$, set
\[
\widetilde{K^{e,c}_{R,D,n}}
:= \Ker \Bigl(
H^d_\m(W_nR(D)) \xrightarrow{F^e} 
H^d_\m(F^e_*W_nR(p^eD))
\xrightarrow{\cdot F^e_*[c]} 
H^d_\m(F^e_*W_nR(p^eD))
\Bigr),
\]
a $W_nR$-submodule of $H^d_\m(W_nR(D))$.  
We define  
\[
\widetilde{0_{R,D,n}^{*}}
:= \bigcup_{\substack{c \in R^{\circ}\\ e_0>0}} 
   \bigcap_{e \geq e_0} \widetilde{K^{e,c}_{R,D,n}}
   \subseteq H^d_\m(W_nR(D)).
\]

Equivalently, $z\in H^d_\m(W_nR(D))$ lies in 
$\widetilde{0_{R,D,n}^{*}}$ if and only if there exist 
$c \in R^{\circ}$ and $e_0>0$ such that 
$z \in \widetilde{K^{e,c}_{R,D,n}}$ for all $e \geq e_0$.
Thus $\widetilde{0_{R,D,n}^{*}}$ is a $W_nR$-submodule.

\item For $c \in R^{\circ}$, set
\[
K^{e,c}_{R,D,n}
:= \Ker\!\left(
H^d_\m(R(D)) \xrightarrow{\Phi^{e,c}_{R,D,n}}
H^d_\m(Q^{e,c}_{R,D,n})
\right),
\]
an $R$-submodule of $H^d_\m(R(D))$.
We define the {\em $n$-quasi-tight closure}
\[
0_{R,D,n}^{*}
:= \bigcup_{\substack{c \in R^{\circ}\\ e_0>0}}
   \bigcap_{e \geq e_0} K^{e,c}_{R,D,n}
   \subseteq H^d_\m(R(D)).
\]

Equivalently, $z \in H^d_\m(R(D))$ lies in $0_{R,D,n}^{*}$ 
if and only if there exist $c \in R^{\circ}$ and $e_0>0$ such that 
$z \in K^{e,c}_{R,D,n}$ for all $e \geq e_0$.
Thus $0_{R,D,n}^{*}$ is an $R$-submodule of $H^d_\m(R(D))$.
\end{enumerate}
\end{definition}

\begin{definition}\label{defn:van-cond}
Using \cref{notation:test-section}, 
we say that $D$ satisfies the condition $(\star)$ if there exists $s \in R^\circ$
such that 
\[
[s]\cdot \Ker\!\left(V_{p^lD,n}\right) = 0
\quad \text{for all integers $n \geq 2$ and $l \geq 0$},
\]
where $[s]$ denotes the Teichm\"{u}ller lift of $s$ in $W_nR$.
\end{definition}

\begin{remark}\label{rem:cond-star}
With the same notation as above, suppose that there exists $s \in R^\circ$ such that
\[
s \cdot H^{d-1}_\m(R(p^lD)) = 0 \quad \text{for all $l \geq 0$}.
\]
Then, by the exact sequence \eqref{eq:key-sequence-for-WnI}, $D$ satisfies the condition $(\star)$.
\end{remark}

\begin{proposition}[cf.~\cite{KTTWYY3}*{Proposition~4.16}]\label{prop:V-R-0^*}
With \cref{notation:test-section}, let $D$ be a $\Q$-Weil divisor on $\Spec R$ satisfying the condition $(\star)$.
Then the following hold. 
\begin{enumerate}
    \item $R(\wt{0_{R,D, n+1}^{*}}) \subseteq \wt{0_{R,D, n}^{*}}$. 
    \item $V_{D,n}^{-1}(\wt{0_{R,D, n}^{*}}) = F_*\wt{0_{R,pD, n-1}^{*}}$.
\end{enumerate}
In particular, we obtain the exact sequence
\begin{equation}\label{eq:ex-0^*-pre}
0 \to F_*\frac{H^d_\m(W_{n-1}R(p^{l+1}D))}{\wt{0^*_{R,p^{l+1}D,n-1}}} \xrightarrow{V_{p^lD,n}} \frac{H^d_\m(W_nR(p^lD))}{\wt{0^*_{R,p^lD,n}}} \xrightarrow{R_{p^lD,n}} \frac{H^d_\m(R(p^lD))}{R_{p^lD,n}(\wt{0^*_{R,p^lD,n}})} \to 0
\end{equation}
for all integers $n \geq 2$ and $l \geq 0$.
\end{proposition}

\begin{proof}
We may apply the same argument as the proof of \cite{KTTWYY3}*{Proposition~4.16} by using the condition $(\star)$.
\end{proof}

\begin{proposition}[cf.~\cite{KTTWYY3}*{Proposition~4.17}]\label{prop:van-tilde-0*}
With \cref{notation:test-section}, let $D$ be a $\Q$-Weil divisor on $\Spec R$ satisfying the condition $(\star)$. 
Choose $t \in R^\circ$ such that $t \in \tau(R,\{p^lD\})$ for every integer $l \geq 0$.
Such an element exists because $D$ is a $\Q$-Weil divisor.
Then 
\[
[t^2] \cdot \wt{0^*_{R,p^lD,n}}=0
\]
for all $l \geq 0$ and $n \geq 1$. 
In particular, 
\[
[t^2] \cdot \Ker(F^e_{p^lD,n})=0
\]
for all $l \geq 0$, $n \geq 1$ and $e \geq 1$. 
\end{proposition}

\begin{proof}
We argue by induction on $n$.

If $n=1$, then $[t] \cdot \wt{0^*_{R,p^lD,1}}=0$ for every $l \geq 0$ by the choice of $t$.

Assume $n \geq 2$ and take $\alpha \in \wt{0^*_{R,p^lD,n}}$.  
By \cref{prop:V-R-0^*}(1) we have $R_{p^lD,n}([t]\alpha)=0$.  
Hence, by the exact sequence \eqref{eq:ex-0^*-pre}, there exists 
$\beta \in \wt{0^*_{R,p^{\,l+1}D,n-1}}$ such that 
\[
V(F_*\beta) = [t]\alpha.
\]
Therefore,
\[
[t^2]\alpha = [t]\,V(F_*\beta) = V(F_*[t^p]\beta) 
\overset{(\star_1)}{=} 0,
\]
where $(\star_1)$ follows from the induction hypothesis (since $\beta \in \wt{0^*_{R,p^{\,l+1}D,n-1}}$ and $p \geq 2$).  

Finally, since $\Ker(F^e_{p^lD,n}) \subseteq \wt{0^*_{R,p^lD,n}}$, the last assertion follows immediately.
\end{proof}

\begin{proposition}[cf.~\cite{KTTWYY3}*{Proposition~4.19}]\label{prop:stab-tilde}
We use the same notation as in \cref{notation:test-section}.
Let $D$ be a $\Q$-divisor on $\Spec R$ satisfying the condition $(\star)$.
Take $t \in R^\circ$ such that $t \in \tau(R,\{p^lD\})$ for every integer $l \geq 0$. 
Fix $c \in R^{\circ} \cap (t^4)$. 
Then the following hold for every integer $n \geq 1$. 
\begin{enumerate}
    \item We have the descending chain
    \[
    H^d_{\m}(W_nR(D)) \;\supseteq\;
    \wt{K^{0,c}_{R,D,n}} \;\supseteq\;
    \wt{K^{1,c}_{R,D,n}} \;\supseteq\; \cdots 
    \;\supseteq\; \wt{K^{e,c}_{R,D,n}} 
    \;\supseteq\; \wt{K^{e+1,c}_{R,D,n}} \;\supseteq\; \cdots.
    \]
    \item There exists an integer $e_0 \geq 0$ such that 
    \[
    \wt{0_{R,D, n}^{*}} \;=\; \wt{K^{e,c}_{R,D, n}}
    \quad\text{for every $e \geq e_0$.}
    \]
\end{enumerate}    
\end{proposition}

\begin{proof}
We take $c' \in R^\circ$ with $c = c't^4$.

We prove (1).
Let $z \in \wt{K^{e+1,c}_{R,D, n}}$. Then
\[
0 \;=\; F^{e+1}_*\!\bigl[c'^{\,p-1}t^{2p-4}c\bigr]\cdot F^{e+1}(z)
   \;=\; F\!\left(F^e_*[c't^2]\cdot F^e(z)\right).
\]
By \cref{prop:van-tilde-0*}, we obtain 
\[
[t^2][c't^2]F^e(z) \;=\; [c]F^e(z) \;=\; 0,
\]
as desired.

Next, we prove (2).
Since $H^d_\m(W_nR(D))$ is an Artinian $W_nR$-module, (1) implies that there exists $e_0>0$ such that
\[
\wt{K^{e, c}_{R,D, n}} = \wt{K^{e_0, c}_{R,D, n}}
\quad\text{for all $e \geq e_0$.}
\]
From \cref{defn:t-c 0}(2) we obtain
\[
\wt{0^{*}_{R,D, n}}
\;\supseteq\; \bigcap_{e \geq 0}\wt{K^{e, c}_{R,D, n}}
\;=\; \wt{K^{e_0, c}_{R,D, n}}.
\]

Conversely, let $z \in \wt{0^{*}_{R,D, n}}$ and $e \geq 0$. Then
\begin{align*}
F_*^e[c]\cdot F^e(z) 
&\in F_*^e\bigl([c]\cdot \wt{0^{*}_{R,p^eD, n}}\bigr) \\
&=F_*^e\bigl([c't^4]\cdot \wt{0^{*}_{R,p^eD, n}}\bigr) \\
&=0,
\end{align*}
where the last equality follows from \cref{prop:van-tilde-0*}.
Hence $z \in \bigcap_{e \geq 0}\wt{K^{e, c}_{R,D, n}}
= \wt{K^{e_0, c}_{R,D, n}}$, proving (2).
\end{proof}

\begin{theorem}[cf.~\cite{KTTWYY3}*{Theorem~4.20}]\label{thm: lifting thm pair}
We use  \cref{notation:test-section}. 
Let $D$ be a $\Q$-Weil divisor on $X=\Spec{R}$ satisfying the condition $(\star)$.
Then 
\[
R_{D,n}(\wt{0^{*}_{R,D, n}})=0^{*}_{R,D, n}.
\]
In particular, we obtain the exact sequence
\begin{equation}\label{eq:ex-0^*}
    0 \to 
    F_*\frac{H^d_\m(W_{n-1}R(p^{l+1}D))}{\wt{0^*_{R,p^{l+1}D,n-1}}}
    \xrightarrow{V_{p^lD,n}} 
    \frac{H^d_\m(W_nR(p^lD))}{\wt{0^*_{R,p^lD,n}}}
    \xrightarrow{R_{p^lD,n}} 
    \frac{H^d_\m(R(p^lD))}{0^*_{R,p^lD,n}} 
    \to 0
\end{equation}
for all integers $l \geq 0$ and $n \geq 2$.
\end{theorem}

\begin{proof}
We may apply the same argument as in the proof of \cite{KTTWYY3}*{Theorem~4.20}, using \cref{prop:stab-tilde} in place of \cite{KTTWYY3}*{Proposition~4.19}.
\end{proof}

\begin{proposition}[cf.~\cite{KTTWYY3}*{Proposition~4.21}]\label{prop: test element pair non-tilda}
We use  \cref{notation:test-section}.
Let $D$ be a $\Q$-Weil divisor on $X$ satisfying the condition $(\star)$.
Take $t \in R^\circ$ such that $t \in \tau(R,\{rD\})$ for every integer $r$.
Fix $c \in R^{\circ} \cap (t^4)$.
Then the following hold for every integer $n \geq 1$:
\begin{enumerate}
    \item 
    \[
    H^d_{\m}(R(D)) \supseteq K^{0,c}_{R,D,n}
    \supseteq K^{1,c}_{R,D,n} \supseteq \cdots \supseteq K^{e,c}_{R,D,n}
    \supseteq K^{e+1,c}_{R,D,n} \supseteq \cdots.
    \]
    \item 
    There exists an integer $e_1 \geq 0$ such that 
    \[
    0^{*}_{R,D,n} = K^{e,c}_{R,D,n}
    \]
    for every integer $e \geq e_1$. 
    \item 
    If there exists an integer $e \geq 0$ such that $K^{e,c}_{R,D,n}=0$, then $0^{*}_{R,D,n}=0$.
\end{enumerate}
\end{proposition}

\begin{proof}
The assertions \textup{(1)} and \textup{(2)} follow from \cref{prop:stab-tilde}, 
\cref{thm: lifting thm pair}, and the equality 
\[
R_{D,n}(\wt{K^{e,c}_{R,D,n}})=K^{e,c}_{R,D,n}.
\]
Assertion \textup{(3)} follows immediately from (1) and (2).
\end{proof}

\begin{lemma}\label{lemma:univ-van-H^d-1}
We use \cref{notation:test-section}.
Let $D$ be a numerically $\Q$-Cartier $\Q$-Weil divisor on $\Spec R$.
Then $D$ satisfies the condition $(\star)$.
\end{lemma}

\begin{proof}
Take a projective birational morphism $f \colon Y \to \Spec R$ such that there exists a $\Q$-Cartier divisor $D_Y$ on $Y$ such that $f_*D_Y = D$ and $D_Y$ is $f$-numerically trivial.
Since $f_*\cO_Y(rD_Y)=R(rD)$, we obtain an $R$-module homomorphism
\[
\varphi_r \colon H^{d-1}_\m(R(rD)) \simeq H^{d-1}_\m(f_*\cO_Y(rD_Y)) \to H^{d-1}_\m(\cO_Y(rD_Y)).
\]

By \cite{Keeler03}*{Theorem~1.5}, there exists an ample Cartier divisor $A$ on $Y$ such that 
\[
R^if_*\cO_Y(rD_Y+A)=0 \quad \text{for all $i \geq 1$ and $r \in \Z$}.
\]
By \cref{rem:become-effective}, there exists $a \in R^\circ$ with ${\rm div}_R(a) \geq f_*A$, hence ${\rm div}_Y(a) \geq A$ by the negativity lemma. 
Therefore $a \cdot \Ker(\varphi_r)=0$.
Thus it suffices to show that there exists a nonzerodivisor in $R$ annihilating $H^{d-1}_\m(\cO_Y(rD_Y))$ for all $r$.

Since $H^{d-1}_\m(\cO_Y(rD_Y))$ is Matlis dual to 
\[
\mathcal{H}^{-(d-1)}Rf_*R\cHom_{\cO_Y}(\cO_Y(rD_Y),\omega_Y^\mydot),
\]
where $\omega_R^\mydot$ is a normalized dualizing complex of $R$ and $\omega_Y^\mydot:=f^!\omega_R^\mydot$, 
it is enough, by considering the associated spectral sequence, to find $s_1,s_2 \in R^\circ$ such that 
\[
s_1 \cdot R^if_*\cO_Y(K_Y-\rdown{rD_Y})=0 \quad (i \geq 1), 
\qquad 
s_2 \cdot \mathcal{E}xt^j(\cO_Y(rD_Y),\omega_Y^\mydot)=0 \quad (j \geq -d+1).
\]

The existence of $s_1$ follows from the same argument used to construct $a$ above. 
Moreover, for each $r$ there exists $s_{2,r} \in R^\circ$ such that 
\[
s_{2,r} \cdot \mathcal{E}xt^j(\cO_Y(rD_Y),\omega_Y^\mydot)=0 \quad \text{for all $j \geq -d+1$}.
\]
Since $D_Y$ is $\Q$-Cartier, the set $\{\cO_Y(rD_Y)\}_{r \in \Z}$ is locally finite up to isomorphism. 
Hence we can choose a single $s_2 \in R^\circ$ working for all $r$, as required.
\end{proof}

\begin{definition}
We use \cref{notation:test-section}.
Let $D$ be a numerically $\Q$-Cartier $\Q$-Weil divisor on $\Spec R$.
For a $\Q$-Weil divisor $B$ on $\Spec R$, we use the notation
\[
W_n\omega_R(-B)
:=\Hom_{W_nR}(W_nR(B),W_n\omega_R).
\]
Take $t \in R^\circ$ with $t \in \tau(R,\{p^lD\})$ for every integer $l \geq 0$, and fix $c \in (t^4) \cap R^\circ$.
\begin{enumerate}
    \item We define a submodule $\tau^c(W_n\omega_R,D)$ of $W_n\omega_R(-D)$ by
    \[
    \tau^c(W_n\omega_R,D) := \sum_{e \geq 1} T^e_n\bigl(F^e_*([c]\,W_n\omega_R(-p^eD))\bigr),
    \]
    where $T^e_n$ denotes the $W_n\omega_R$-dual of 
    \[
    F^e \colon W_nR(D) \longrightarrow F^e_*W_nR(p^eD).
    \]
    Furthermore, we define a submodule $\tau^c_n(\omega_R,D)$ of $\omega_R(-D)$ by
    \[
    \tau^c_n(\omega_R,D) := T_{1,n}^{-1}\bigl(\tau^c(W_n\omega_R,D)\bigr),
    \]
    where $T_{1,n} \colon \omega_R(-D) \to W_n\omega_R(-D)$ is the $W_n\omega_R$-dual of the restriction map $W_nR(D) \to R(D)$.
    \item Since $\tau^c_n(\omega_R,D) \subseteq \tau^c_{n+1}(\omega_R,D)$ for every integer $n \geq 1$, 
    there exists an integer $n_0 \geq 1$ such that 
    \[
    \tau^c_{n_0}(\omega_R,D) = \tau^c_n(\omega_R,D)
    \quad \text{for every } n \geq n_0.
    \]
    We then define 
    \[
    \tau^{c,q}(\omega_R,D) := \tau^c_{n_0}(\omega_R,D).
    \]
\end{enumerate}
\end{definition}

\begin{proposition}[cf.~\cite{KTTWYY3}*{Proposition~4.25, Theorem~4.27}]\label{prop:tau-dual}
We use \cref{notation:test-section}.
Let $D$ be a numerically $\Q$-Cartier $\Q$-Weil divisor on $\Spec R$.
Take $t \in R^\circ$ with $t \in \tau(R,\{p^lD\})$ for every integer $l \geq 0$, and fix $c \in (t^4) \cap R^\circ$.
\begin{enumerate}
    \item We have
    \begin{align*}
        \Bigl( \frac{W_n\omega_R(-D)}{\tau^c(W_n\omega_R,D)} \Bigr)^\vee &\simeq \wt{0^*_{R,D,n}}, \\
        \Bigl( \frac{\omega_R(-D)}{\tau^c_n(\omega_R,D)} \Bigr)^\vee &\simeq 0^*_{R,D,n},
    \end{align*}
    where $(-)^\vee$ denotes the Matlis dual.
    \item The submodules $\tau^c(W_n\omega_R,D)$, $\tau^c_n(\omega_R,D)$, and $\tau^{c,q}(\omega_R,D)$ are independent of the choice of $c$.
    \item We have an isomorphism
    \[
    \tau^{c,q}(\omega_R,D) \otimes_R \hat{R} \;\simeq\; \tau^{c,q}(\omega_{\hat{R}},\iota^*D),
    \]
    where $\iota \colon \Spec \hat{R} \to \Spec R$ is the natural morphism.
\end{enumerate}
\end{proposition}

\begin{proof}
The assertion (1) follows by the same argument as in the proof of \cite{KTTWYY3}*{Proposition~4.25}.
Assertions (2) and (3) follow from (1).
\end{proof}

\begin{definition}\label{defn:tau^q}
We use \cref{notation:test-section}.
Let $D$ be a numerically $\Q$-Cartier  $\Q$-Weil divisor on $\Spec R$.
Take $t \in R^\circ$ with $t \in \tau(R,\{p^lD\})$ for every integer $l \geq 0$, and fix $c \in (t^4) \cap R^\circ$.
Then, for simplicity, we denote 
\[
\tau^{c}(W_n\omega_R,D), \quad \tau^c_n(\omega_R,D), \quad \tau^{c,q}(\omega_R,D)
\]
by
\[
\tau(W_n\omega_R,D), \quad \tau_n(\omega_R,D), \quad \tau^q(\omega_R,D),
\]
respectively.
\end{definition}

\begin{theorem}[cf.~\cite{KTTWYY3}*{Corollary~5.8}]\label{thm:comp-test-ideal-mult-ideal}
We use \cref{notation:test-section}.
Let $D$ be a numerically $\Q$-Cartier divisor on $X$.
Then we have $\tau^q(\omega_R,D) \subseteq \mathcal{J}(\omega_R,D)$.
\end{theorem}

\begin{proof}
We may assume that $(R,\m)$ is local and set $d:=\dim R$.
It is enough to show that for every projective birational morphism $f \colon Y \to \Spec R$ from a normal integral scheme $Y$, we have
\[
{\rm Ker}\Big(H^d_\m(R(D)) \to H^d_\m(\cO_Y(D_Y))\Big) \subseteq 0^*_{R,D,n}
\]
for every integer $n \geq 1$, where $D_Y:=f^*_{\num}D$.

We consider the following commutative diagram, in which each horizontal sequence is exact:
\[
\begin{tikzcd}
    F_*H^d_\m(W_{n-1}R(pD)) \arrow[r,"V"] \arrow[d,twoheadrightarrow] & H^d_\m(W_nR(D)) \arrow[r] \arrow[d,twoheadrightarrow] & H^d_\m(R(D)) \arrow[r] \arrow[d,twoheadrightarrow] & 0 \\
    F_*H^d_\m(W_{n-1}\cO_Y(pD_Y)) \arrow[r,"V"] & H^d_\m(W_n\cO_Y(D_Y)) \arrow[r] & H^d_\m(\cO_Y(D_Y)) \arrow[r] & 0.
\end{tikzcd}
\]

Thus, it is enough to prove that 
\[
{\rm Ker}\Big(H^d_\m(W_nR(D)) \to H^d_\m(W_n\cO_Y(D_Y))\Big) \subseteq \wt{0^*_{R,D,n}}
\]
by \cref{thm: lifting thm pair}.

After replacing $Y$ by a higher model, we may assume that $D_Y$ is $\Q$-Cartier.
By the proof of \cref{lemma:univ-van-H^d-1}, there exists an ample Cartier divisor $A$ on $Y$ and $c \in R^\circ$ such that $R^if_*\cO_Y(p^lD_Y+mA)=0$ for all $i,m \geq 1$, $l \geq 0$ and $\mathrm{div}_Y(c) \geq A$.
It follows that 
\[
R^if_*W_n\cO_Y(p^lD_Y+A)=0 \quad \text{for all $i,n \geq 1$, $l \geq 0$}.
\]
Therefore we obtain the commutative diagram
\[
\begin{tikzcd}
    H^d_\m(W_nR(D)) \arrow[r] \arrow[d] & F^e_*H^d_\m(W_nR(p^eD+\mathrm{div}_R(c)))  \\
    H^d_\m(W_n\cO_Y(D_Y)) \arrow[r] & F^e_*H^d_\m(W_n\cO_Y(p^eD_Y+A)) \arrow[u]
\end{tikzcd}
\]
for all $e,n \geq 1$.

Hence
\[
{\rm Ker}\Big(H^d_\m(W_nR(D)) \to H^d_\m(W_n\cO_Y(D_Y))\Big) 
\subseteq \bigcap_{e \geq 1}\wt{K^{e,c}_{R,D,n}} 
\subseteq \wt{0^*_{R,D,n}},
\]
as required.
\end{proof}

\section{Log canonicity of quasi-\texorpdfstring{$F$}{F}-splitting}

\begin{proposition}\label{prop:comp-norm}
We use \cref{notation:test-section}.
Let $D$ be a numerically $\Q$-Cartier $\Q$-Weil divisor on $\Spec R$ and $n \geq 1$ an integer. 
Let $g \in R^\circ$ and $e_0 \in \Z_{\geq 1}$.
We take $t \in R^\circ$ such that $t \in \tau(R,\{p^lD\})$ for every $l \geq 0$.
Let $c =t^8g^4$.
We set $f:=c^pg$ and $E:=(1-1/p^{e_0}){\rm div}(f)$.
Let $\psi_n \colon H^d_\m(W_nR(D)) \to H^d_\m(W_nR(D+E))$ be the natural map for $n \in \Z_{\geq 1}$.
Then there exists an integer $e' \geq e_0$ such that: 
\begin{enumerate}
    \item  $\psi_n^{-1}(\wt{0^*_{R,D+E,n}})=\wt{K^{e,cf^{p^e-p^{e-e_0}}}_{R,D,n}}$ for all  $e \geq e'$,
    \item $\psi_1^{-1}(0^*_{R,D+E,n})=K^{e,cf^{p^e-p^{e-e_0}}}_{R,D,n}$ for all  $e \geq e'$, and
    \item $K^{e,c^pf^{p^e-p^{e-e_0}}}_{R,D,n} \subseteq K^{e_0,c^pf^{p^{e_0}-1}}_{R,D,n}$ for every integer $e \geq e_0$.
\end{enumerate}
In particular, 
\[
\psi_1^{-1}(0^*_{R,D+E,n}) \subseteq K^{e_0,c^pf^{p^{e_0}-1}}_{R,D,n}.
\]
\end{proposition}

\begin{proof}
We set $D':=D+E$ and $f_e:=f^{p^e-p^{e-e_0}}$.
Since $\cdot [f_e] \colon H^d_\m(W_nR(p^eD')) \to H^d_\m(W_nR(p^eD))$ is an isomorphism, we have 
\begin{equation}\label{eq:eq1}
    \psi_n^{-1}(\wt{K^{e,c'}_{R,D',n}}) = \wt{K^{e,c'f_e}_{R,D,n}}
\end{equation}
for $c' \in R^\circ$.
 
We note
\begin{align*}
    t^2g &\in tg \cdot \tau(R,\{p^lD\})
          = \tau(R,\{p^lD\}+{\rm div}(tg)) \\
         &\subseteq \tau(R,\{p^lD'\}).
\end{align*}
Hence there exists an integer $e' \geq e_0$ such that
\begin{equation}\label{eq:eq2}
\wt{K^{e,c}_{R,D',n}}=\wt{0^*_{R,D',n}}
\end{equation}
for all $e \geq e'$ by \cref{prop:stab-tilde}.
Combining \eqref{eq:eq1} and \eqref{eq:eq2} gives (1).

For (2), if $n=1$ the assertion follows from (1).  
Assume $n \geq 2$.  
Applying (1) to $(n,D,f)$ and $(n-1,pD,f^p)$ and replacing $e'$, we obtain
\begin{align*}
    \frac{H^d_\m(F_*W_{n-1}R(pD'))}{F_*\wt{0^*_{R,pD',n-1}}} &\simeq \frac{H^d_\m(F_*W_{n-1}R(pD))}{F_*\wt{K^{e,c^pf_{e+1}}_{R,pD,n-1}}}, \\
    \frac{H^d_\m(W_nR(D'))}{\wt{0^*_{R,D',n}}} &\simeq \frac{H^d_\m(W_nR(D))}{\wt{K^{e,cf_e}_{R,D,n}}}
\end{align*}
for all $e \geq e'$.  
Thus, by the exact sequence in  \cref{thm: lifting thm pair},
\[
K^{e,cf_e}_{R,D,n}=R^{n-1}(\wt{K^{e,cf_e}_{R,D,n}})=\psi_1^{-1}(0^*_{R,D',n}),
\]
as desired.

For (3), let $e \geq e_0$ and $\alpha \in \wt{K}^{e+1,c^pf_{e+1}}_{R,D,n}$.  
Then
\[
0=[c^pf_{e+1}]F^{e+1}(\alpha)=F([cf_e]F^e(\alpha)).
\]
By \cref{prop:van-tilde-0*}, we have $[t^2cf_e]F^e(\alpha)=0$.  
Since $c^p \in (ct^2)$, it follows that $[c^pf_e]F^e(\alpha)=0$.  
Hence $\alpha \in \wt{K}^{e,c^pf_e}_{R,D,n}$. Taking the images under $R_{D,n}$, we obtain (3).
\end{proof}

\begin{proposition}\label{prop:comp-split-regular-norm}
We use \cref{notation:test-section}.
Assume  that $K_R$ is effective.
Let $\Delta$ be an effective $\Q$-Weil divisor.
Take $g \in R^\circ$ such that ${\rm div}(g) \geq D:=K_R+\Delta$.
Let $c \in R^\circ$ and set $f:=cg$.
For integers $e \geq 0, n \geq 1$ set
\[
D_e:=K_R+\tfrac{p^e-1}{p^e}\Delta.
\]
Then there exists a natural map
\[
H^d_\m(R(D)) \xrightarrow{\ \cdot f\ } H^d_\m(R(D_e)).
\]
Moreover, we have
\[
f \cdot K^{e,cf^{p^e-1}}_{R,D,n} \subseteq {\rm Ker}\big(H^d_\m(\Phi^{e}_{R,D_e,n})\big).
\]
\end{proposition}

\begin{proof}
We have
\[
f\cdot R(D) \subseteq g \cdot R(D) = R(D-{\rm div}(g)) 
\subseteq R\!\left(\tfrac{p^e-1}{p^e}D\right) 
\overset{(\star_1)}{\subseteq} R(D_e),
\]
where $(\star_1)$ follows from the assumption that $K_R$ is effective.
Applying $H^d_\m(-)$, we obtain the natural map
\[
\cdot f \colon H^d_\m(R(D)) \longrightarrow H^d_\m(R(D_e)).
\]

Next, consider the commutative diagram
\begin{equation}\label{eq:comm-prop-comp'}
\begin{tikzcd}
    W_nR(D) \arrow[r,"F^e"] \arrow[d,"\cdot \lbrack f \rbrack"] 
      & F^e_*W_nR(p^eD) \arrow[r,"\cdot F^e_*\lbrack cf^{p^e-1} \rbrack"] 
          & F^e_*W_nR(p^eD) \arrow[d,"\cdot F^e_*\lbrack g \rbrack"] \\
    W_nR(D_e) \arrow[rr,"F^e"] & & F^e_*W_nR(p^eD_e), 
\end{tikzcd}
\end{equation}
where the right vertical map exists because
\[
[g]W_nR(p^eD) = W_nR(p^eD-{\rm div}(g)) 
\subseteq W_nR((p^e-1)D) \subseteq W_nR(p^eD_e).
\]

From \eqref{eq:comm-prop-comp'}, we obtain the commutative diagram
\[
\begin{tikzcd}[column sep=2cm]
    R(D) \arrow[r,"\Phi^{e,cf^{p^e-1}}_{R,D,n}"] \arrow[d,"\cdot f"] 
        & Q^{e,cf^{p^e-1}}_{R,D,n} \arrow[d,"\cdot F^e_* \lbrack g \rbrack"] \\
    R(D_e) \arrow[r,"\Phi^e_{R,D_e,n}"] & Q^e_{R,D_e,n}.
\end{tikzcd}
\]
Therefore,
\[
f \cdot K^{e,cf^{p^e-1}}_{R,D,n}
   = f \cdot {\rm Ker}(H^d_\m(\Phi^{e,cf^{p^e-1}}_{R,D,n}))
   \subseteq {\rm Ker}(H^d_\m(\Phi^e_{R,D_e,n})),
\]
as desired.
\end{proof}

\begin{theorem}\label{thm:qFs-test-norm}
We use \cref{notation:test-section}.
Let $\Delta$ be an effective $\Q$-Weil divisor such that $D:=K_R+\Delta$ is numerically $\Q$-Cartier.
If $(R,(p^e-1)/p^e \Delta)$ is quasi-$F^e$-split for every $e \in \Z_{\geq 1}$, then there exists $f \in R^\circ$ such that for every rational number $\varepsilon >0$, we have
\[
f \in \tau^q(\omega_{R},D+(1-\varepsilon){\rm div}(f)).
\]
\end{theorem}

\begin{proof}
We may assume that $K_R$ is an effective Weil divisor.
Take $g \in R^\circ$ such that ${\rm div}(g) \geq D$.
Since $D$ is a $\Q$-Weil divisor, there exists $t \in R^\circ$ such that $t \in \tau(R,\{p^lD\})$ for every integer $l \geq 0$.
Set $c:=(t^8g^4)^p$, $f:=cg$, $E_e:=(1-1/p^e){\rm div}(f)$, and $D_e:=K_R+(p^e-1)/p^e\Delta$ for every $e \geq 1$.
It suffices to show that $f \in \tau^q(\omega_{R},D+E_e)$ for every $e \geq 1$.

By \cref{prop:tau-dual}(3), we may further assume that $R$ is complete.
Fix $e \geq 1$.
Then there exists an integer $n \geq 1$ such that $(R,(p^e-1)/p^e\Delta)$ is $n$-quasi-$F^e$-split.
Consider the composition
\[
\begin{tikzcd}
    \dfrac{H^d_\m(R(D+E_e))}{0^*_{R,D+E_e,n}} 
        \arrow[r,"(\star_3)"] 
      & \dfrac{H^d_\m(R(D))}{K^{e,cf^{p^e-1}}_{R,D,n}} 
        \arrow[r,"(\star_4)"] 
      & \dfrac{H^d_\m(R(D_e))}{{\rm Ker}(H^d_\m(\Phi^e_{R,D_e,n}))} 
        \arrow[r,"(\star_5)"] 
      & H^d_\m(R(D_e)),
\end{tikzcd}
\]
denoted by $\sigma$, where $(\star_3)$ follows from \cref{prop:comp-norm}, $(\star_4)$ is induced by multiplication by $f$ as in \cref{prop:comp-split-regular-norm}, and $(\star_5)$ follows from the fact that ${\rm Ker}(H^d_\m(\Phi^e_{R,D_e,n}))=0$ by the $n$-quasi-$F^e$-splitting of $(R,(p^e-1)/p^e\Delta)$.

Taking the Matlis dual of $\sigma$, we obtain a map
\[
R=\omega_R(-D_e) \xrightarrow{\sigma^\vee} \tau_n(\omega_{R},D+E_e) \subseteq \tau^q(\omega_{R},D+E_e).
\]
Since $\sigma$ is a composition of natural maps and the map induced by multiplication by $f$, the dual $\sigma^\vee$ is given by multiplication by $f$.
In particular, $f \in \tau^q(\omega_{R},D+E_e)$, as desired.
\end{proof}

\begin{theorem}\label{thm:qFs-to-lc-norm}
Let $R$ be an $F$-finite Noetherian normal domain of characteristic $p>0$.
Suppose that $\Delta$ is an effective $\Q$-Weil divisor on $\Spec R$ such that $K_R+\Delta$ is numerically $\Q$-Cartier and $S=\rdown{\Delta}$ is reduced.
\begin{enumerate}[label=\textup{(\arabic*)}]
\item If $(R,(p^e-1)/p^e\Delta)$ is quasi-$F^e$-split for every integer $e \geq 1$, then $(R,\Delta)$ is numerically log canonical. 
\item In particular, if $(R,\Delta)$ is purely quasi-$F$-split, $\Delta$ has standard coefficients and $R(p^l(K_R+\Delta)-S)$ is Cohen-Macaulay for every integer $l \ge 0$, then $(R,\Delta)$ is numerically log canonical.
\end{enumerate}
\end{theorem}

\begin{proof}
After replacing $R$ by its localization, we may assume that $(R,\m)$ is local.
For (1), by \cref{thm:qFs-test-norm}, there exists $f \in R^\circ$ such that for every rational number $0<\varepsilon<1$,
\[
f \in \tau^q(\omega_{R},K_R+\Delta+(1-\varepsilon){\rm div}(f))
      \subseteq \mathcal{J}(R,\Delta+(1-\varepsilon){\rm div}(f)),
\]
where the inclusion follows from \cref{thm:comp-test-ideal-mult-ideal}.
Thus, $(R,\Delta)$ is numerically log canonical by a similar argument as in the proof of \cite{ST25}*{Lemma 3.5}.
The assertion in (2) follows from (1) and \cref{cor:qFs-to-qF^eS}.
\end{proof}

\begin{cor}\label{cor:qFS-to-lc-nonnormal}
Let $(X,\Delta)$ be a pair as in \cref{setting:define purely qFS}.
We further assume that $X=\Spec R$.
We further assume that $\nu^*(K_X+\Delta)$ is numerically $\Q$-Cartier, where $\nu \colon \Spec R^n \to \Spec R$ denotes the normalization.
\begin{enumerate}[label=\textup{(\arabic*)}]
\item If $(R,(p^e-1)/p^e\Delta)$ is quasi-$F^e$-split for every $e$, then $(R,\Delta)$ is numerically semi log canonical. 
\item In particular, if $(R,\Delta)$ is purely quasi-$F$-split, $\Delta$ has standard coefficients and $R(p^l(K_R+\Delta)-S)$ is Cohen-Macaulay for every integer $l \ge 0$, then $(R,\Delta)$ is numerically semi log canonical.
\end{enumerate}
\end{cor}
\begin{proof}
    The assertion follows from \cref{thm:qFs-to-lc-norm} and \cref{prop:normalization-qFs-pair}.
\end{proof}

\section{Two-dimensional quasi-\texorpdfstring{$F$}{F}-split singularities}

In this section, we prove that a $\Z_{(p)}$-Gorenstein two-dimensional log canonical pair $(X,\Delta)$ is purely quasi-$F^{\infty}$-split (\cref{thm:qFs-lc-Z_p-index}).
We also show that, when $\Delta$ has standard coefficients, the $\Z_{(p)}$-Gorenstein assumption is necessary (\cref{cor:qFs-index}).
As a consequence, we obtain a classification of two-dimensional quasi-$F$-split singularities (\cref{thm:class-qFs}).

\subsection{Birational transformation rule, adjunction and inversion of adjunction}
In this subsection, we study the behavior of quasi-$F$-splitting under birational morphisms (\cref{prop:birational transformation rule}) and restriction to snc divisors (\cref{prop:adj-inv}).

\begin{proposition}\label{prop:birational transformation rule}
    Let $f \colon Y \to X$ be a proper birational morphism between connected normal proper schemes over a Noetherian $F$-finite local ring $(R,\m)$ of characteristic $p>0$.
    Suppose that $\Delta$ is an effective $\Q$-Weil divisor on $X$ such that $(X,\Delta)$ is log canonical.
    Set $\Delta_Y \coloneqq f^*(K_X+\Delta)-K_Y$ and $S_Y \coloneqq \Delta_Y^{=1}$.
    \begin{enumerate}[label=\textup{(\arabic*)}]
    \item If the map 
    \[
    H^d_{\m}(\Phi^{S_Y,e}_{Y,K_Y+\Delta_Y,n}) \colon H^d_{\m}(\sO_Y(K_Y+\Delta_Y-S_Y)) \to H^d_{\m}(Q^{S_Y, e}_{Y,K_Y+\Delta_Y,n})
    \]
    is injective, then $(X, \Delta)$ is purely $n$-quasi-$F^e$-split. 
    When $\Delta_Y$ is effective, this injectivity condition is equivalent to saying that $(Y,\Delta_Y)$ is purely $n$-quasi-$F^e$-split (see \cref{prop:p-qFS equivalent def}).
    \item Conversely, if $R^if_*\sO_Y(p^l(K_Y+\Delta_Y)-S_Y)=0$ for every $i,l>0$, then the pure $n$-quasi-$F^e$-splitting of $(X,\Delta)$ implies the injectivity of the map in \textup{(1)}. 
    \end{enumerate}
\end{proposition}
\begin{proof}
We write $S:=\rdown{\Delta}$ and $D \coloneqq K_X+\Delta$.
Let $\sK_X$ and $\sK_Y$ denote the sheaves of total quotients of $X$ and $Y$, respectively.
The natural isomorphism
\[
f_* W_n \sK_Y \xrightarrow{\sim} W_n \sK_X
\]
induces the morphism
\[
f_*W_n\cI_{S_Y}(f^*D) \to W_n \cI_S(D).
\]
Let $P$ be the $W_n\sO_X$-module such that the following diagram is a pushout:
\[
\begin{tikzcd}
    f_*W_n\cI_{S_Y}(f^*D) \arrow[r,"f_*F^e"] \arrow[d,"f_*R^{n-1}"] & f_*F^e_*W_n\cI_{S_Y}(p^ef^*D) \arrow[d] \\
    f_*\sO_Y(f^*D-S_Y) \arrow[r] & P.
\end{tikzcd}
\]
Then we have the following diagram:
\[
\begin{tikzcd}
    f_*\sO_Y(f^*D-S_Y) \arrow[r] \arrow[d,"u"] \arrow[rr,"f_* \Phi^{S_Y,e}_{Y,f^*D,n}", bend left=10] & P \arrow[r,"v"] \arrow[d,"w"] & f_*Q^{S_Y,e}_{Y,f^*D,n} \\
    \sO_X(D-S) \arrow[r,"\Phi^{S,e}_{X,D,n}"] & Q^{S,e}_{X,D,n},&
\end{tikzcd}
\]
where $u,v,w$ are isomorphisms in codimension one.
Taking the local cohomology, we have
\begin{equation}\label{eq:diagram in birational trans}
\begin{tikzcd}
    H^d_{\m}(\sO_Y(f^*D-S_Y)) \arrow[rr,"H^d_{\m}(\Phi^{S_Y,e}_{Y,f^*D,n})"] && H^d_{\m}(Q^{S_Y,e}_{Y,f^*D,n}) \\
    H^d_{\m}(f_*\sO_Y(f^*D-S_Y)) \arrow[rr,"H^d_{\m}(f_*\Phi^{S_Y,e}_{Y,f^*D,n})"] \arrow[u] \arrow[d,"\sim",sloped] && H^d_{\m}(f_* Q^{S_Y,e}_{Y,f^*D,n}) \arrow[u] \arrow[d,"\sim",sloped] \\
    H^d_{\m}(\sO_X(D-S)) \arrow[rr,"H^d_{\m}(\Phi^{S,e}_{X,D,n})"] && H^d_{\m}(Q^{S,e}_{X,D,n}).
\end{tikzcd}
\end{equation}
For (1), it is enough to show that the left vertical map 
\[
\beta \colon H^d_{\m}(\sO_X(D-S)) \to H^d_{\m}(\sO_Y(f^*D-S_Y))
\]
is injective.
The Matlis dual of $\beta$ is the completion of 
\[
H^0(Y,\cO_Y(K_Y-\rdown{f^*D-S_Y})) \longrightarrow H^0(X,\cO_X),
\]
which is an isomorphism since we have
\[
K_Y-\rdown{f^*D-S_Y}=S_Y-\rdown{\Delta_Y} \geq 0.
\]

For (2), since $\beta$ is an isomorphism by the above argument, it suffices to show that the right vertical map in the diagram \eqref{eq:diagram in birational trans} is an isomorphism.
In order to do this, it is enough to prove the vanishing 
\[
R^if_* Q^{S_Y,e}_{Y,f^*D,n} =0
\]
for every $i>0$.

From the short exact sequence (cf.~\eqref{eq:key-sequence-for-WnI})
\[
0 \to F_*W_{n-1}\cI_{S_Y}(p^{l+1}f^*D) \to W_n\cI_{S_Y}(p^lf^*D) \to \cO_Y(p^lf^*D-S_Y) \to 0,
\]
it follows that
\[
R^i f_* W_n \cI_{S_Y}(p^lf^*D) = 0 \quad \text{for all } i,l,n>0.
\]
Combining this vanishing with the following exact sequence
\[
0 \to F_*W_{n-1}\cI_{S_Y}(pf^*D) \to F^e_*W_n\cI_{S_Y}(p^ef^*D) \to Q^{S_Y,e}_{Y,f^*D,n} \to 0,
\]
we deduce that $R^i f_* Q^{S_Y,e}_{Y,f^*D,n}=0$ for every $i>0$, as desired.
\end{proof}

We next consider an adjunction and inversion of adjunction type result.
Let $X$ be a regular connected Noetherian $F$-finite scheme of positive characteristic, $S$ be a reduced divisor on $X$, and $\Theta$ be a $\Q$-divisor on $X$ such that $(X, S \cup \Supp(\Theta))$ is snc and the support of $\Theta$ contains no irreducible component of $S$.
We note that $\Theta|_S$ is a Mumford $\Q$-divisor on $S$.
Suppose that $u \colon \sO_X(D) \to \sO_S(E)$ is an $\sO_X$-homomorphism, where $D$ is a divisor on $X$ and $E$ is a Mumford divisor on $S$.
Then for every integer $m \ge 1$, we have the $\sO_X$-homomorphism
\[
u^{\otimes m} \colon \sO_X(mD + \rdown{m\Theta} ) \to \sO_S(mE+\rdown{m\Theta|_S}).
\]
By using these $\sO_X$-homomorphisms, we define the $W\sO_X$-homomorphism
\[
W u \colon W \sO_X(D+\Theta) \to W \sO_S(E+\Theta|_S).
\]
We remark that the same construction works even if $E$ is an AC divisor on $S$ (see also \cref{rem:on setting} (iii)). 

\begin{proposition}
    Let $X$ be a connected $d$-dimensional regular proper scheme over an $F$-finite Noetherian local ring $(R,\m)$ of characteristic $p>0$, $S$ be a reduced divisor on $X$, and $\Theta$ be a $\Q$-divisor on $X$.
    Let $K_X$ be a canonical divisor on $X$ and $K_S$ be a canonical Mumford divisor $($or, more generally, a canonical AC divisor in the sense of \cite{ST23}*{Subsection A.2}$)$ on $S$.
    Suppose that the following conditions hold:
    \begin{enumerate}[label=\textup{(\alph*)}]
    \item $(X, S \cup \Supp(\Theta))$ is snc,
    \item the support of $\Theta$ contains no irreducible component of $S$, and 
    \item $H^{d-1}_{\m}(\sO_X(p^l(K_X+S+\Theta)))=0$ for every $l \ge 0$.
    \end{enumerate}
    Then, for every integer $e>0$, the following statements hold:
    \begin{enumerate}[label=\textup{(\arabic*)}]\label{prop:adj-inv}
    \item If the map
    \[
    H^d_{\m}(\Phi^{S,e}_{X,K_X+\Theta+S,n}) \colon H^d_\m(\cO_X(K_X+\Theta)) \longrightarrow H^d_\m(Q^{S,e}_{X,K_X+S+\Theta,n})
    \]
    is injective for some $n \ge 1$, then the morphism
    \[
    H^{d-1}_{\m}(\Phi^{0, e}_{S,K_S+\Theta|_S,m}) \colon H^{d-1}_\m(\cO_S(K_S+\Theta|_S)) \longrightarrow H^{d-1}_\m(Q^{0,e}_{S,K_S+\Theta|_S,m})
    \]
    is also injective for some integer $m$.
    When $\rdown{\Theta}=0$, the above assertion is equivalent to saying that if $(X,S+\Theta)$ is purely quasi-$F^e$-split, then $(S,\Theta|_S)$ is quasi-$F^e$-split (\cref{prop:p-qFS equivalent def}).
    \item If the map 
    \[
    H^d_{\m}(F^e) \colon H^{d}_{\m}(W\sO_X(K_X+S+\Theta)) \to H^d_{\m}(F^e_* W\sO_X(p^e(K_X+S+\Theta)))
    \]
    is injective, then the converse of {\rm(1)} holds.
    \end{enumerate}
\end{proposition}

\begin{proof}
    We write $D_X \coloneqq K_X+S+\Theta$ and $D_S \coloneqq K_S+\Theta|_S$.
    The Poincar\'{e} residue map
    \[
    \mathrm{Res} \colon \omega_X(S)|_S \xrightarrow{\sim} \omega_S
    \]
    induces an $\sO_X$-homomorphism
    \[
    u \colon \sO_X(K_X+S) \to \sO_S(K_S).
    \]
    Noting that the sequence 
    \[
    0 \to \sO_X(\rdown{mD_X}-S) \to \sO_X(\rdown{mD_X}) \xrightarrow{u^{\otimes m}} \sO_S(\rdown{mD_S}) \to 0
    \]
    is exact for every $m \ge 1$, we have the following exact sequence
    \[
    0 \to W \cI_S(p^lD_X) \to W \sO_X(p^lD_X) \xrightarrow{W (u^{\otimes p^l})} W \sO_S(p^lD_S) \to 0
    \]
    for every $l \ge 0$.
    
    Since we assume $H^{d-1}_{\m}(\sO_X(p^lD_X))=0$ for every $l \ge 0$, it follows from the exact sequence \eqref{eq:key-sequence-for-WnI} and \cref{prop:lim and coh} that one has 
    \[
    H^{d-1}_\m(W\cO_X(D_X))=H^{d-1}_{\m}(W\sO_X(p^eD_X))=0.
    \]
    Therefore, we have the commutative diagram whose rows are exact
\begin{equation*}
\begin{tikzcd}[column sep=0.4cm]
    0 \arrow[r] & H_\m^{d-1}(W\cO_S(D_S)) \arrow[r] \arrow[d,"F^e_S"] & H^{d}_\m(W\cI_S(D_X)) \arrow[r] \arrow[d,"F^e_{X,S}"] & H^d_\m(W\cO_X(D_X)) \arrow[r] \arrow[d,"F^e_X"] & 0 \\
    0 \arrow[r] & F^e_*H^{d-1}_\m(W\cO_S(p^eD_S)) \arrow[r] & F^e_*H^d_\m(W\cI_S(p^eD_X)) \arrow[r] & F^e_*H^d_\m(W\cO_X(p^eD_X)) \arrow[r] & 0.
\end{tikzcd}
\end{equation*}
By the snake lemma, we obtain the following exact sequence
\begin{equation}\label{eq:adjunction snake}
0 \to \Ker(F^e_S) \to \Ker(F^e_{X,S}) \to \Ker (F^e_X).
\end{equation}
We also note that by a similar argument as in the proof of \cref{prop:p-qFS equivalent def} (c) $\Leftrightarrow$ (d), the morphism $H^{d}_{\m}(\Phi^{S,e}_{X,D_X,n})$ (resp.~$H^{d-1}_{\m}(\Phi^{e}_{S,D_S,n})$) is injective for some $n \geq 1$ if and only if the image of $\Ker(F^e_{X,S})$ (resp.~$\Ker(F^e_S)$) by the morphism $\gamma$ (resp.~$\beta$) in the following commutative diagram
\[
\begin{tikzcd}
    H^{d-1}_{\m}(W\sO_S(D_S)) \arrow[r] \arrow[d,"\beta"] & H^d_{\m}(W\cI_S(D_X)) \arrow[d,"\gamma"] \\
    H^{d-1}_{\m}(\sO_S(D_S)) \arrow[r] & H^d_{\m}(\sO_X(D_X-S)).
\end{tikzcd}
\]
is zero.

The assertion in (1) now follows from the injectivity of the bottom horizontal map.
For (2), combining the assumption with \eqref{eq:adjunction snake}, we have
\[
\Ker(F^e_S) \simeq \Ker(F^e_{X,S}),
\]
which proves that $\gamma(\Ker(F^e_{X,S}))=0$ if $\beta(\Ker(F^e_S))=0$. \qedhere

\end{proof}

\subsection{Reduction step to dlt centers}

In this subsection, we use the following notation.

\begin{notation}\label{notation:two-dim-normal}
Let $(R,\m)$ be a $2$-dimensional $F$-finite Noetherian normal local domain of characteristic $p>0$ and $\Delta$ be an effective $\Q$-Weil divisor on $X=\Spec R$ such that $(X,\Delta)$ is log canonical but not klt.

Suppose that $g \colon Z \to X$ is a dlt blow-up of $(X,\Delta)$; that is,
\[
g^*(K_X+\Delta) = K_Z + g^{-1}_*\Delta + \Exc(g),
\]
and $(Z,\Delta_Z)$ is dlt with $\Delta_Z := g^{-1}_*\Delta + \Exc(g)$.
Take a log resolution $h \colon Y \to Z$ of $(Z,\Delta_Z)$ such that $h$ is an isomorphism over the simple normal crossing locus of $(Z,\Delta_Z)$.
We set $f := g \circ h \colon Y \to X$.

We write $\Delta_Y := f^*(K_X+\Delta)-K_Y$, $D_Y \coloneqq K_Y+\Delta_Y$, $S := \Delta_Y^{=1}$, and $\Delta_S := \Delta_Y^{<1}|_{S}$.
\end{notation}

\begin{lemma}\label{lem:exis-good-resol}
We use \cref{notation:two-dim-normal}.
Then the following hold:
\begin{enumerate}[label=\textup{(\arabic*)}]
    \item $h \colon Y \to Z$ induces an isomorphism $S \to S_Z \coloneqq \Delta_Z^{=1}$,
    \item $\Delta_S$ is effective, and 
    \item if $\Delta$ has standard coefficients, then so does $\Delta_S$.
\end{enumerate}
\end{lemma}

\begin{proof}
Write $S_Z = \Delta_Z^{=1}=E_1+\cdots+E_r$ with each $E_i$ prime.
Since $\Supp \Delta_Z$ is simple normal crossing near every intersection of two components of $S_Z$ and $(Z,g^{-1}_*\Delta^{<1}+E_i)$ is plt, we have
\[
S=h^{-1}_*S_Z.
\]
We also note that each $E_i$ is regular since $(Z,E_i)$ is plt (cf.~\cite{kollar13}*{Paragraph 3.35}).
Therefore, the induced morphism $S \to S_Z$ is an isomorphism.  
This proves (1).

For (2) and (3), let $B_{E_i} \coloneqq \mathrm{Diff}_{E_i}(0)$ be the $\Q$-Weil divisor on $E_i$ defined in \cite{kollar13}*{Theorem 3.36}.
It then follows from \cite{kollar13}*{Equations (4.7.1) and (4.2.10)} that we have
\[
\Delta_S = \sum_{i=1}^r \bigl( B_{E_i} + \Delta_Z^{<1}|_{E_i}\bigr),
\]
where we identify $S$ with $S_Z$.
Thus $\Delta_S$ is effective.
Moreover, if $\Delta$ has standard coefficients, then so does $B_{E_i}+\Delta_Z^{<1}|_{E_i}$ by \cite{KTTWYY1}*{Lemma~2.30(5)}.
\end{proof}

\begin{proposition}\label{prop:vanishing in dim 2}
    With the notation as in \cref{notation:two-dim-normal}, let $l \ge 0$ be an integer.
    Then the following hold:
    \begin{enumerate}[label=\textup{(\arabic*)}]
    \item We have $H^1_{\m}(\sO_Y(p^lD_Y))=0$.
    \item If we have $\{p^l\Delta\} \le \Delta$ (e.g.~$\Delta$ has standard coefficients), then one has
    \[
    R^1f_*\sO_Y(p^lD_Y-S)=0.
    \]
    \end{enumerate}
\end{proposition}

\begin{proof}
    By local duality (\cref{lem:local duality}), the assertion in (1) is equivalent to the vanishing
    \[
    R^1f_*\cO_Y(K_Y-\lfloor p^lD_Y \rfloor)=0.
    \]
    Since we have
    \[
    K_Y-\lfloor p^lD_Y \rfloor = K_Y-p^lD_Y+\{p^lD_Y\},
    \]
    and $-D_Y$ is a nef and big $\Q$-divisor over $X$, the vanishing follows from relative Kawamata-Viehweg vanishing (\cite{Tanaka18}*{Theorem 3.3}).

    For (2), we write $S_Z \coloneqq \rdown{\Delta_Z}$ and $D_Z \coloneqq K_Z+\Delta_Z$.
    Since we have $h_*\sO_Y(p^lD_Y-S) = \sO_Z(p^lD_Z-S_Z)$, it suffices to show the vanishing 
    \[
    R^1h_*\sO_Y(p^lD_Y-S) = R^1g_*\sO_Z(p^lD_Z-S_Z) =0.
    \]
    
    Since we have $\{p^lD_Z\} \le \Delta_Z$, the pair $(Z, S_Z+ \{p^lD_Z\})$ is dlt.
    It then follows from the proof of \cite{KTTWYY1}*{Proposition 2.29} that we have $R^1h_*\sO_Y(p^lD_Y-S)=0$ outside the snc locus of $(Z,S_Z)$.
    Combining this with the assumption that $h$ is an isomorphism around the snc locus of $(Z,S_Z)$, we conclude that 
    \[
    R^1h_*\sO_Y(p^lD_Y-S)=0.
    \]

    On the other hand, noting that we have
    \[
\rdown{p^lD_Z-S_Z}
= p^lD_Z-S_Z-\{p^lD_Z\}
= K_Z + g^{-1}_*(\Delta^{<1}-\{p^l\Delta\}) + (p^l-1)D_Z,
\]
it follows from relative Kawamata--Viehweg vanishing (\cite{Tanaka18}*{Theorem 3.3}) that one has $R^1 g_* \cO_Z(p^lD_Z-S_Z)=0$, as desired.
\end{proof}

\begin{lemma}\label{lem:qFs-p^lDelta}
We use  \cref{notation:two-dim-normal}.
For every integer $l \geq 0$, there exists an integer $n \geq 1$ such that the map
\[
H^2_\m(\cO_Y(p^lD_Y)) \longrightarrow H^2_\m(Q^{0,1}_{Y,p^lD_Y,n})
\]
is injective.
\end{lemma}

\begin{proof}
Fix $l \geq 0$.
Take an effective exceptional divisor $E$ on $Y$ such that $-E$ is ample.
Then there exists a rational number $\varepsilon >0$ such that
\[
\lfloor p^lD_Y+\varepsilon E \rfloor=\lfloor p^lD_Y \rfloor.
\]
Set $B:=p^lD_Y+\varepsilon E$.
Noting that $B$ is anti-ample, it follows from the proof of \cite{KTTWYY1}*{Theorem~5.13} that the map
\[
H^2_\m(\cO_Y(B)) \longrightarrow H^2_\m(Q^{0,1}_{Y,B,n})
\]
is injective for some integer $n \geq 1$.

Therefore, from the commutative diagram
\[
\begin{tikzcd}
    H^2_\m(\cO_Y(p^lD_Y)) \arrow[r] \arrow[d,equal] &
    H^2_\m(Q^{0,1}_{Y,p^lD_Y,n}) \arrow[d] \\
    H^2_\m(\cO_Y(B)) \arrow[r] &
    H^2_\m(Q^{0,1}_{Y,B,n}),
\end{tikzcd}
\]
we see that the top horizontal map
\[
H^2_\m(\cO_Y(p^lD_Y)) \to H^2_\m(Q^{0,1}_{Y,p^lD_Y,n})
\]
is injective, as desired.
\end{proof}

\begin{lemma}\label{lem:infty-qFs}
We use the same notation as in \cref{notation:two-dim-normal}.
The map 
\[
F^e \colon H^2_\m(W\cO_Y(D_Y)) \longrightarrow F^e_*H^2_\m(W\cO_Y(p^{e}D_Y))
\]
is injective for every integer $e>0$. 
\end{lemma}

\begin{proof}
Let $l \ge 0$ be an integer and we write 
\begin{align*}
K^1_{p^lD_Y,n} & \coloneqq \Ker (H^2_{\m}(W_n\sO_Y(p^lD_Y)) \xrightarrow{F} F_*H^2_{\m}(W_n\sO_Y(p^{l+1}D_Y))) \\
K^1_{p^lD_Y,\infty} & \coloneqq \Ker (H^2_{\m}(W\sO_Y(p^lD_Y)) \xrightarrow{F} F_*H^2_{\m}(W\sO_Y(p^{l+1}D_Y))) \\
& \cong \varprojlim_{n} K^1_{p^lD_Y, n}.
\end{align*}
The natural projection $K^{1}_{p^lD_Y,\infty} \to K^1_{p^lD_Y,1}$ is zero since it factors through $K^{1}_{p^lD_Y,n} \to K^1_{p^lD_Y,1}$ for a suitable $n$ by \cref{lem:qFs-p^lDelta} and the proof of \cref{prop:p-qFS equivalent def} (c) $\Leftrightarrow$ (d).
Combining this with the vanishing in \cref{prop:vanishing in dim 2} (1), the assertion follows from a similar argument as in the proof of \cref{thm:qFs-to-qF^es}.
\end{proof}

\begin{theorem}\label{thm:red-to-dlt}
We use  \cref{notation:two-dim-normal}.
If $(S,\Delta_S)$ is quasi-$F^e$-split, then $(X,\Delta)$ is purely quasi-$F^e$-split.
Furthermore, the converse also holds if $\{p^l\Delta\} \le \Delta $ for every $l \ge 1$ (e.g.~$\Delta$ has standard coefficients).
\end{theorem}

\begin{proof}
Applying \cref{prop:vanishing in dim 2} and \cref{lem:infty-qFs}, the assertion follows from \cref{prop:birational transformation rule} and \cref{prop:adj-inv}.
\end{proof}

\subsection{Quasi-\texorpdfstring{$F^e$}{Fe}-splitting of dlt centers}
In this subsection, we give a sufficient condition for a one-dimensional projective pair $(S,\Delta_S)$ with $K_S+\Delta_S \sim_{\Q} 0$ to be quasi-$F$-split (\cref{prop:van-B_0}).

\begin{lemma}\label{lem:van-B_0-irr}
Let $S$ be a smooth projective curve over a perfect field $k$ and $D$ a $\Q$-Weil divisor on $S$ such that there exists an integer $e \geq 1$ with $(p^e-1)D \sim 0$.
Then we have
\[
\varprojlim_{n} H^0(S, B_{S,D,n})=0,
\]
where we write
\[
B_{S,D,n}:={\rm Coker}\bigl(W_n\cO_S(D) \xrightarrow{F} F_*W_n\cO_S(pD)\bigr)
\]
for every integer $n \geq 1$.
\end{lemma}
\begin{proof}
Let $E:=\rup{\{D\}}$.
By the proof of \cite{KTTWYY1}*{Theorem~5.13}, we have the exact sequence
\begin{equation}\label{eq:ex-B-Z-Omega}
    0 \to B_{S,D,n} \to F^n_*\Omega^1(\log E)(p^nD) \xrightarrow{C^n} \Omega^1(\log E)(D) \to 0,
\end{equation}
where we note that $Z_n\Omega^1(\log E)(D) \simeq F^n_*\Omega^1(\log E)(p^nD)$ since $S$ is one-dimensional.

Set $V_n := H^0(S, \Omega^1(\log E)(p^nD))$ and $V := V_0$.  
Since $(p^e-1)D \sim 0$, we have $V_{ne}\simeq V$ for every $n \geq 1$.  
Thus, $C^e$ induces a $p^{-e}$-linear map
\[
\varphi \colon F^e_*V \longrightarrow V.
\]

As $V$ is finite-dimensional over $k$, there exists $n_0 \geq 1$ such that
\[
{\rm Im}(\varphi^n) = {\rm Im}(\varphi^{n_0}) \quad \text{for all } n \geq n_0.
\]
Set $V' := {\rm Im}(\varphi^{n_0})$. Then $\varphi$ restricts to a surjection
\[
\varphi' \colon F^e_*V' \twoheadrightarrow V',
\]
which is in fact an isomorphism since $\dim F^e_*V' = \dim V'$.

Now for $n \geq n_0$, we obtain
\[
\varphi^n({\rm Ker}(\varphi^{2n})) = \varphi^n\bigl({\rm Ker}(F^{2en}_*V \to F^{en}_*V')\bigr) = 0.
\]
In particular, the transition maps in the inverse system $\{F^n_*V_n\}_n$ become injective in the limit, so that
\[
\varprojlim_{n} F^n_*V_n \hookrightarrow V.
\]

Finally, from the exact sequence \eqref{eq:ex-B-Z-Omega}, we deduce an exact sequence
\[
0 \to \varprojlim_{n} H^0(S, B_{S,D,n}) \to \varprojlim_{n} F^n_*V_n \to V.
\]
Since the right-hand map is injective, we have $\varprojlim_{} H^0(B_{S,D,n})=0$, as claimed.
\end{proof}
 
\begin{prop}\label{prop:van-B_0}
    Let $S$ be a purely one-dimensional reduced projective scheme over an $F$-finite infinite field $k$ of characteristic $p>0$ and $\Delta$ be an effective Mumford $\Q$-divisor on $S$.
    We denote by $\Delta_{\overline{k}}$ the flat pullback of $\Delta$ to $S \times_{k} \Spec \overline{k}$.
    We further assume that the following conditions hold:
    \begin{enumerate}[label=\textup{(\roman*)}]
    \item $\rdown{\Delta_{\overline{k}}}=0$,
    \item $(p^e-1)(K_S+\Delta) \sim 0$ for some $e>0$, 
    \item $S$ is geometrically snc over $k$, that is, all irreducible components $S_i$ and all scheme-theoretic intersections $S_i \cap S_j$ are smooth over $k$, and
    \item $S$ has only hypersurface singularities.
    \end{enumerate}
    Then the pair $(S,\Delta)$ is quasi-$F^{\infty}$-split.
\end{prop}

\begin{proof}
    By \cref{prop:base change descent}, after replacing $S$ by $S \times_k \Spec \overline{k}$, we may assume that $k$ is algebraically closed.
    Let $K_S$ be a canonical Mumford divisor on $S$.
    Let $S_1,\ldots,S_r$ be the irreducible components of $S$, and put $D_S:=K_S+\Delta$.
For integers $l \geq 0$ and $n \geq 1$, we define the $W_n\cO_S$-module
\[
B_{S,p^lD_S,n}:={\rm Coker}\bigl(W_n\cO_S(p^lD_S) \xrightarrow{F} F_*W_n\cO_S(p^{l+1}D_S)\bigr).
\]
For each $1 \leq i \leq r$, we also set $D_{S_i}:=(D_S)|_{S_i}$.

\begin{claim}\label{claim:comp-B}
We have
\[
B_{S,p^lD_S,n} \simeq \bigoplus_{i=1}^r B_{S_i,p^lD_{S_i},n}.
\]
\end{claim}

\begin{claimproof}
By the assumptions (iii) and (iv), we have the commutative diagram in which each horizontal sequence is exact for every $n \geq 1$:
\begin{equation}\label{eq:ex-diag-comp}
\begin{tikzcd}[column sep=0.5cm]
    0 \arrow[r] & W_n\cO_S(p^lD_S) \arrow[r] \arrow[d,"F"] & \bigoplus_{i=1}^r W_n\cO_{S_i}(p^lD_{S_i}) \arrow[r] \arrow[d,"F"] & \bigoplus_{i>j} W_n\cO_{S_i \cap S_j} \arrow[r] \arrow[d,"F"] & 0 \\
    0 \arrow[r] & F_*W_n\cO_S(p^{l+1}D_S) \arrow[r] & \bigoplus_{i=1}^r F_*W_n\cO_{S_i}(p^{l+1}D_{S_i}) \arrow[r] & \bigoplus_{i>j} F_*W_n\cO_{S_i \cap S_j} \arrow[r] & 0,
\end{tikzcd}
\end{equation}
Since each $S_i \cap S_j$ is a finite disjoint union of spectra of perfect fields, the right vertical map in \eqref{eq:ex-diag-comp} is an isomorphism.
Applying the snake lemma to \eqref{eq:ex-diag-comp} yields the desired isomorphism.
\end{claimproof}

By \cref{claim:comp-B} and \cref{lem:van-B_0-irr}, it follows that
\[
\varprojlim_{n } H^0(S, B_{S,p^lD_S,n})=0
\]
for every $l \geq 0$.
In particular, the Frobenius morphism
\[
F \colon H^1(S, W\cO_S(p^lD_S)) \hookrightarrow F_*H^1(S, W\cO_S(p^{l+1}D_S))
\]
is injective.
Therefore, for every $e \geq 1$, the map
\[
F^e \colon H^1(S, W\cO_S(D_S)) \hookrightarrow F^e_*H^1(S, W\cO_S(p^eD_S))
\]
is injective, and hence $(S,\Delta)$ is quasi-$F^e$-split by \cref{prop:infty-qFs}, as desired.
\end{proof}

\subsection{Quasi-\texorpdfstring{$F^e$}{Fe}-splitting of two-dimensional log canonical pairs}

\begin{theorem}\label{thm:qFs-lc-Z_p-index}
With the notation as in \cref{notation:two-dim-normal}, we assume that $R/\m$ is perfect.
If the Cartier index of $K_X+\Delta$ is not divisible by $p$, then $(X,\Delta)$ is purely quasi-$F^{\infty}$-split.
\end{theorem}

\begin{proof}
We may assume that $R$ is complete by \cref{completion}.
We denote by $R_{\infty}$ the localization of $R \otimes_{\F_p} \overline{\F}_p$ by a maximal ideal $\n$ of $R \otimes_{\F_p} \overline{\F}_p$ and by $R_l$ the localization of $R \otimes_{\F_p} \F_{p^l}$ by the maximal ideal $\n \cap (R \otimes_{\F_p} \F_{p^l})$.
Since $R_{\infty}$ is a direct limit of $\{R_l\}_l$, after replacing $R$ by $R_{\infty}$ and $X$ by $X \times_R \Spec R_{\infty}$, we may reduce to the case where $R/\m$ is infinite (\cref{cor:asc-des ind-etale}).

By \cref{thm:red-to-dlt}, it suffices to show that the pair $(S,\Delta_S)$ is quasi-$F^{\infty}$-split.
Let $S_1, \dots, S_r$ be the irreducible components of $S$.
It follows from a similar argument as in the proof of \cref{claim:comp-B} that one has
\[
B_{S,p^lD_S,n} \simeq \bigoplus_{i=1}^r B_{S_i,p^lD_{S_i},n},
\]
where we use the same notations as in the proof of \Cref{prop:van-B_0}.

Since $B_{S_i,p^lD_{S_i},n}$ is a torsion-free $\sO_{S_i}$-module, if $S_i$ is not an exceptional divisor then 
\[
H^0_\m(B_{S_i,p^lD_{S_i},n})=H^0_{f^{-1}(\m) \cap S_i} (S_i, B_{S_i,p^lD_{S_i},n}) = 0.
\]
On the other hand, if $S_i$ is an exceptional divisor, then it follows from \cref{lem:van-B_0-irr} that we have
\[
\varprojlim_{n} H^0_\m(B_{S_i,p^lD_{S_i},n})=\varprojlim_{n} H^0(S_i, B_{S_i,p^lD_{S_i},n}) =0.
\]
Therefore, we have
\[
\varprojlim_{n} H^0_\m(B_{S,p^lD_S,n})=0
\]
for every $l \geq 0$.
It then follows from a similar argument as in \cref{prop:van-B_0} that the pair $(S,\Delta_S)$ is quasi-$F^{\infty}$-split, as desired.
\end{proof}

\begin{remark}\label{2-dimklt=>qFregular}
The conclusion of Theorem \ref{thm:qFs-lc-Z_p-index} also holds when $(X,\Delta)$ is klt, without any restriction on the Cartier index of $K_X+\Delta$. 
Indeed, \cite{KTTWYY3}*{Theorem 7.9} shows that a 2-dimensional affine klt pair $(X=\Spec R,\Delta)$ is quasi-$F$-regular when $R$ is essentially of finite type over a perfect field. 
As explained in Remark \ref{rem:klt to qFR} below, this result extends to $F$-finite Noetherian normal local rings. Thus, in the klt case, $(X,\Delta)$ is quasi-$F$-regular, and hence purely quasi-$F^\infty$-split. 
\end{remark}

\begin{corollary}\label{cor:mod-p-reduction}
Let $(S,\n)$ be a two-dimensional normal domain essentially of finite type over a field of characteristic zero.
Let $\Delta$ be an effective $\Q$-Weil divisor on $\Spec S$ such that $(S,\Delta)$ is log canonical.
Then for almost all prime ideals $\fp$, the mod $\fp$ reduction of $(S,\Delta)$ is purely quasi-$F^e$-split for every integer $e \geq 1$.
\end{corollary}

\begin{proof}
Denote by $r$ the Cartier index of $K_S+\Delta$.
If $r$ is not divisible by the characteristic of $\kappa(\fp)$, then by \cref{thm:qFs-lc-Z_p-index} and Remark \ref{2-dimklt=>qFregular}, the mod $\fp$ reduction of $(S,\Delta)$ is purely quasi-$F^e$-split for every integer $e \geq 1$.
\end{proof}

\begin{corollary}\label{cor:qFs-index}
With the notation as in \cref{thm:qFs-lc-Z_p-index}, we further assume that $\Delta$ has standard coefficients and $\rdown{\Delta} =0$.
Then the following conditions are equivalent:
\begin{enumerate}[label=\textup{(\alph*)}]
\item $(X,\Delta)$ is quasi-$F$-split.
\item $(X,\Delta)$ is quasi-$F^{\infty}$-split.
\item The Cartier index of $K_X+\Delta$ is not divisible by $p$.
\end{enumerate}
\end{corollary}

\begin{proof}
(a) $\Leftrightarrow$ (b) follows from \cref{cor:qFs-to-qF^eS}, and (c) $\Rightarrow$ (b) follows from \cref{thm:qFs-lc-Z_p-index}.
We prove (b) $\Rightarrow$ (c).
We assume that the Cartier index of $K_X+\Delta$ is divisible by $p$.
By \cref{prop:index-one-cov}, we may assume that the Cartier index of $K_X+\Delta$ is $p^a$ for some integer $a \geq 1$.
By \cref{thm:red-to-dlt}, it suffices to show that $(S,\Delta_S)$ is not quasi-$F^a$-split.

We first assume that $\Delta_S=0$.
Then a dlt blow-up $g \colon Z \to X$ of $(X,\Delta)$ is a log resolution of $(X,\Delta)$ and 
\[
g^*(K_X+\Delta)=K_Z+S 
\]
is Cartier.
It follows from \cref{lem:num lc surface} that $K_X+\Delta$ itself is Cartier, contradicting the assumption that the Cartier index is $p^a$ with $a \geq 1$.
Thus, we must have $\Delta_S \neq 0$.

For every integer $l \geq a$, we compute
\[
L:=(1-p^l)K_S-\rdown{p^l\Delta_S}=(1-p^l)K_S-p^l\Delta_S \sim_{\Q} -\Delta_S.
\]
Since $L\sim_{\Q}-\Delta_S$, there exists an integer $m>0$ such that
\[
\sO_S(mL)\simeq\sO_S(-m\Delta_S).
\]
Because $\Delta_S\neq0$ and $H^0(S,\sO_S)$ is a field,
we have $H^0(S,\sO_S(-m\Delta_S))=0$.
If $s\in H^0(S,\sO_S(L))$ were nonzero,
then $s^{\otimes m}$ would be nonzero since $S$ is reduced,
a contradiction.
Combining this with \cref{proposition:no-gl-section}, $(S,\Delta_S)$ is not quasi-$F^a$-split, as required.
\end{proof}

\begin{example}
For every prime number $p>0$, there exists a two-dimensional log canonical pair which is not quasi-$F$-split.
Let $R=k[[x,y]]$, where $k$ is an algebraically closed field of characteristic $p>0$, and let $f \colon Y \to \Spec R$ be the blow-up at the origin $(0,0)$.
Denote by $E$ the exceptional prime divisor of $f$.
Then there exist prime divisors $B_1,\ldots,B_{2p}$ on $X=\Spec R$ such that $f^*B_i=f^{-1}_*B_i+E$ for every $1 \leq i \leq 2p$, and such that 
\[
f^{-1}_*B_1+\cdots+f^{-1}_*B_{2p}+E
\]
has simple normal crossing support.
Set
\[
\Delta:=\frac{1}{p}(B_1+\cdots+B_{2p}).
\]
Then we have
\[
f^*(K_X+\Delta)=K_Y-E+f^{-1}_*\Delta+2E
   =K_Y+E+f^{-1}_*\Delta=:K_Y+\Delta_Y.
\]
Thus $\Delta_Y^{=1}=S$ is a projective line, and 
\[
\Delta_S=\frac{1}{p}(P_1+\cdots+P_{2p})
\]
for some points $P_1,\ldots,P_{2p}$ on $S$.
Since
\[
H^0\bigl(S, \cO_S((1-p^l)K_S-\rdown{p^l\Delta_S})\bigr)=0
\quad \text{for every integer } l \geq 1,
\]
the pair $(S,\Delta_S)$ is not quasi-$F$-split by \cref{proposition:no-gl-section}.
Moreover, since $p\Delta_Y$ is a $\Z$-Weil divisor, the pair $(X,\Delta)$ is not quasi-$F$-split by \cref{thm:red-to-dlt}.
\end{example}

\subsection{Classification of two-dimensional quasi-\texorpdfstring{$F$}{F}-split singularities}
As a special case of \cref{cor:qFs-index}, we consider the case of $\Delta=0$ and give a classification of quasi-$F$-split surface singularities (\cref{thm:class-qFs} and Theorem \ref{thm:class-qFs imperfect}).

\begin{remark}\label{rem:graph of lc}
Let $(R,\m)$ be a $2$-dimensional excellent normal local domain that is log canonical but not klt.
\begin{enumerate}[label=\textup{(\arabic*)}]
\item If $R$ is a rational singularity, then the dual graph (see \cite{Sato25}*{Definition 2.13} for the definition) of its minimal resolution is one of the following:
\begin{enumerate}[label=\textup{(\alph*)}]
\item Star-shaped of type $(2,3,6)$, $(3,3,3)$, or $(2,4,4)$.
        
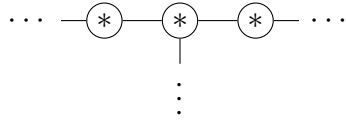
\[
\begin{tikzpicture}
\draw node (ldot) at (1,0){$\cdots$};
\node[draw, shape=circle, inner sep=1.8pt] (E1 alpha) at (2,0){$*$};
\node[draw, shape=circle, inner sep=1.8pt] (C) at (3,0){$*$};
\node[draw, shape=circle, inner sep=1.8pt] (E1 Beta) at (4,0){$*$};
\draw node (rdot) at (5,0){$\cdots$};
\draw node (ddot) at (3,-1){\rotatebox{90}{$\cdots$}};

\draw (ldot)--(E1 alpha);
\draw (E1 alpha)--(C);
\draw (C)--(E1 Beta);
\draw (E1 Beta)--(rdot);
\draw (C)--(ddot);
\end{tikzpicture}
        \]
        
\item ${}_*\tilde{D}_{n+3}$ ($n \ge 1$)
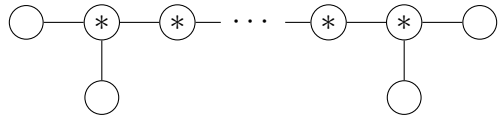
\[
\begin{tikzpicture}
\node[draw, shape=circle, inner sep=4.5pt] (LU) at (0,0){};
\node[draw, shape=circle, inner sep=1.8pt] (C1) at (1,0){$*$};
\node[draw, shape=circle, inner sep=4.5pt] (LD) at (1,-1){};
\node[draw, shape=circle, inner sep=1.8pt] (C2) at (2,0){$*$};
\draw node (Dot) at (3,0){$\cdots$};
\node[draw, shape=circle, inner sep=1.8pt] (C3) at (4,0){$*$};
\node[draw, shape=circle, inner sep=1.8pt] (C4) at (5,0){$*$};
\node[draw, shape=circle, inner sep=4.5pt] (RU) at (6,0){};
\node[draw, shape=circle, inner sep=4.5pt] (RD) at (5,-1){};

\draw (LU)--(C1);
\draw (LD)--(C1);
\draw (C1)--(C2);
\draw (C2)--(Dot);
\draw (Dot)--(C3);
\draw (C3)--(C4);
\draw (RU)--(C4);
\draw (RD)--(C4);
\end{tikzpicture}
\]

\item[\textup{(a')}] Twisted star-shaped of type $(3,3,3)$ or $(2,4,4)$.
\[
\begin{tikzpicture}
\node[draw, shape=circle, inner sep=1.8pt] (L0) at (-2,0){$*$};
\draw[shift={(0,0.4)}] (-2,0) node{\tiny $1$};
\draw node (LDot) at (-1,0){$\cdots$};
\node[draw, shape=circle, inner sep=1.8pt] (L1) at (0,0){$*$};
\draw[shift={(0,0.4)}] (0,0) node{\tiny $1$};
\node[draw, shape=circle, inner sep=1.8pt] (L2) at (1,0){$*$};
\draw[shift={(0,0.4)}] (1,0) node{\tiny $2$};
\draw node (Dot) at (2,0){$\cdots$};
\node[draw, shape=circle, inner sep=1.8pt] (L3) at (3,0){$*$};
\draw[shift={(0,0.4)}] (3,0) node{\tiny $2$};

\draw (LDot.180)--(L0.0);
\draw (LDot.0)--(L1.180);
\draw (L1.10)--(L2.170);
\draw (L1.-10)--(L2.190);
\draw (L2.10)--(Dot.175);
\draw (L2.-10)--(Dot.185);
\draw (Dot.5)--(L3.170);
\draw (Dot.-5)--(L3.190);

\node[draw, shape=circle, inner sep=1.8pt] (R0) at (5,0){$*$};
\draw[shift={(0,0.4)}] (5,0) node{\tiny $1$};
\draw node (RDot) at (6,0){$\cdots$};
\node[draw, shape=circle, inner sep=1.8pt] (R1) at (7,0){$*$};
\draw[shift={(0,0.4)}] (7,0) node{\tiny $1$};
\node[draw, shape=circle, inner sep=1.8pt] (R2) at (8,0){$*$};
\draw[shift={(0,0.4)}] (8,0) node{\tiny $3$};
\draw node (RRDot) at (9,0){$\cdots$};
\node[draw, shape=circle, inner sep=1.8pt] (R3) at (10,0){$*$};
\draw[shift={(0,0.4)}] (10,0) node{\tiny $3$};

\draw (RDot.180)--(R0.0);
\draw (RDot.0)--(R1.180);
\draw (R1.0)--(R2.180);
\draw (R1.10)--(R2.170);
\draw (R1.-10)--(R2.190);
\draw (R2.0)--(RRDot.180);
\draw (R2.10)--(RRDot.175);
\draw (R2.-10)--(RRDot.185);
\draw (RRDot.0)--(R3.180);
\draw (RRDot.5)--(R3.170);
\draw (RRDot.-5)--(R3.190);

\end{tikzpicture}
\]
\item[\textup{(b')}] Twisted ${}_*\widetilde{D}_{n+3}$ ($n \ge 1$).
\[
\begin{tikzpicture}
\node[draw, shape=circle, inner sep=1.8pt] (L1) at (0,0){$*$};
\draw[shift={(0,0.4)}] (0,0) node{\tiny $1$};
\node[draw, shape=circle, inner sep=1.8pt] (L2) at (1,0){$*$};
\draw[shift={(0,0.4)}] (1,0) node{\tiny $2$};
\draw node (Dot) at (2,0){$\cdots$};
\node[draw, shape=circle, inner sep=1.8pt] (L3) at (3,0){$*$};
\draw[shift={(0,0.4)}] (3,0) node{\tiny $2$};
\node[draw, shape=circle, inner sep=1.8pt] (L4) at (4,0){$*$};
\draw[shift={(0.2,0.4)}] (4,0) node{\tiny $2$};
\node[draw, shape=circle, inner sep=4.5pt] (U1) at (4,1){};
\draw[shift={(0,0.4)}] (4,1) node{\tiny $2$};
\node[draw, shape=circle, inner sep=4.5pt] (D1) at (5,0){};
\draw[shift={(0,0.4)}] (5,0) node{\tiny $2$};

\draw (L1.10)--(L2.170);
\draw (L1.-10)--(L2.190);
\draw (L2.10)--(Dot.175);
\draw (L2.-10)--(Dot.185);
\draw (Dot.5)--(L3.170);
\draw (Dot.-5)--(L3.190);
\draw (L3.10)--(L4.170);
\draw (L3.-10)--(L4.190);
\draw (L4.100)--(U1.-100);
\draw (L4.80)--(U1.-80);
\draw (L4.10)--(D1.170);
\draw (L4.-10)--(D1.190);
\end{tikzpicture}
,\quad \textup{ etc.}
\]
\end{enumerate}
See \cite{Sato25}*{Figure 2, 3} and \cite{Sato25}*{Theorem A.3} for more details.
\item Let $r$ be the Gorenstein index of $R$, that is, the minimal integer $r>0$ such that $rK_R$ is Cartier.
\begin{enumerate}[label=\textup{(\roman*)}]
\item If $R$ is not a rational singularity, then $r=1$ (\cref{lem:num lc surface} (1)).
\item If the dual graph of $R$ is star-shaped of type $(2,3,6)$ (resp.~$(3,3,3)$, $(2,4,4)$), then it follows from \cref{lem:num lc surface} (2) and the proof of \cite{kollar13}*{Theorem 3.38} that $r=6$ (resp.~$r=3$, $r=4$).
\item If the dual graph is ${}_*\tilde{D}_{n+3}$, then for a minimal resolution $f \colon Y \to X$, we have 
\[
K_Y-f^*K_X = -\frac{1}{2} (C_1+\cdots + C_4) - (E_1+\cdots +E_n),
\]
where $C_i$ are the leaves of the dual graph and $E_i$ are others.
Therefore, one has $r=2$ in this case.
\item If the dual graph is twisted star-shaped of type $(3,3,3)$, $(2,4,4)$, and twisted ${}_*\tilde{D}_{n+3}$, then it follows from \cite{kollar13}*{Paragraph 3.41} that $r=3, 4$ and $2$, respectively.
\end{enumerate}
\end{enumerate}
\end{remark}

\begin{theorem}\label{thm:class-qFs}
Let $(R,\m)$ be a $2$-dimensional $F$-finite Noetherian normal local domain of characteristic $p>0$ with $R/\m$ perfect.
Then the following conditions are equivalent:
\begin{enumerate}[label=\textup{(\alph*)}]
\item $R$ is quasi-$F$-split.
\item $R$ is quasi-$F^{\infty}$-split.
\item $R$ is log canonical and satisfies one of the following conditions:
\begin{enumerate}[label=\textup{(\roman*)}]
    \item $R$ has log terminal singularities,
    \item $R$ is not a rational singularity,
    \item $p \neq 2,3$ and the dual graph is star-shaped of type $(2,3,6)$,
    \item $p \neq 3$ and the dual graph is star-shaped or twisted star-shaped of type $(3,3,3)$, or
    \item $p \neq 2$ and the dual graph is ${}_*\widetilde{D}_{n+3}$, twisted ${}_*\widetilde{D}_{n+3}$ ($n \geq 1$), or star-shaped or twisted star-shaped of type $(2,4,4)$.
\end{enumerate}
\end{enumerate}
\end{theorem}

\begin{proof}
The equivalence (a) $\Leftrightarrow$ (b) follows from \cref{cor:qFs-to-qF^eS}.
By \cref{thm:qFs-to-lc-norm}, in any case we may assume that $R$ is log canonical.
If $R$ has log terminal singularities, then $R$ is quasi-$F$-split by \cite{KTTWYY1}*{Theorem~C} (see also \cref{rem:klt to qFR} below).
If $R$ is not log terminal, then the assertion follows from \cref{cor:qFs-index} and \cref{rem:graph of lc} (2).
\end{proof}

\begin{remark}\label{rem:klt to qFR}
    In \cite{KTTWYY3}*{Theorem~C}, the ring $R$ is assumed to be essentially of finite type over a perfect field, an assumption inherited from the corresponding hypothesis in \cite{TWY}*{Theorem~3.27}. 
    However, this assumption can be removed when $R$ is $F$-finite. 
    Indeed, \cite{TWY}*{Theorem~3.27} remains valid when $\Omega$ is the sheaf of differentials over $\mathbb{F}_p$, and hence \cite{KTTWYY3}*{Theorem~C} also holds for $F$-finite rings. 
    To see this, we first note that if $A$ is an $F$-finite Noetherian regular local ring, then, by \cite{KTTWYY3}*{Claim~9.4}, there exists a differential basis of $A$ over $\mathbb{F}_p$ containing any given regular system of parameters.
    Consequently, if $Y$ is an $F$-finite Noetherian regular scheme and $E=\sum_i E_i$ is an SNC divisor on $Y$, then the sheaf of absolute logarithmic differentials $\Omega_{Y/\mathbb{F}_p}(\log E)$ can be defined in the same way as in the usual setting. 
    Moreover, the absolute logarithmic de Rham complex is obtained, just as in the standard case, by base change from a suitable complex over $\mathbb{F}_p$; see \cite{esnault_viehweg}*{Claim~9.15}. 
    It follows that the Cartier map is an isomorphism (see also \cite{BLM}*{Theorem~9.5.6}).
    Thus, the proof of \cite{TWY}*{Theorem~3.27} carries over to the setting where $\Omega$ denotes the sheaf of absolute differentials. We leave the details to the reader.
\end{remark}

\begin{thm}\label{thm:class-qFs imperfect}
Let $(R,\m)$ be a $2$-dimensional $F$-finite Noetherian normal local domain of characteristic $p>3$.
Then the following conditions are equivalent:
\begin{enumerate}[label=\textup{(\alph*)}]
\item $R$ is quasi-$F$-split.
\item $R$ is quasi-$F^{\infty}$-split.
\item $R$ is log canonical.
\end{enumerate}
\end{thm}

\begin{proof}
    If $R/\m$ is perfect, then the assertion follows from \cref{thm:class-qFs}.
    We assume that $R/\m$ is imperfect, and in particular it is infinite.
    (a) $\Rightarrow$ (b) follows from \cref{cor:qFs-to-qF^eS} and (b) $\Rightarrow$ (c) follows from \cref{thm:qFs-to-lc-norm}.
    
    We prove the implication (c) $\Rightarrow$ (a). 
    If $R$ is log terminal, then the assertion follows from \cite{KTTWYY1}*{Theorem~C} (see also \cref{rem:klt to qFR}).
    From now on, we assume that $R$ is log canonical but not log terminal.
    By \cite{Sato25}*{Corollary 4.15}, there is an \'etale finite local homomorphism $(R,\m) \hookrightarrow (R',\m')$ such that the exceptional divisor of a minimal resolution of $\Spec R'$ is geometrically snc over $R'/\m'$.
    We note that by the proof of \cite{Sato25}*{Corollary 4.15} (cf.~\cite{Sato25}*{Lemma 4.3}), we have $R' = R[t]/(G(t))$ for some monic polynomial $G(t) \in R[t]$, and in particular the inclusion $R \hookrightarrow R'$ splits.
    Combining this with \cref{etale-cov}, we may replace $R$ by $R'$.
    Then the assertion follows from \cref{rem:graph of lc} (2), \cref{thm:red-to-dlt} and \cref{prop:van-B_0}.
\end{proof}

\subsection{Non-normal case}
Noting that every $\Q$-divisor on a two-dimensional normal scheme is numerically $\Q$-Cartier, it follows from \cref{cor:qFS-to-lc-nonnormal} that a two-dimensional purely quasi-$F^{\infty}$-split pair is numerically slc.
In this subsection, we consider the converse implication (\cref{prop:slc F-pure} and \cref{memo:Whitney umbrella}).

We recall that an $F$-finite ring $R$ is $F$-pure if the Frobenius morphism $F \colon R \to F_*R$ splits.
We note that $R$ is $F$-pure if and only if it is $1$-quasi-$F$-split.

For a reduced scheme $X$, we denote by $X^n$ the normalization of $X$ and by $\mathfrak{c}_X \subseteq \sO_X$ the conductor ideal.
We set $C_X \subseteq X$ and $C_{X^n} \subseteq X^n$ to be the subschemes defined by $\mathfrak{c}_X$.
We say that $X$ has \emph{hereditary surjective trace} if there is some irreducible component $C_i$ of $(C_X)_{\mathrm{red}}$ dominated by an irreducible component $B_i$ of $C_{X^n}$ such that the trace map 
\[
\mathrm{Tr} \colon \sO_{B_i^n} \to \sO_{C_i^n}
\]
is surjective, and $C_i$ has hereditary surjective trace (cf.~\cite{MS12}*{Definition 3.5}).
We note that a normal scheme has hereditary surjective trace.

\begin{example}\label{eg:hereditary surj}
If $X=\Spec R$ for a one-dimensional excellent reduced local ring $(R,\m)$ with $R/\m$ perfect, then $X$ has hereditary surjective trace.
\end{example}

\begin{prop}\label{prop:slc F-pure}
    Let $(R,\m)$ be a purely $2$-dimensional $F$-finite Noetherian reduced semi-normal local ring which satisfies $\Serre{2}$ and $\Gorenstein$-conditions.
    Take a canonical Mumford divisor (or more generally, a canonical AC divisor in the sense of \cite{ST23}*{Subsection A.2}) $K_R$ on $\Spec R$.
    We further assume that 
    \begin{enumerate}[label=\textup{(\roman*)}]
    \item $R$ is of characteristic $p>2$,
    \item $R/\m$ is algebraically closed,
    \item $R$ is not normal, and
    \item $(p^e-1)K_R$ is Cartier for some $e>0$.
    \end{enumerate}
    Then $R$ is slc if and only if it is $F$-pure.
\end{prop}

\begin{proof}
    We first prove that $X = \Spec R$ has hereditary surjective trace.
    Since $R$ is seminormal and $\mathfrak{c}_X$ satisfies $\Serre{2}$ condition, $C_X$ is a reduced divisor (\cite{Tra70}*{Lemma 1.3}).
    We fix irreducible components $C_i$ of $C_X$ and $B_i$ of $C_{X^n}$ such that $B_i$ dominates $C_i$.
    Then it follows from \cite{HS}*{Theorem 12.2.2} that the degree of finite morphism $B_i \onto C_i$ is at most two.
    Since we assume $p >2$, the trace map 
    \[
    \mathrm{Tr} \colon \sO_{B_i^n} \to \sO_{C_i^n}
    \]
    is surjective.
    Combining this with \cref{eg:hereditary surj}, we conclude that $X$ has hereditary surjective trace, as claimed.

    It then follows from \cite{MS12}*{Corollary 4.3} that $X$ is $F$-pure if and only if so is the pair $(X^n, C_{X^n})$.
    On the other hand, it follows from \cite{hw02}*{Theorem 4.5} that $(X^n, C_{X^n})$ is $F$-pure if and only if it is log canonical, which completes the proof.
\end{proof}

\begin{example}\label{memo:Whitney umbrella}
    The proposition above does not hold without assumption (i), as shown by the following counterexample:
    Let $k$ be an algebraically closed field in characteristic $2$.
    We consider the Whitney umbrella
    \[
    R \coloneqq k[x,y,z]/(x^2+y^2z),
    \]
    which is a Gorenstein semi-normal $2$-dimensional domain.
    Since the normalization $R^n$ is
    \[
    R^n = R\left[ x/y\right] \cong k \left[y, x/y \right]
    \]
    and the conductor divisor $C$ is the line $Z(y)$, the pair $(R^n, C)$ is simple normal crossing.
    Therefore, $R$ has slc singularities.

    On the other hand, it follows from Fedder's criterion (\cite{KTY22}*{Corollary B}) that $R$ is not quasi-$F$-split.
    In fact, if we set $I_n \subseteq S =k[x,y,z]$ as in \cite{KTY22}*{Theorem A}, then we can see by induction on $n$ that $I_n \subseteq (x^2, y^2)$.
\end{example}

\subsection{Log canonical thresholds versus quasi-\texorpdfstring{$F$}{F}-splitting thresholds}

Let $X$ be an excellent normal integral scheme with a dualizing complex $\omega_X^{\bullet}$.
We further assume that $K_X$ is numerically $\Q$-Cartier.
The \emph{log canonical threshold} of a numerically $\Q$-Cartier effective $\Q$-divisor $D$ with respect to $X$ is defined by
\[
\lct(X;D):=\sup\{a \in \Q_{\geq 0} \mid (X,aD)\ \text{is numerically log canonical}\}.
\]

\begin{definition}\label{defn:threshold}
Let $R$ be a Noetherian $F$-finite normal ring of positive characteristic, and let $D$ be an effective $\Q$-Weil divisor on $X := \Spec R$.
The \emph{quasi-$F$-pure threshold} of $D$ with respect to $X$ is
\[
\qfpt(X;D):=\sup\{a \in \Q_{\geq 0} \mid (X,aD)\ \text{is quasi-$F^{\infty}$-split}\}.
\]
\end{definition}

\begin{theorem}\label{lct-vs-qfpt}
Let $(R,\m)$ be a Noetherian $F$-finite normal local domain of positive characteristic, and let $D$ be an effective $\Q$-Weil divisor on $X := \Spec R$.
\begin{enumerate}[label=\textup{(\arabic*)}]
    \item If $K_R$ and $D$ are numerically $\Q$-Cartier, then $\qfpt(X;D) \leq \lct(X;D)$.
    \item Assume that $\dim R=2$, $R/\m$ is perfect, and $R$ is quasi-$F$-split.
    Then $\qfpt(X;D)=\lct(X;D)$.
\end{enumerate}
\end{theorem}

\begin{proof}
Assertion (1) follows directly from
\cref{thm:qFs-to-lc-norm}.

We prove (2).
By \cref{thm:class-qFs}, $X$ is log canonical and
quasi-$F^{\infty}$-split.
Since every $\Q$-Weil divisor on a normal excellent
surface is numerically $\Q$-Cartier, assertion (1) gives
\[
0 \leq \qfpt(X;D) \leq \lct(X;D).
\]
If $D=0$, then both thresholds are $+\infty$.
We may therefore assume that $D\neq 0$.

Suppose, for a contradiction, that
\[
\qfpt(X;D)<\lct(X;D).
\]
In particular, $\lct(X;D)>0$, so there exists
$b\in\Q_{>0}$ such that $(X,bD)$ is numerically
log canonical.
Since $bD\neq 0$, \cref{lem:num lc surface}(1)
implies that $X$ has a rational singularity.
It follows from \cref{lem:num lc surface}(2)
that $X$ is $\Q$-factorial.
Thus numerical log canonicity and log canonicity
coincide for the pairs $(X,tD)$ with $t\in\Q_{\geq 0}$.

First suppose that $X$ is klt.
Choose $a\in\Q$ such that
\[
\qfpt(X;D)<a<\lct(X;D).
\]
There exists $b\in\Q$ with $a<b$ such that
$(X,bD)$ is log canonical.
Since $(X,0)$ is klt and
\[
aD=\frac{a}{b}(bD)+\left(1-\frac{a}{b}\right)0,
\]
the pair $(X,aD)$ is klt.
By \cite{KTTWYY3}*{Theorem~7.9},
$(X,aD)$ is quasi-$F$-regular, and hence
quasi-$F^{\infty}$-split.
This contradicts $a>\qfpt(X;D)$.

Now suppose that $X$ is not klt.
Applying \cref{cor:qFs-index} with boundary zero,
we see that the Cartier index $r$ of $K_X$
is not divisible by $p$.
Choose an integer $M>0$ such that $MD$ is Cartier.
Since
\[
\left\{
M\frac{u}{v}
\;\middle|\;
u,v\in\Z_{>0},\ p\nmid v
\right\}
\]
is dense in $\R_{>0}$, we can choose
$a=Mu/v$ in this set such that
\[
\qfpt(X;D)<a<\lct(X;D).
\]
Then
\[
rv(K_X+aD)=v(rK_X)+ru(MD)
\]
is Cartier.
Since $p\nmid rv$, the Cartier index of $K_X+aD$
is not divisible by $p$.

The pair $(X,aD)$ is log canonical.
It is not klt, since $X$ is not klt and $aD$ is effective.
Therefore, \cref{thm:qFs-lc-Z_p-index} implies that
$(X,aD)$ is purely quasi-$F^{\infty}$-split.
Moreover, $a<\lct(X;D)$ implies that
$\rdown{aD}=0$.
Consequently, $(X,aD)$ is quasi-$F^{\infty}$-split.
This again contradicts $a>\qfpt(X;D)$.
\end{proof}

\section{Appendix}

In this section, we use the same method as that of the proof of \cite{KTTWYY3}*{Theorem~3.44} to show that quasi-$F^\infty$-split rings are log canonical. In particular, we get that Gorenstein quasi-$F$-split rings (e.g.\ quasi-$F$-split complete intersections) are log canonical. This has been a key missing piece of the theory built in \cite{KTTWYY1}.
Jakub Witaszek taught us the following proof.

\begin{theorem} \label{thm:lc-quasi} 
Let $R$ be a Noetherian normal $F$-finite ring of characteristic $p>0$ and $\Delta$ be an effective $\Q$-Weil divisor on $\Spec R$ such that $D:=K_R+\Delta$ is numerically $\Q$-Cartier and $\rdown{\Delta}=0$. 
Assume that $(R,\Delta)$ is  quasi-$F^e$-split for some integer $e>0$. Let $\pi \colon X \to \Spec R$ be a projective birational map from a normal scheme $X$ and let $\epsilon>0$ be a rational number.
Suppose that there exists a Cartier divisor $A$ on $X$ such that
\begin{enumerate}[label=\textup{(\arabic*)}]
\item $R^i\pi_*\cO_X(p^{e+r}\pi_{num}^*D + p^rA) = 0$ for all integers $i>0$ and $r \geq 0$,
\item $0 \leq -A \leq p^e\epsilon E$ for the reduced effective exceptional divisor $E$ on $X$.
\end{enumerate}
Then $\pi_*\cO_X(\rup{K_X - \pi_{num}^*D + \epsilon E})= R$.
\end{theorem}

\begin{proof}
We may assume $(R,\m)$ is local and set $d:=\dim R$.
Pick $n$ such that $(R,\Delta)$ is purely $n$-quasi-$F^e$-split. Consider the following diagram.
\begin{equation}\label{eq:diag-ap}
\begin{tikzcd}
H^d_\m(W_nR(D)) \ar{r}{F^e} & H^d_\m(F^e_*W_nR(p^eD)) \\
H^d_\m(\pi_*W_n\cO_X(\pi_{num}^*D-\epsilon E)) \ar{d}{\phi} \ar{u}{=}  &  H^d_\m(F^e_*\pi_*W_n\cO_X(p^e\pi_{num}^*D)) \ar{u}{=}  \\
H^d_\m(R\pi_*W_n\cO_X(\pi_{num}^*D-\epsilon E)) \ar{r}{F^e} & H^d_\m(F^e_*R\pi_*W_n\cO_X(p^e(\pi_{num}^*D-\epsilon E))), \ar{u}{\psi} 
\end{tikzcd}
\end{equation}
where 
\begin{itemize}
    \item the upper left and right vertical arrows are identities, because the sheaf $\pi_*W_n\cO_X(\pi_{num}^*D-\epsilon E)$ agrees with $W_nR(D)$ outside of a locus of codimension at least $2$;
    \item the morphism $\phi$ is the inclusion of $\cH^0$ of the complex $R\pi_*W_n\cO_X(\pi_{num}^*D-\epsilon E)$;
    \item the morphism $\psi$ exists by the claim below.
\end{itemize}
\begin{claim*}
The morphism
\[
R\pi_*W_n\cO_X(p^e(\pi_{num}^*D-\epsilon E)) \to R\pi_*W_n\cO_X(p^e\pi_{num}^*D)
\]
induced by the inclusion $W_n\cO_X(p^e(\pi_{num}^*D-\epsilon E)) \subseteq W_n\cO_X(p^e\pi_{num}^*D)$ factors through
\[
\pi_*W_n\cO_X(p^e\pi_{num}^*D). 
\]
\end{claim*}
\begin{claimproof}
It is enough to argue that the map
\[
R^{>0}\pi_*W_n\cO_X(p^e(\pi_{num}^*D-\epsilon E)) \to R^{>0}\pi_*W_n\cO_X(p^e\pi_{num}^*D)
\]
is zero. Since 
\[
W_n\cO_X(p^e(\pi_{num}^*D-\epsilon E)) \subseteq W_n\cO_X(p^e\pi_{num}^*D + A) \subseteq W_n\cO_X(p^e\pi_{num}^*D),
\]
it is enough to argue that 
\[
R^{>0}\pi_*W_n\cO_X(p^e\pi_{num}^*D + A) = 0.
\]
This is immediate by (1) and induction on $n$ in view of the short exact sequence
\[
0 \to F_*W_{n-1}\cO_X(p^{e+1}\pi_{num}^*D + pA) \to W_n\cO_X(p^e\pi_{num}^*D + A) \to \cO_X(p^e\pi_{num}^*D + A) \to 0.
\]
\end{claimproof}\\
By the diagram \eqref{eq:diag-ap} we have a sequence of maps:    
\begin{align*}
H^d_\m(W_nR(D)) \xleftarrow{=} H^d_\m(\pi_*W_n\cO_X(\pi_{num}^*D-\epsilon E)) &\to H^d_\m(R\pi_*W_n\cO_X(\pi_{num}^*D-\epsilon E)) \\
&\to H^d_\m(F^e_*W_nR(p^eD)).  
\end{align*}
In particular, we get an inclusion
\begin{multline*}
{\rm Ker}\Big(H^d_\m(W_nR(D)) \xrightarrow{\phi'} H^d_\m(R\pi_*W_n\cO_X(\pi^*_{num}D-\varepsilon E)) \Big) \\
\subseteq {\rm Ker}\Big(H^d_\m(W_nR(D)) \xrightarrow{F^e} H^d_\m(F^e_*W_nR(p^eD)) \Big).  
\end{multline*}
Furthermore, we obtain
\begin{align*}
    &R^{n-1}\Big({\rm Ker}\Big(H^d_\m(W_nR(D)) \xrightarrow{\phi'} H^d_\m(R\pi_*W_n\cO_X(\pi^*_{num}D-\varepsilon E)) \Big)\Big) \\
    \subseteq& R^{n-1}\Big({\rm Ker}\Big(H^d_\m(W_nR(D)) \xrightarrow{F^e} H^d_\m(F^e_*W_nR(p^eD)) \Big)\Big) \\
    \overset{(\star_1)}{=}&0,
\end{align*}
where $(\star_1)$ follows from $n$-quasi-$F^e$-splitting of $(R,\Delta)$.

We consider the following commutative diagram in which each horizontal sequence is exact:
{\scriptsize
\[
\begin{tikzcd}[column sep=0.3cm]
    H^d_\m(F_*W_{n-1}R(pD)) \arrow[r] \arrow[d,twoheadrightarrow] & H^d_\m(W_nR(D)) \arrow[r] \arrow[d,twoheadrightarrow] & H^d_\m(R(K_R)) \arrow[r] \arrow[d,twoheadrightarrow,"\alpha"] & 0 \\
    H^d_\m(F_*W_{n-1}\cO_X(p(\pi^*_{num}D-\varepsilon E))) \arrow[r] & H^d_\m(W_n\cO_X(\pi^*_{num}D-\varepsilon E)) \arrow[r] &H^d_\m(\cO_X(\pi^*_{num}D-\varepsilon E)) \arrow[r] & 0.
\end{tikzcd}
\]}
By diagram chasing, we obtain 
\begin{multline*}
    {\rm Ker}\Big(H^d_\m(R(K_R)) \xrightarrow{\alpha} H^d_\m(\cO_X(\pi^*_{num}D-\varepsilon E)) \Big) \\
    =R^{n-1}\Big({\rm Ker}\Big(H^d_\m(W_nR(D)) \xrightarrow{\phi'} H^d_\m(R\pi_*W_n\cO_X(\pi^*_{num}D-\varepsilon E)) \Big)\Big)=0.
\end{multline*}
Taking Matlis dual of $\alpha$, we obtain 
\[
\pi_*\cO_X(\rup{K_X-\pi^*_{num}D+\varepsilon E})=R,
\]
as desired.
\end{proof}

In order to show that quasi-$F^\infty$-split rings are log canonical, we need one more technical lemma.
\begin{lem}\label{l deJong}
Let $X$ be a normal variety over a field $k$ of characteristic $p>0$. 
Then there exists a projective birational morphism $\rho : Y \to X$ 
from a normal $\Q$-factorial variety $Y$. 
\end{lem}

\begin{proof}
By (the proof of) de Jong's alterations (cf.\ \cite{CR12}*{Remark 4.3.2}), 
there exist proper surjective generically finite morphisms of normal varieties 
\[
Z'' \xrightarrow{f} Z' \xrightarrow{g} Z \xrightarrow{\varphi} X, 
\]
such that 
\begin{itemize}
\item $Z''$ is regular, 
\item $Z' = Z''/G$, that is, $f$ is the quotient by a finite group $G$ acting on $Z''$, 
\item $g$ is a projective birational morphism, and 
\item $\varphi$ is a finite purely inseparable surjective morphism (that is, a finite universal homeomorphism). 
\end{itemize}
As a finite image of a regular scheme, $Z'$ is $\Q$-factorial. 
Since $\varphi$ is a finite purely inseparable morphism, there exists $e$ and $\psi \colon X \to Z$ such that 
\[
(F^e \colon X \to X)=(X \xrightarrow{\psi} Z \xrightarrow{\varphi} X).
\]
Let $Y$ be the normalization of $Z'$ in $\psi^\sharp \colon K(Z) \to K(X)$. Then we obtain the diagram
\[
\begin{tikzcd}
    Y \arrow[r] \arrow[d,"\rho"] \arrow[rr,bend left,"F^e"] & Z' \arrow[d] \arrow[r] & Y \arrow[d,"\rho"] \\
    X \arrow[r,"\psi"] \arrow[rr,"F^e",bend right] & Z \arrow[r,"\varphi"] & X
\end{tikzcd}
\]
Since $Y$ is a finite image of a $\Q$-factorial variety $Z'$, $Y$ is $\Q$-factorial.  
\end{proof}

\begin{corollary} \label{cor:quasiFinfty-are-lc}
Let $R$ be a normal domain of finite type over an $F$-finite field of characteristic $p>0$ and $\Delta$ be an effective $\Q$-Weil divisor on $\Spec R$ such that $D:=K_R+\Delta$ is numerically $\Q$-Cartier and $\rdown{\Delta}=0$. 
If $(R,\Delta)$ is quasi-$F^{\infty}$-split, then $(R,\Delta)$ is numerically log canonical.
\end{corollary}
\begin{proof}
Let $\pi \colon X \to \Spec R$ be any projective birational morphism such that $X$ is normal. To show that $(R,\Delta)$ is numerically log canonical, we need to argue that 
\begin{equation} \label{eq:quasiFinfty-are-lc}
    \pi_*\cO_X(\rup{K_X-\pi_{num}^*D + \epsilon E})= R \text{ for every $\epsilon>0$.}
\end{equation}

Let $\rho \colon Y \to \Spec R$ be a projective birational map such that $Y$ is $\bQ$-factorial, the existence of which is guaranteed by \cref{l deJong}. By \cite{KW24}*{Lemma~8}, we may replace $\rho$ by a higher projective birational morphism so that
\begin{enumerate}
\item there exists a $\rho$-ample, $\rho$-exceptional Cartier divisor $A$,
\item $\rho \colon Y \to \Spec R$ factors through $\pi \colon X \to \Spec R$.
\end{enumerate}
By further replacing $\pi$ by $\rho$, we may thus assume that there exists a $\pi$-ample, $\pi$-exceptional Cartier divisor $A$ on $X$ and $X$ is $\Q$-factorial.
It is enough to show that \eqref{eq:quasiFinfty-are-lc} holds for this new choice of $X$. From now on, we fix $\epsilon >0$.

We take an integer $e_1 \geq 1$ such that $0 \leq -A \leq p^{e_1}\varepsilon E$.
Since $p^{e_1}\pi^*_{num}D+A$ is a $\pi$-ample $\Q$-Cartier divisor, there exists an integer $e_2 \geq 1$ such that 
\[
R^i\pi_*\cO_X(p^{e'}(p^{e_1}\pi^*_{num}D+A))=0
\]
for all integers $i >0$ and $e' \geq e_2$.
We put $e:=e_1+e_2$ and replace $A$ by $p^{e_2}A$. 
Then we have
\[
R^i\pi_*\cO_X(p^{r}(p^{e}\pi^*_{num}D+A))=0
\]
for all integers $i>0$ and $r \geq 0$.
Therefore the assumptions of \cref{thm:lc-quasi} are satisfied, and hence \eqref{eq:quasiFinfty-are-lc} holds. 
This concludes the proof of numerical log canonicity of $(R,\Delta)$.
\end{proof}

\bibliographystyle{skalpha}
\bibliography{bibliography.bib}

@article {LZ,
    AUTHOR = {Langer, Andreas and Zink, Thomas},
     TITLE = {De {R}ham-{W}itt cohomology for a proper and smooth morphism},
   JOURNAL = {J. Inst. Math. Jussieu},
  FJOURNAL = {Journal of the Institute of Mathematics of Jussieu. JIMJ.
              Journal de l'Institut de Math\'{e}matiques de Jussieu},
    VOLUME = {3},
      YEAR = {2004},
    NUMBER = {2},
     PAGES = {231--314},
      ISSN = {1474-7480},
   MRCLASS = {14F30 (14F40)},
  MRNUMBER = {2055710},
MRREVIEWER = {Martin C. Olsson},
       DOI = {10.1017/S1474748004000088},
       URL = {https://doi.org/10.1017/S1474748004000088},
}

@article{CST,
  author = {Carvajal-Rojas, Javier and Schwede, Karl and Tucker, Kevin},
  title = {Fundamental groups of~${F}$-regular singularities via ${F}$-signature},
  journal = {Annales scientifiques de l'\'Ecole Normale Sup\'erieure},
  pages = {993--1016},
  publisher = {Soci\'et\'e Math\'ematique de France. Tous droits r\'eserv\'es},
  volume = {Ser. 4, 51},
  number = {4},
  year = {2018},
  doi = {10.24033/asens.2370},
  mrnumber = {3861567},
  language = {en},
  url = {https://www.numdam.org/articles/10.24033/asens.2370/}
}

@article {Keeler03,
    AUTHOR = {Keeler, Dennis S.},
     TITLE = {Ample filters of invertible sheaves},
   JOURNAL = {J. Algebra},
  FJOURNAL = {Journal of Algebra},
    VOLUME = {259},
      YEAR = {2003},
    NUMBER = {1},
     PAGES = {243--283},
      ISSN = {0021-8693,1090-266X},
   MRCLASS = {14F05 (16S38)},
  MRNUMBER = {1953719},
MRREVIEWER = {Michel\ Van den Bergh},
       DOI = {10.1016/S0021-8693(02)00557-4},
       URL = {https://doi.org/10.1016/S0021-8693(02)00557-4},
}

@book {Serre79,
    AUTHOR = {Serre, Jean-Pierre},
     TITLE = {Local fields},
    SERIES = {Graduate Texts in Mathematics},
    VOLUME = {67},
      NOTE = {Translated from the French by Marvin Jay Greenberg},
 PUBLISHER = {Springer-Verlag, New York-Berlin},
      YEAR = {1979},
     PAGES = {viii+241},
      ISBN = {0-387-90424-7},
   MRCLASS = {12Bxx},
  MRNUMBER = {554237},
}

@incollection {Gabber,
    AUTHOR = {Gabber, Ofer},
     TITLE = {Notes on some {$t$}-structures},
 BOOKTITLE = {Geometric aspects of {D}work theory. {V}ol. {I}, {II}},
     PAGES = {711--734},
 PUBLISHER = {Walter de Gruyter, Berlin},
      YEAR = {2004},
      ISBN = {3-11-017478-2},
   MRCLASS = {14F05 (14F43)},
  MRNUMBER = {2099084},
MRREVIEWER = {Maurizio\ Cailotto},
}

@article {BLM,
    AUTHOR = {Bhatt, Bhargav and Lurie, Jacob and Mathew, Akhil},
     TITLE = {Revisiting the de {R}ham--{W}itt complex},
   JOURNAL = {Ast\'{e}risque},
  FJOURNAL = {Ast\'{e}risque},
    NUMBER = {424},
      YEAR = {2021},
     PAGES = {viii+165},
      ISSN = {0303-1179,2492-5926},
      ISBN = {978-2-85629-937-1},
   MRCLASS = {14F30 (11G25 14F40 14G17)},
  MRNUMBER = {4275461},
MRREVIEWER = {Yong\ Suk\ Moon},
       DOI = {10.24033/ast},
       URL = {https://doi.org/10.24033/ast},
}

@article {CR12,
    AUTHOR = {Chatzistamatiou, Andre and R\"{u}lling, Kay},
     TITLE = {Hodge-{W}itt cohomology and {W}itt-rational singularities},
   JOURNAL = {Doc. Math.},
  FJOURNAL = {Documenta Mathematica},
    VOLUME = {17},
      YEAR = {2012},
     PAGES = {663--781},
      ISSN = {1431-0635},
}

@book{kollar13,
  author =        {Koll{\'a}r, J{\'a}nos},
  pages =         {x+370},
  publisher =     {Cambridge University Press, Cambridge},
  series =        {Cambridge Tracts in Mathematics},
  title =         {Singularities of the minimal model program},
  volume =        {200},
  year =          {2013},
  doi =           {10.1017/CBO9781139547895},
  isbn =          {978-1-107-03534-8},
  url =           {http://dx.doi.org/10.1017/CBO9781139547895},
}

@article {yobuko19,
    AUTHOR = {Yobuko, Fuetaro},
     TITLE = {Quasi-{F}robenius splitting and lifting of {C}alabi-{Y}au
              varieties in characteristic {$p$}},
   JOURNAL = {Math. Z.},
  FJOURNAL = {Mathematische Zeitschrift},
    VOLUME = {292},
      YEAR = {2019},
    NUMBER = {1-2},
     PAGES = {307--316},
      ISSN = {0025-5874},
}

@article {tanaka22,
    AUTHOR = {Tanaka, Hiromu},
     TITLE = {Vanishing theorems of {K}odaira type for {W}itt canonical
              sheaves},
   JOURNAL = {Selecta Math. (N.S.)},
  FJOURNAL = {Selecta Mathematica. New Series},
    VOLUME = {28},
      YEAR = {2022},
    NUMBER = {1},
     PAGES = {Paper No. 12, 50},
      ISSN = {1022-1824},
}

@article{hw02,
  author =        {Hara, Nobuo and Watanabe, Kei-Ichi},
  journal =       {J. Algebraic Geom.},
  number =        {2},
  pages =         {363--392},
  title =         {{$F$}-regular and {$F$}-pure rings vs. log terminal and log
                   canonical singularities},
  volume =        {11},
  year =          {2002},
  doi =           {10.1090/S1056-3911-01-00306-X},
  issn =          {1056-3911},
  url =           {http://dx.doi.org/10.1090/S1056-3911-01-00306-X},
}

@article{hara98,
  author =        {Hara, N.},
  journal =       {Adv. Math.},
  number =        {1},
  pages =         {33--53},
  title =         {Classification of two-dimensional {$F$}-regular and
                   {$F$}-pure singularities},
  volume =        {133},
  year =          {1998},
  doi =           {10.1006/aima.1997.1682},
  issn =          {0001-8708},
  url =           {http://dx.doi.org/10.1006/aima.1997.1682},
}

@article{KW24,
  author =        {Koll{\'a}r, J{\'a}nos and Witaszek, Jakub},
  journal =       {Science China Mathematics},
  title =         {Resolution and alteration with ample exceptional divisor},
  year =          {2024},
}

@book {HS,
    AUTHOR = {Huneke, Craig and Swanson, Irena},
     TITLE = {Integral closure of ideals, rings, and modules},
    SERIES = {London Mathematical Society Lecture Note Series},
    VOLUME = {336},
 PUBLISHER = {Cambridge University Press, Cambridge},
      YEAR = {2006},
     PAGES = {xiv+431},
      ISBN = {978-0-521-68860-4; 0-521-68860-4},
   MRCLASS = {13B22 (13A18 13A30 13A35 13H15 14A05)},
  MRNUMBER = {2266432},
MRREVIEWER = {Liam\ O'Carroll},
}

@book {Huybrechts,
    AUTHOR = {Huybrechts, D.},
     TITLE = {Fourier-{M}ukai transforms in algebraic geometry},
    SERIES = {Oxford Mathematical Monographs},
 PUBLISHER = {The Clarendon Press, Oxford University Press, Oxford},
      YEAR = {2006},
     PAGES = {viii+307},
      ISBN = {978-0-19-929686-6; 0-19-929686-3},
   MRCLASS = {14F05 (14-02 18E30)},
  MRNUMBER = {2244106},
MRREVIEWER = {Bal\'azs\ Szendr\H oi},
       DOI = {10.1093/acprof:oso/9780199296866.001.0001},
       URL = {https://doi.org/10.1093/acprof:oso/9780199296866.001.0001},
}

@incollection {KollarTfS2,
    AUTHOR = {Koll{\'a}r, J{\'a}nos},
     TITLE = {Duality and normalization, variations on a theme of {S}erre
              and {R}eid},
 BOOKTITLE = {Recent developments in algebraic geometry---to {M}iles {R}eid for his 70th birthday},
    SERIES = {London Math. Soc. Lecture Note Ser.},
    VOLUME = {478},
     PAGES = {216--252},
      NOTE = {With an appendix by Hailong Dao},
 PUBLISHER = {Cambridge Univ. Press, Cambridge},
      YEAR = {2022},
      ISBN = {978-1-009-18085-6},
   MRCLASS = {14F06 (14F08)},
  MRNUMBER = {4480570},
MRREVIEWER = {Jean-Marc\ Dr\'ezet},
}

@article{Hashimoto15,
  author    = {Mitsuyasu Hashimoto},
  title     = {${F}$-finiteness of homomorphisms and its descent},
  journal   = {Osaka Journal of Mathematics},
  volume    = {52},
  number    = {1},
  pages     = {205--215},
  year      = {2015},
  doi       = {10.18910/57681},
}

@article{KTY22,
  author =        {Kawakami, T. and Takamatsu, T. and Yoshikawa, S.},
  journal =       {arXiv:2204.10076},
  title =         {Fedder criterion for quasi-${F}$-splittings {I}},
  year =          {2022},
}

@article{KTY25,
  author =        {Kawakami, T. and Takamatsu, T. and Yoshikawa, S.},
  journal =       {arXiv:2511.17270},
  title =         {Fedder criterion for quasi-${F}$-splittings {II}},
  year =          {2025},
}

@misc{stacks-project,
  author =        {{The Stacks Project Authors}},
  title =         {\itshape {S}tacks {P}roject},
  year =          {2014},
}

@book{hartshorne77,
  author =        {Hartshorne, R.},
  publisher =     {Springer-Verlag, New York},
  title =         {Algebraic Geometry},
  year =          {1977},
}

@book{esnault_viehweg,
  author =        {Esnault, H\'el\`ene and Viehweg, Eckart},
  pages =         {vi+164},
  publisher =     {Birkh\"auser Verlag, Basel},
  series =        {DMV Seminar},
  title =         {Lectures on vanishing theorems},
  volume =        {20},
  year =          {1992},
  doi =           {10.1007/978-3-0348-8600-0},
  isbn =          {3-7643-2822-3},
  url =           {http://dx.doi.org/10.1007/978-3-0348-8600-0},
}

@book {hartshorne_local_cohomology,
    AUTHOR = {Hartshorne, Robin},
     TITLE = {Local cohomology},
    SERIES = {Lecture Notes in Mathematics},
    VOLUME = {41},
      NOTE = {A seminar given by A. Grothendieck, Harvard University, Fall,
              1961},
 PUBLISHER = {Springer-Verlag, Berlin-New York},
      YEAR = {1967},
     PAGES = {vi+106},
   MRCLASS = {14.55 (18.00)},
  MRNUMBER = {224620},
MRREVIEWER = {F.\ Oort},
}

@article{illusie_de_rham_witt,
  author =        {Illusie, Luc},
  journal =       {Ann. Sci. \'Ecole Norm. Sup. (4)},
  number =        {4},
  pages =         {501--661},
  title =         {Complexe de de\thinspace {R}ham-{W}itt et cohomologie
                   cristalline},
  volume =        {12},
  year =          {1979},
  issn =          {0012-9593},
  url =           {http://www.numdam.org/item?id=ASENS_1979_4_12_4_501_0},
}

@article{TWY,
  title={Quasi-${F^e}$-splittings and quasi-${F}$-regularity},
  author={Tanaka, Hiromu and Witaszek, Jakub and Yobuko, Fuetaro},
  journal={arXiv preprint arXiv:2404.06788},
  year={2024}
}

@article{KTTWYY1,
  title={Quasi-${F}$-splittings in birational geometry},
  author={Kawakami, Tatsuro and Takamatsu, Teppei and Tanaka, Hiromu and Witaszek, Jakub and Yobuko, Fuetaro and Yoshikawa, Shou},
  journal={arXiv preprint arXiv:2208.08016},
  year={2022}
}

@article{KTTWYY3,
  title={Quasi-${F}$-splittings in birational geometry {III}},
  author={Kawakami, Tatsuro and Takamatsu, Teppei and Tanaka, Hiromu and Witaszek, Jakub and Yobuko, Fuetaro and Yoshikawa, Shou},
  journal={arXiv preprint arXiv:2408.01921},
  year={2024}
}

@article{ST23,
  author       = {Sato, Kenta and Takagi, Shunsuke},
  title        = {Deformations of log terminal and semi log canonical singularities},
  journal      = {Forum of Mathematics, Sigma},
  year         = {2023},
  volume       = {11},
  doi          = {10.1017/fms.2023.28},
  note         = {Open Access},
}

@article{Yoshikawa22,
  author       = {Yoshikawa, Shou},
  title        = {Global {$F$}-splitting of surfaces admitting an int-amplified endomorphism},
  journal      = {Manuscripta Mathematica},
  year         = {2022},
  volume       = {169},
  pages        = {271--296},
  doi          = {10.1007/s00229-021-01331-5},
}

@article{RRS96,
  author = {Reid, Les and Roberts, Leslie G. and Singh, Balwant},
  title = {On weak subintegrality},
  journal = {Journal of Pure and Applied Algebra},
  volume = {114},
  number = {1},
  pages = {93--109},
  year = {1996},
  doi = {10.1016/0022-4049(95)00157-3},
}

@article{schwedetucker14,
  title={On the behavior of test ideals under finite morphisms},
  author={Schwede, Karl and Tucker, Kevin},
  journal={Journal of Algebraic Geometry},
  volume={23},
  number={3},
  pages={399--443},
  year={2014}
}

@article{Tra70,
author = {Traverso, Carlo},
fjournal = {Annali della Scuola Normale Superiore di Pisa - Classe di Scienze},
journal ={Ann. Sc. Norm. Sup. Pisa},
number = {4},
pages = {585-595},
publisher = {Scuola normale superiore},
title = {Seminormality and {Picard} group},
url = {http://eudml.org/doc/83540},
volume = {24},
year = {1970},
}

@article{Tanaka18,
  author    = {Tanaka, Hiromu},
  title     = {Minimal model program for excellent surfaces},
  journal   = {Annales de l'Institut Fourier},
  volume    = {68},
  number    = {1},
  pages     = {345--376},
  year      = {2018},
}

@article{Sato25,
  author={Sato, Kenta},
  title={General hyperplane sections of log canonical threefolds in positive characteristic},
  journal={Journal of the Institute of Mathematics of Jussieu},
  pages={1--28},
  year={2025},
  publisher={Cambridge University Press}
}

@article{MS12,
  title={Semi-log canonical vs {$F$}-pure singularities},
  author={Miller, Lance Edward and Schwede, Karl},
  journal={Journal of Algebra},
  volume={349},
  number={1},
  pages={150--164},
  year={2012},
  publisher={Elsevier}
}

@article{Hoc77,
  title={Cyclic purity versus purity in excellent Noetherian rings},
  author={Hochster, Melvin},
  journal={Transactions of the American Mathematical Society},
  volume={231},
  number={2},
  pages={463--488},
  year={1977}
}

@article {HR76,
    AUTHOR = {Hochster, Melvin and Roberts, Joel L.},
     TITLE = {The purity of the {F}robenius and local cohomology},
   JOURNAL = {Advances in Math.},
  FJOURNAL = {Advances in Mathematics},
    VOLUME = {21},
      YEAR = {1976},
    NUMBER = {2},
     PAGES = {117--172},
      ISSN = {0001-8708},
   MRCLASS = {14B15 (13C15 14M99)},
  MRNUMBER = {417172},
MRREVIEWER = {Gerhard\ Pfister},
       DOI = {10.1016/0001-8708(76)90073-6},
       URL = {https://doi.org/10.1016/0001-8708(76)90073-6},
}

@article {MR85,
    AUTHOR = {Mehta, V. B. and Ramanathan, A.},
     TITLE = {Frobenius splitting and cohomology vanishing for {S}chubert
              varieties},
   JOURNAL = {Ann. of Math. (2)},
  FJOURNAL = {Annals of Mathematics. Second Series},
    VOLUME = {122},
      YEAR = {1985},
    NUMBER = {1},
     PAGES = {27--40},
      ISSN = {0003-486X,1939-8980},
   MRCLASS = {14M15 (20G10)},
  MRNUMBER = {799251},
MRREVIEWER = {H.\ H.\ Andersen},
       DOI = {10.2307/1971368},
       URL = {https://doi.org/10.2307/1971368},
}

@article {HH89,
    AUTHOR = {Hochster, Melvin and Huneke, Craig},
     TITLE = {Tight closure and strong {$F$}-regularity},
      NOTE = {Colloque en l'honneur de Pierre Samuel (Orsay, 1987)},
   JOURNAL = {M\'em. Soc. Math. France (N.S.)},
  FJOURNAL = {M\'emoires de la Soci\'et\'e{} Math\'ematique de France.
              Nouvelle S\'erie},
    NUMBER = {38},
      YEAR = {1989},
     PAGES = {119--133},
      ISSN = {0037-9484},
   MRCLASS = {13H10 (13A50 13C14)},
  MRNUMBER = {1044348},
MRREVIEWER = {W.\ V.\ Vasconcelos},
}

@article {ST25,
    AUTHOR = {Sato, Kenta and Takagi, Shunsuke},
     TITLE = {Arithmetic and geometric deformations of {$F$}-pure and
              {$F$}-regular singularities},
   JOURNAL = {Amer. J. Math.},
  FJOURNAL = {American Journal of Mathematics},
    VOLUME = {147},
      YEAR = {2025},
    NUMBER = {2},
     PAGES = {561--596},
      ISSN = {0002-9327,1080-6377},
   MRCLASS = {14G17 (13A35 14B05)},
  MRNUMBER = {4887970},
       DOI = {10.1353/ajm.2025.a954650},
       URL = {https://doi.org/10.1353/ajm.2025.a954650},
}

\end{document}